\documentclass[reqno]{amsart}

\usepackage{longtable}
\usepackage{graphicx}
\usepackage{tikz}
\usepackage[outline]{contour}
\usepackage{afterpage}

\usepackage{pdftricks}
\begin{psinputs}
  \usepackage{pstricks}
  \usepackage{pst-node}
  \usepackage{pst-eucl}
  \usepackage{multido}
\end{psinputs}

\usepackage{array}

\usepackage{amsrefs,amssymb,amsthm}
\usepackage{needspace}

\usepackage{subfigure}

\usepackage{fix-cm}
\usepackage{hyperref}
\hypersetup{colorlinks,linkcolor=black,filecolor=black,urlcolor=black,citecolor=black}

\newcommand{\cent}{{\tt c}}
\newcommand{\foot}{\mathrm{f}}
\newcommand{\slope}{{\tt s}}

\newcommand{\nofig}[1]{#1}

\newtheorem{theorem}{Theorem}[section]
\newtheorem{proposition}[theorem]{Proposition}
\newtheorem{lemma}[theorem]{Lemma}

\newtheorem{claim}[theorem]{Claim}

\theoremstyle{definition}

\newtheorem{remark}[theorem]{Remark}

\newtheorem{definition}[theorem]{Definition}
\newtheorem{conjecture}[theorem]{Conjecture}

\newcommand{\st}{\;|\;}

\newcommand{\sbs}{\subseteq}
\newcommand{\stm}{\setminus}

\newcommand{\tref}[1]{Theorem~\ref{t.#1}}
\newcommand{\dref}[1]{Definition~\ref{d.#1}}
\newcommand{\pref}[1]{Proposition~\ref{p.#1}}
\newcommand{\hyref}[1]{hypothesis~(\ref{P.#1})}
\newcommand{\lref}[1]{Lemma~\ref{l.#1}}
\newcommand{\fref}[1]{Figure~\ref{f.#1}}
\newcommand{\Sref}[1]{Step~\ref{S.#1}}
\newcommand{\sref}[1]{Section~\ref{s.#1}}

\newcommand{\eref}[1]{(\ref{e.#1})}

\newcommand{\ipart}[2]{\needspace{3ex}\smallskip \noindent \textbf{#1} \emph{#2}}

\numberwithin{equation}{section}
\numberwithin{figure}{section}
\numberwithin{theorem}{section}

\newcommand{\abs}[1]{|#1|}

\newcommand{\1}{\mathbf{1}}
\newcommand{\Z}{\mathbb{Z}}
\newcommand{\N}{\mathbb{N}}
\newcommand{\R}{\mathbb{R}}

\newcommand{\C}{\mathbb{C}}
\newcommand{\trace}{\operatorname{trace}}

\newcommand{\B}{\mathcal{B}}

\newcommand{\eee}{\mathcal{E}}

\newcommand{\fff}{\mathcal{F}}
\newcommand{\hhh}{\mathcal{H}}

\newcommand{\kkk}{\mathcal{K}}

\newcommand{\rrr}{\mathcal{R}}
\newcommand{\sss}{\mathcal{S}}
\newcommand{\ttt}{\mathcal{T}}

\newcommand{\I}{\mathbf{i}}
\newcommand{\id}{\mathbf{1}}
\newcommand{\Ga}{\Z[\I]}
\newcommand{\mat}[1]{\left[ \begin{matrix} #1 \end{matrix} \right]}

\renewcommand{\Im}{\mathfrak{Im}}

\renewcommand{\Im}{\mathrm{Im}}

\newcommand{\change}[1]{#1}

\begin{document}

\title[Integer superharmonic matrices]{The Apollonian structure of integer superharmonic matrices}

\author{Lionel Levine}

\author{Wesley Pegden}

\author{Charles K. Smart}

\date{May 22, 2014}

\begin{abstract}
We prove that the set of quadratic growths attainable by integer-valued superharmonic functions on the lattice $\mathbb{Z}^2$ has the structure of an Apollonian circle packing.  This completely characterizes the PDE which determines the continuum scaling limit of the Abelian sandpile on the lattice $\Z^2$.
\end{abstract}

\maketitle

\section{Introduction}
\subsection{Main results}
This paper concerns the growth of integer-valued superharmonic functions on the lattice $\Z^2$; that is, functions $g : \Z^2 \to \Z$ with the property that the value at each point $x$ is at least the average of the values at the four lattice neighbors $y\sim x$.  In terms of the Laplacian operator, these are the functions $g$ satisfying:
\begin{equation}
\label{e.superharmonic}
\Delta g(x) := \sum_{y\sim x} (g(y) - g(x)) \leq 0
\end{equation}
for all $x \in \Z^2$.  Our goal is to understand when a given quadratic growth at infinity, specified by a $2 \times 2$ real symmetric matrix $A$, can be attained by an integer-valued superharmonic function.  For technical reasons, it is convenient to replace $0$ by $1$ in the inequality above. (It is straightforward to translate between the two versions using the function $f(x)=\frac12 x_1 (x_1+1)$, which has $\Delta f \equiv 1$.) So we seek to determine, for each matrix $A$, whether there exists a function $g : \Z^2 \to \Z$ such that
\begin{equation}
\label{e.super}
g(x) = \tfrac{1}{2} x^t A x + o(|x|^2) \quad \mbox{and} \quad \Delta g(x) \leq 1 \quad \mbox{for all }x \in \Z^2.
\end{equation}
When this holds, we say that the matrix $A$ is {\em integer superharmonic}, and we call $g$ an \emph{integer superharmonic representative} for $A$.

\begin{figure}[t]
\begin{center}
\nofig{\input{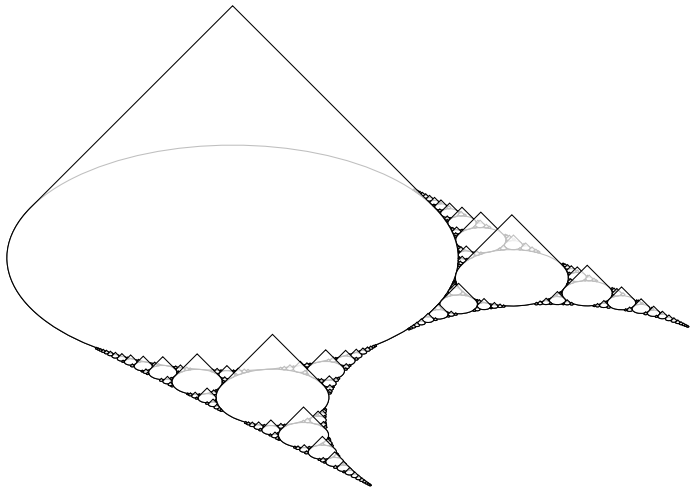}}
\end{center}
\caption{One $2\Z^2$-period of the boundary of the set of integer superharmonic matrices, as characterized by \tref{superharmonic}.}
\label{f.cone}
\end{figure}

We will relate the set of integer superharmonic matrices to an Apollonian circle packing.  Recall that every triple of pairwise tangent general circles (circles or lines) in the plane has exactly two \emph{Soddy general circles} tangent to all three.
  An \emph{Apollonian circle packing} is a minimal collection of general circles containing a given triple of pairwise-tangent general circles that is closed under the addition of Soddy general circles.  Let $\B_k$ $(k\in \Z)$ denote the Apollonian circle packing generated by the vertical lines through the points $(2k,0)$ and $(2k+2,0)$ in $\R^2$ together with the circle of radius $1$ centered at $(2k+1,0)$.  The \emph{band packing} is the union $\B=\bigcup_{k\in \Z} \B_k$, which is a circle packing of the whole plane, plus the vertical lines.

To each circle $C \in \B$ with center $(x_1, x_2) \in \R^2$ and radius $r > 0$, we associate the matrix

\begin{equation}
A_C = \frac{1}{2} \mat{ r + x_1 & x_2 \\ x_2 & r - x_1 }.
\label{l.AC}
\end{equation}

We use the semi-definite order on the space $S_2$ of $2 \times 2$ real symmetric matrices, which sets $A \leq B$ if and only if $B - A$ is positive semidefinite.  Our main result relates integer superharmonic matrices to the band packing.

\begin{theorem}
\label{t.superharmonic}
$A \in S_2$ is integer superharmonic if and only if either $A \leq A_C$ for some circle $C \in \B$, or $\trace(A)\leq 0$.
\end{theorem}
This theorem implies that the boundary of the set of integer superharmonic matrices looks like the surface displayed in \fref{cone}: it is a union of slope-1 cones whose bases are the circles $C \in \B$ and whose peaks are the matrices $A_C$.  Here we have identified $S_2$ with $\R^3$ with coordinates $(x_1,x_2,r)$, and $\mathcal{B}$ lies in the $r=0$ plane.  \change{Note that this embedding of $S_2$ with $\R^3$ has the property that the downset of a matrix $A$ is the slope-1 downward cone in $\R^3$ whose vertex corresponds to $A$.  In particular, the intersection of the cone 
	\[ \{A \in S_2 \mid A \leq A_C \} \]
with the $r=0$ plane is the closed disk bounded by $C$.  Note that the condition $\trace(A)\leq 0$ in the Theorem pulls in a Cantor set of matrices in the $\trace(A)=0$ plane which are in the topological closure of the downset of the set of matrices $A_C$ ($C\in \mathcal{B}$), but not in the set itself.}

\begin{figure}[t]
\begin{center}
\nofig{\includegraphics[width=1\textwidth]{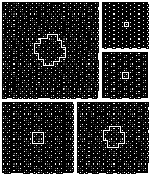}}
\end{center}
\caption{The periodic pattern $\Delta g_C$ of the integer superharmonic representative \tref{odometer}, shown  for five different circles $C$. 
Black, patterned, and white cells correspond to sites $x \in \Z^2$ where $\Delta g_C(x)$ equals $1$, $0$, and $-2$, respectively.  (In general, $-1$ can also occur as a value.)  In each case the fundamental tile $T_C$ is identified by a white outline.
Clockwise from the top left is $\Delta g_C$ for the circle $(153,17,120)$, its three parents $(4,1,4)$, $(9,1,6)$, and $(76,7,60)$, and its Soddy precursor $(25,1,20)$.  The circles themselves are drawn in \fref{circ}.
} 
\label{f.odom}
\end{figure}

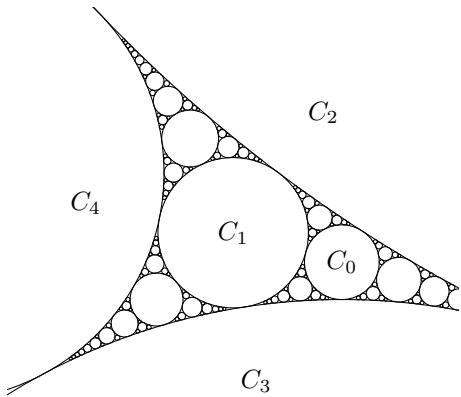
\begin{figure}[t]
\begin{center}
\nofig{\begin{tikzpicture}
\clip (-3,-3) rectangle (3,3);
\draw (12.00000,16.00000) circle (19.00000);
\draw (1.44444,-9.33333) circle (8.44444);
\draw (-3.96000,0.80000) circle (3.04000);
\draw (0.00000,0.00000) circle (1.00000);
\draw (1.44444,-0.39216) circle (0.49673);
\draw (-0.57214,1.25373) circle (0.37811);
\draw (-1.01852,-0.88889) circle (0.35185);
\draw (2.20313,-0.62500) circle (0.29688);
\draw (-0.86458,1.75000) circle (0.19792);
\draw (2.67273,-0.77922) circle (0.19740);
\draw (-1.44340,-1.20755) circle (0.17925);
\draw (1.14880,0.20131) circle (0.16630);
\draw (0.88983,-0.74576) circle (0.16102);
\draw (-0.06518,1.13966) circle (0.14153);
\draw (2.99259,-0.88889) circle (0.14074);
\draw (-0.54593,-0.99133) circle (0.13172);
\draw (-0.79333,0.80000) circle (0.12667);
\draw (-1.04160,-0.41600) circle (0.12160);
\draw (-1.04160,2.01600) circle (0.12160);
\draw (-1.68000,-1.37143) circle (0.10857);
\draw (2.02326,-0.27907) circle (0.09302);
\draw (1.86058,-0.80769) circle (0.09135);
\draw (1.00434,0.40391) circle (0.08252);
\draw (-1.16017,2.18182) circle (0.08225);
\draw (0.67118,-0.84495) circle (0.07908);
\draw (0.13650,1.06486) circle (0.07357);
\draw (-1.83142,-1.47126) circle (0.07280);
\draw (-0.34657,-1.01083) circle (0.06859);
\draw (1.37847,0.16667) circle (0.06597);
\draw (-0.60624,1.69644) circle (0.06592);
\draw (1.09732,-0.83221) circle (0.06376);
\draw (-0.86178,0.62281) circle (0.06328);
\draw (-0.93750,1.50000) circle (0.06250);
\draw (-1.03560,-0.23301) circle (0.06149);
\draw (-1.20896,-1.25373) circle (0.05970);
\draw (2.55191,-0.55269) circle (0.05933);
\draw (-1.24512,2.29508) circle (0.05933);
\draw (2.44136,-0.88889) circle (0.05864);
\draw (-1.41938,-0.97199) circle (0.05753);
\draw (1.05556,0.00000) circle (0.05556);
\draw (0.91098,-0.53123) circle (0.05456);
\draw (-0.14489,1.31818) circle (0.05398);
\draw (-1.93681,-1.53846) circle (0.05220);
\draw (-0.68325,-1.10995) circle (0.04974);
\draw (0.91974,0.50485) circle (0.04919);
\draw (-0.87563,0.95431) circle (0.04822);
\draw (0.55309,-0.88889) circle (0.04691);
\draw (-0.21429,1.02463) circle (0.04680);
\draw (-1.17150,-0.52174) circle (0.04589);
\draw (0.24452,1.01599) circle (0.04500);
\draw (1.93853,-0.17021) circle (0.04492);
\draw (-1.30896,2.37736) circle (0.04481);
\draw (-0.62637,-0.83516) circle (0.04396);
\draw (-0.62637,0.83516) circle (0.04396);
\draw (1.73210,-0.84988) circle (0.04388);
\draw (-0.23604,-1.01493) circle (0.04201);
\draw (-0.90824,-0.51078) circle (0.04201);
\draw (2.90590,-0.72904) circle (0.04110);
\draw (2.82618,-0.96137) circle (0.04077);
\draw (-2.01446,-1.58678) circle (0.03926);
\draw (2.15342,-0.29385) circle (0.03798);
\draw (-0.88506,1.98501) circle (0.03798);
\draw (-0.89421,0.52695) circle (0.03792);
\draw (1.97403,-0.86820) circle (0.03724);
\draw (-1.02831,-0.13470) circle (0.03709);
\draw (-1.07171,1.86047) circle (0.03682);
\draw (1.47571,0.13882) circle (0.03517);
\draw (-1.35869,2.43983) circle (0.03504);
\draw (1.19141,-0.85868) circle (0.03399);
\draw (-1.53994,-1.39759) circle (0.03391);
\draw (-1.65443,-1.23198) circle (0.03320);
\draw (0.86432,0.56505) circle (0.03263);
\draw (-0.51458,1.66038) circle (0.03259);
\draw (1.97222,-0.39216) circle (0.03105);
\draw (0.47897,-0.91303) circle (0.03103);
\draw (1.88152,-0.68742) circle (0.03073);
\draw (-2.07407,-1.62319) circle (0.03060);
\draw (-0.95040,1.40800) circle (0.03040);
\draw (0.31178,0.98204) circle (0.03034);
\draw (-1.11975,-1.25661) circle (0.02955);
\draw (-0.16567,-1.01493) circle (0.02836);
\draw (1.11458,0.39286) circle (0.02827);
\draw (-0.18855,1.38788) circle (0.02826);
\draw (2.49648,-0.48498) circle (0.02818);
\draw (-1.39852,-0.88889) circle (0.02815);
\draw (-1.39852,2.48889) circle (0.02815);
\draw (2.35769,-0.91049) circle (0.02777);
\draw (1.02484,-0.07746) circle (0.02777);
\draw (0.92352,-0.45022) circle (0.02742);
\draw (0.76781,-0.88889) circle (0.02707);
\draw (1.88997,-0.11801) circle (0.02638);
\draw (-0.64249,1.78133) circle (0.02638);
\draw (0.98209,0.29752) circle (0.02617);
\draw (-0.74667,-1.15125) circle (0.02595);
\draw (1.66441,-0.86604) circle (0.02571);
\draw (-0.90280,1.02264) circle (0.02530);
\draw (0.70356,-0.74576) circle (0.02526);
\draw (-0.91293,0.46660) circle (0.02526);
\draw (-0.99278,1.56773) circle (0.02493);
\draw (-1.02219,-0.07311) circle (0.02480);
\draw (0.10026,1.15625) circle (0.02474);
\draw (1.33453,0.24601) circle (0.02473);
\draw (-1.05518,2.16164) circle (0.02467);
\draw (2.63554,-0.56031) circle (0.02464);
\draw (-2.12129,-1.65161) circle (0.02452);
\draw (2.51216,-0.93214) circle (0.02433);
\draw (-1.17430,2.07634) circle (0.02417);
\draw (-1.22026,-0.57179) circle (0.02398);
\draw (1.02130,-0.87567) circle (0.02380);
\draw (-1.26909,-1.31164) circle (0.02378);
\draw (-0.27289,0.98622) circle (0.02328);
\draw (0.82524,0.60495) circle (0.02322);
\draw (0.04988,1.02190) circle (0.02311);
\draw (-1.43113,2.52843) circle (0.02311);
\draw (-0.41425,-1.07246) circle (0.02295);
\draw (-1.48977,-1.01083) circle (0.02286);
\draw (-0.56272,0.85346) circle (0.02227);
\draw (0.42807,-0.92807) circle (0.02204);
\draw (-0.66118,-0.77922) circle (0.02193);
\draw (-0.67396,1.64055) circle (0.02189);
\draw (-1.73829,-1.48808) circle (0.02183);
\draw (0.35766,0.95720) circle (0.02183);
\draw (1.52928,0.11940) circle (0.02181);
\draw (-0.40165,-0.93937) circle (0.02163);
\draw (1.31250,0.10909) circle (0.02159);
\draw (-1.80986,-1.37943) circle (0.02154);
\draw (-0.90045,0.69799) circle (0.02125);
\draw (-0.85966,-0.55130) circle (0.02125);
\draw (-0.85966,1.53089) circle (0.02125);
\draw (1.24528,-0.87014) circle (0.02109);
\draw (1.07172,-0.75145) circle (0.02097);
\draw (1.01864,0.06687) circle (0.02083);
\draw (-1.09848,-0.28571) circle (0.02056);
\draw (-0.07222,1.30160) circle (0.02056);
\draw (2.21026,-0.30769) circle (0.02051);
\draw (-0.11690,-1.01371) circle (0.02042);
\draw (0.84661,-0.56959) circle (0.02038);
\draw (-0.78386,0.65326) circle (0.02038);
\draw (2.02751,-0.88889) circle (0.02011);
\draw (-2.15962,-1.67442) circle (0.02008);
\draw (-1.24574,-1.18313) circle (0.01991);
\draw (2.51979,-0.62500) circle (0.01979);
\draw (-0.97766,-0.28995) circle (0.01975);
\draw (2.45821,-0.81242) circle (0.01966);
\draw (1.12887,0.01649) circle (0.01959);
\draw (-1.36387,-1.02540) circle (0.01950);
\draw (-0.46799,1.63734) circle (0.01938);
\draw (2.86687,-0.68293) circle (0.01931);
\draw (-1.45833,2.56098) circle (0.01931);
\draw (0.96465,-0.58182) circle (0.01919);
\draw (2.76761,-0.97384) circle (0.01911);
\draw (1.20000,1.60000) node {$C_2$};
\draw (0.30588,-1.97647) node {$C_3$};
\draw (-1.96040,0.39604) node {$C_4$};
\draw (0.00000,0.00000) node {$C_1$};
\draw (1.44444,-0.39216) node {$C_0$};
\end{tikzpicture}}
\vspace{-1cm}
\end{center}
\caption{The child circle $C_0 = (153,17,120)$ and its parents $C_1 = (76,7,60)$, $C_2 = (4,1,4)$, and $C_3 = (9,1,6)$ in $\B$.  The other Soddy circle $C_4 = (25,1,20)\in \B$ for the parents of $C_0$ is the \emph{Soddy precursor} of $C_0$.  The triples of integers are curvature coordinates, defined in \sref{descartes}.}
\label{f.circ}
\end{figure}

To prove \tref{superharmonic} we will recursively construct an integer superharmonic representative $g_C$ for each matrix $A_C$ $(C \in \B)$.  Let
	\begin{equation} \label{e.LC} L_C = \{ v \in \Z^2 \mid A_C v \in \Z^2 \}. \end{equation}  
By the extended Descartes circle theorem of Lagarias, Mallows and Wilks \cite{Lagarias-Mallows-Wilks}, each $A_C$ has rational entries, so $L_C$ is a full-rank sublattice of $\Z^2$.

The following theorem, from which \tref{superharmonic} will follow, encapsulates the essential properties of the $g_C$'s we construct.

\begin{theorem}
\label{t.odometer}
For each circle $C\in \B$,
there exists an integer superharmonic representative $g_C$ for $A_C$ which satisfies the periodicity condition
\begin{equation}
\label{e.periodic}
g_C(x + v) = g_C(x) + x^t A_C v + g_C(v)
\end{equation}
for all $v\in L_C$ and $x \in \Z^2$. Moreover, $g_C$ is maximal in the sense that $g - g_C$ is bounded whenever $g : \Z^2 \to \Z$ satisfies $\Delta g \leq 1$ and $g \geq g_C$.
\end{theorem}
\noindent \change{We can always choose some $b\in \R^2$ so that $g(v_i)+b\cdot v_i=g(-v_i)-b\cdot v_i$ for two basis vectors $v_1,v_2$ of $L_C$.  In particular, choosing $g(v-v)=g(0)=0$, this implies with \eref{periodic} that $g_C(x)- \frac 1 2 x^t A_C x - b^t x$ is $L_C$-periodic for some $b \in \R^2$. In particular, $g$ witnesses that $A_C \in \Gamma$.  We call an integer superharmonic representative with this property an \emph{odometer} for $A_C$.} The construction of $g_C$ is explicit but rather elaborate.  In \sref{overview} we outline the steps of this construction and give the derivation of \tref{superharmonic} from \tref{odometer}.  We now briefly survey a few connections to our work.
  
\subsection{Hexagonal tilings of the plane by $90^\circ$ symmetric tiles}
\label{s.hextilings}
\fref{odom} shows the Laplacians $\Delta g_C$ for a triple of circles in $\mathcal{B}$ and their two Soddy circles.  
In each case $\Delta g_C$ is periodic and we have outlined a fundamental domain $T_C$ on whose boundary $\Delta g_C=1$. A major component of our paper is the construction of these $T_C$, which turn out to have a remarkable tiling property.

To state it precisely, for $x \in \Z^2$ write $s_x=\{x_1,x_1+1\}\times \{x_2,x_2+1\}\sbs \Z^2$ 
and $\bar s_x=[x_1,x_1+1]\times [x_2,x_2+1]\sbs \R^2$. 
If $T$ is a set of squares $s_x$, we call $T$ a \emph{tile} if the set
\begin{align}
\label{e.disc}
I(T)&:=\bigcup_{s_x\in T}\bar s_x\sbs \R^2
\end{align}
is a topological disk.  A {\em tiling} of $\Z^2$ is a collection of tiles $\ttt$ such that every square $s_x$ ($x\in \Z^2$) belongs to exactly one tile from $\ttt$.  

\begin{theorem}
\label{t.tiling}
For every circle $C\in \B$, 
there is a tile $T_C \sbs \Z^2$ with $90^\circ$ rotational symmetry, such that $T_C+L_C$ is a tiling of $\Z^2$.  Moreover, except when $C$ has radius 1, each tile in $T_C+L_C$ borders exactly 6 other tiles.
\end{theorem}

The tiles $T_C$ have the peculiar feature of being more symmetric than their plane tilings (which are only $180^\circ$ symmetric).  We expect this to be a strong restriction.  In particular, call a tiling {\em regular} if it has the form $T+L$ for some tile $T$ and lattice $L \sbs \Z^2$, and \emph{hexagonal} if each tile borders exactly 6 other tiles.  For regular tilings $\ttt, \ttt'$ of $\Z^2$, write $\ttt'\prec \ttt$ if each tile in $T\in \ttt$ is a union of tiles from $\ttt'$, and call the regular tiling $\ttt$ \emph{primitive} if $\ttt'\prec \ttt$ implies that either $\ttt'=\ttt$, or $\ttt'$ is the tiling of $\Z^2$ by squares $s_x$.

\begin{conjecture}
If $\ttt$ is a primitive, regular, hexagonal tiling of $\Z^2$ by $90^\circ$ symmetric tiles, then $\ttt = T_C + L_C + v$ for some $C \in \B$ and some $v \in \Z^2$.
\end{conjecture}

\subsection{Apollonian circle packings}
\label{s.descartes}
\begin{figure}[t]
\begin{center}
\nofig{\includegraphics[width=0.22\textwidth]{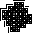} \hspace{0.5in} \includegraphics[width=0.22\textwidth]{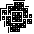}}
\end{center}
\caption{The tile of the circle $(153,17,120)$, decomposed into the tiles of its parents and Soddy precursor.}
\label{f.dec}
\end{figure}

Our proof of \tref{odometer} is a recursive construction which mimics the recursive structure of Apollonian circle packings.  Identifying a circle with center $(x_1,x_2)\in \R^2$ and radius $\frac 1 c\in \R$  with its {\em curvature coordinates} $(c,c x_1, c x_2)$ (some care must be taken in the case of lines), the Soddy circles $C_0$ and $C_4$ of a pairwise tangent triple of circles $C_1,C_2,C_3$ satisfy the linear equality 
\begin{equation}
\label{e.soddy}
C_0 + C_4 = 2(C_1 + C_2 + C_3).
\end{equation}
This is a consequence of the Extended Descartes Theorem of Lagarias, Mallows, and Wilks \cite{Lagarias-Mallows-Wilks}, and can be used, for example, to prove that every circle in $\B$ has integer curvature coordinates.   Pairwise tangent circles $C_1,C_2,C_3,C_4$ constitute a \emph{Descartes quadruple}. Under permutation of indices, \eref{soddy} gives four different ways of producing a new Descartes quadruple sharing three circles in common with the original.  These four transformations correspond to the four generators of the \emph{Apollonian group} of Graham, Lagarias, Mallows, Wilks, and Yan \cite{Graham-Lagarias-Mallows-Wilks-Yan}, which acts on the Descartes quadruples of a circle packing.  Our proof works by explicitly determining the action of the same Apollonian group on the set of maximal superharmonic representatives, by giving an operation on our family of integer superharmonic representatives analogous to the operation \eref{soddy} for circles.  In particular, referring to \fref{dec}, the example tile $T_0$ is seen to decompose into 2 copies each of $T_1,T_2,T_3$, with the copies of $T_1$ overlapping on a copy of $T_4$; note, for example, that the tile areas must therefore satisfy
$\abs{T_0}+\abs{T_4}=2(\abs{T_1}+\abs{T_2}+\abs{T_3}).$

It follows from \eqref{e.soddy} that if the three generating general circles of an Apollonian circle packing have integer curvatures, then every general circle in the packing has integer curvature.  In such a packing, the question of which integers arise as curvatures has attracted intense interest over the last decade \cite{Graham-Lagarias-Mallows-Wilks-Yan, Sarnak, kontorovich2011apollonian,
bourgain2011proof, bourgain2013local}; see \cite{fuchs2013counting} for a survey.

\tref{superharmonic} can be regarded as a new characterization of the circles appearing in the band packing: the circles in $\mathcal{B}$ correspond to integer superharmonic matrices that are maximal in the semidefinite order. 
In constructing $g_C$ and $T_C$ by an analogue of the Descartes rule \eqref{e.soddy}, we follow ideas of Stange \cite{Stange}, who associates circles to ``lax lattices'' and proves a Descartes rule relating the bases of lattices corresponding to four mutually tangent general circles.   Our proof also associates to each circle in $\mathcal{B}$ more detailed arithmetic information: Theorems \ref{t.odometer} and \ref{t.tiling} associate to each circle $C \in \B$ an integer superharmonic representative $g_C : \Z^2 \to \Z$ and a fundamental tile $T_C$. The curvature of $C$ can be recovered as the area of $T_C$.

\subsection{Abelian sandpile}
We briefly describe the Abelian sandpile model of Bak, Tang, and Wiesenfeld \cite{Bak-Tang-Wiesenfeld} that motivated our work.
Put $n$ chips at the origin of $\Z^2$.  In the sandpile model, a vertex having at least 4 chips \emph{topples} by sending one chip to each lattice neighbor. Dhar \cite{dhar1990self} observed that the final configuration of chips does not depend on the order of topplings. This final configuration $s_n$ displays impressive large-scale patterns \cite{ostojic2003patterns,dhar2009pattern,Caracciolo-Paoletti-Sportiello,sadhu2011effect} and has been proposed by Dhar and Sadhu as a model of ``proportionate growth'' \cite{dhar2013sandpile}: as $n$ increases, patterns inside $s_n$ scale up in proportion to the size of the whole.
The second and third authors \cite{Pegden-Smart} have shown that as $n \to \infty$ the sandpile $s_n$ has a scaling limit on $\R^d$. The \emph{sandpile PDE} that describes this limit depends on the set 
of integer superharmonic $d\times d$ matrices.
In the companion paper \cite{Levine-Pegden-Smart} we use \tref{superharmonic} to construct exact solutions to the sandpile PDE. An intriguing open problem is to describe the set of integer superharmonic matrices and analyze the associated sandpile processes for periodically embedded graphs other than $\Z^2$.   (See \cite{Wes}.)

The weak-$*$ limiting sandpile $s : \R^2 \to \R$ appears to have the curious property that it is locally constant on an open neighborhood of the origin. Regions of constancy in $s$ correspond to regions of periodicity in $s_n$.  Ostojic \cite{ostojic2003patterns} proposed classifying which periodic patterns occur in $s_n$. Caracciolo, Paoletti, and Sportiello \cite{Caracciolo-Paoletti-Sportiello} and Paoletti \cite{paoletti2012deterministic} give an experimental protocol that recursively generates $2$-dimensional periodic ``backgrounds'' and $1$-dimensional periodic ``strings.'' 
While this protocol makes no explicit reference to Apollonian circle packings, we believe that the $2$-dimensional backgrounds it generates are precisely the Laplacians $\Delta g_C$ for $C \in \mathcal{B}$. Moreover, periodic regions in $s_n$ empirically correspond to the Laplacians $\Delta g_C$ of our odometers $g_C$ ($C\in \B$), up to error sets of asymptotically negligible size.

\subsection{Supplementary Materials}  One view of this paper is as a proof that a certain recursive algorithm is correct.  That is, we verify that a recursive construction produces odometers.  We explicitly coded this algorithm in Matlab/Octave code and this is available on the second author's website and the arXiv.  We have also included a larger version of the appendix in the arXiv submission.

\subsection*{Acknowledgments} We appreciate several useful discussions with Guglielmo Paoletti, Scott Sheffield, Andrea Sportiello, Katherine Stange, Andr\'as Szenes, and David Wilson.  The authors were partially supported by NSF grants \href{http://nsf.gov/awardsearch/showAward?AWD_ID=1004696}{DMS-1004696}, \href{http://www.nsf.gov/awardsearch/showAward?AWD_ID=1004595}{DMS-1004595}, and \href{http://www.nsf.gov/awardsearch/showAward?AWD_ID=1243606}{DMS-1243606}, respectively.

\section{Overview of the proof}
\label{s.overview}
Here we outline the main elements of the proof of \tref{odometer}, and derive \tref{superharmonic} from \tref{odometer}.
The example in \fref{dec} motivates the following recursive strategy for constructing the odometers $g_C$:

\begin{enumerate}
\item \label{S.tiles} Construct fundamental domains (tiles) $T_C$ for each circle $C \in \B$ such that
\begin{enumerate}
\item \label{S.tile} $T_C$ tiles the plane periodically; and
\item \label{S.decompose} $T_C$ decomposes into copies of smaller tiles $T_{C'}$ with specified overlaps.
\end{enumerate}
\item \label{S.todom} Use the decomposition of $T_C$ to recursively define a \emph{\change{tile} odometer} on $T_C$.
\item \label{S.odom} Extend the \change{tile} odometer to $\Z^2$ via the periodicity condition \eref{periodic}, and
check that the resulting extension $g_C$ satisfies $\Delta g_C\leq 1$ and is maximal in the sense stated in \tref{odometer}.
\end{enumerate}
In \sref{degenerate}, we carry out this strategy completely for two especially simple classes of circles.  

For the general case of \Sref{tiles}, we begin by associating to each pair of tangent circles $C,C' \in \B$ a pair of Gaussian integers $v(C,C'), a(C,C') \in \Z[\I]$, in \sref{lattice}.  The $v(C,C')$'s will generate our tiling lattices, while the $a(C,C')$'s will describe affine relationships among tile odometers.  Recursive relations for these $v$'s and $a$'s, collected in \lref{lattice}, are used extensively throughout the paper.

In \sref{tiles} we construct the tile $T_C$ recursively by gluing copies of tiles of the parent circles of $C$ as pictured in \fref{dec}, using the vectors $v(C,C')$ to specify the relative locations of the subtiles. The difficulty now lies in checking that the resulting set is in fact a topological disk and tiles the plane. This is proved in \lref{tiling}, which recursively establishes certain compatibility relations for tiles corresponding to tangent circles in $\B$.  The argument uses a technical lemma, proved in \sref{topology}, which allows one to prove that a collection of tiles forms a tiling from its local intersection properties.

\sref{odometers} carries out \Sref{todom} by constructing a \emph{tile odometer} (an integer-valued function on $T_C$) for each $C\in \B$. The recursive construction of the tile odometers 
mirrors that of the tiles (compare Definitions \ref{d.T0} and \ref{d.O0})
using the additional data of the Gaussian integers $a(C,C')$, but we must check that it is well-defined where subtiles intersect. A crucial difference between the tiles and the tile odometers is that the former have $90^\circ$ rotational symmetry, a fact which is exploited in the proof of the key \lref{tiling}.
As tile odometers are merely $180^\circ$ symmetric, we require an extra ingredient to assert compatibility of some pairs of tile odometers.  In particular, we show in \sref{strings} that tile boundaries can be suitably decomposed into smaller tiles, a fact which is used in the induction to verify some of the compatible tile odometer pairs. 

The final step is carried out in \sref{maximality}. In \lref{periodic} we use the compatibility relations among tile odometers 
to show that each tile odometer has a well-defined extension $g_C : \Z^2 \to \Z$ satisfying \eref{periodic}. We then show $\Delta g_C\leq 1$ by giving an explicit description of the values of the Laplacian $\Delta g_C$ on subtile intersections (\lref{interior}). In \lref{maximal}, we use this same explicit description of $\Delta g_C$ to prove maximality (the fact that $\Delta g_C \equiv 1$ on the boundary of $T_C$ will be crucial here). The proof of \tref{odometer} is completed by showing that the lattice $L_C$ of \eref{LC} is the same as the lattice $\Lambda_C$ generated by the vectors $v(C,C_i)$, where $C_1,C_2,C_3$ are the parent circles of $C$. \change{Strangely, our proof of this elementary statement depends on everything else in this paper, as it uses the structure of the constructed odometers.  It seems reasonable to expect that a self-contained proof that $L_C=\Lambda_C$ can be found.}

\section{Integer functions on the lattice}

In this section we establish properties of integer functions on $\Z^2$.  In particular, we prove that $\Gamma$ is both downward and topologically closed and derive \tref{superharmonic} from \tref{odometer}.  Two key properties of the set of integer superharmonic matrices are used for this derivation: that the set is downward closed, and that the set is topologically closed.  To begin, we recall some basic properties of the Laplacian.

\begin{proposition}
  The Laplacian agrees with its continuum analogue on quadratic polynomials, is linear, and it monotone.  That is, the following properties hold.
  \begin{enumerate}
  \item  If $A \in \R^{2 \times 2}_{sym}$, $b \in \R^2$, $c \in \R$ and $u(x) = \tfrac12 x \cdot A x$, then $\Delta u \equiv \trace A.$

  \item  $\Delta (u + v) = \Delta u + \Delta v$.

  \item  $\Delta \max \{ u, v \} \leq \max \{ \Delta u, \Delta v \}$, where $\max$ is point-wise. \qed
  \end{enumerate}
\end{proposition}

In particular, this proposition tells us the adding and affine function $u(x) = b \cdot x + c$ does not change the Laplacian.  We next prove that quadratic polynomials with non-negative definite Hessians can be approximated by integer-valued functions with non-negative Laplacians.

\begin{lemma}
If $A \in S_2$ satisfies $A \geq \mathbf{0}$, then there exists $g : \Z^2 \to \Z$ such that
\begin{equation*}
\Delta g(x) \geq 0 \quad \mbox{and} \quad g(x) = \tfrac{1}{2} x^t A x + O(1 + |x|)
\end{equation*}
for all $x \in \Z^2$.
\end{lemma}

Together with the linearity of the Laplacian, this lemma establishes that the set of integer superharmonic matrices is downward closed in the semidefinite order by giving a suitable integer approximation of the difference $x \mapsto \frac 1 2 x^t (B-A) x$.

\begin{proof}
Define $q(x) = \tfrac{1}{2} x^t A x$ and
\begin{equation}
\label{e.gx}
g(x) = \sup_{p \in \Z^2} \inf_{y \in \Z^2} \lfloor q(y) + |y| \rfloor + p\cdot( x - y).
\end{equation}
If we set $p = \lfloor Dq(x) \rfloor$ in the supremum, then we obtain
\begin{equation*}
\begin{aligned}
g(x) & \geq \inf_{y \in \Z^2} q(y) + |y| + \lfloor Dq(x) \rfloor \cdot (x - y) - 1 \\
& \geq \inf_{y \in \Z^2} q(y) + Dq(x) \cdot (x - y) - |x| - 1\\
& \geq q(x) - |x| - 1,
\end{aligned}
\end{equation*}
where we used the triangle inequality in the second step and the convexity of $q$ in the third.  If we set $y = x$ in the infimum, then we obtain
\begin{equation*}
g(x) \leq \sup_{p \in \Z^2} q(x) + |x| + 1 = q(x) + |x| + 1.
\end{equation*}
Finally, since $g$ is a pointwise supremum of affine functions (since the $y$ value realizing the infemum in \eref{gx} is independent of $x$), we obtain $\Delta g \geq 0$ by the monotonicity of $\Delta$.
\end{proof}

For topological closure, we will use the following consequence of a discrete Harnack inequality \cite{Kuo-Trudinger}:
\begin{proposition}
  There is a $C > 0$ such that, if $u : \Z^2 \to \R$, $u(x) \geq 0$, and $\Delta u(x) \leq 1$, then $u(x) \leq C (u(0) + |x|^2)$. \qed
\end{proposition}

We can now prove:

\begin{lemma}
  \label{l.closure}
The set of integer superharmonic matrices is topologically closed.
\end{lemma}

\begin{proof}
  Suppose $A_k \in \Gamma$ and $\lim_{k \to \infty} A_k = A$.  We must show $A \in \Gamma$.

  Step 1. We find $u_k : \Z^2 \to \Z$ and $b_k \in \R^2$ such that $$\Delta u_k(x) \leq 1, \quad u_k(x) \geq \tfrac12 x \cdot A_k x + b_k \cdot x, \quad u_k(0) \leq 1, \quad |b_k| \leq 1.$$
  Replacing $A_k$ with $A_k - \tfrac1k I$ and adding a constant to each $u_k$, we may assume that $u_k(x) \geq \tfrac12 x \cdot A_k x$.  In fact, we may assume that $0 \leq u_k(x_k) - \tfrac12 x_k \cdot  A x_k \leq 1$ holds for some $x_k \in \Z^2$.  Replacing $u$ by $x \mapsto u(x + x_k) - \tfrac12 x_k \cdot A x_k$ and setting $b_k = A_k x_k$, we obtain the first three properties.  Subtracting from $u_k$ the integer harmonic linear function $x \mapsto \tilde b_k x$, for some integer rounding $\tilde b_k$ of $b_k$, we may replace $b_k$ by $b_k - \tilde b_k$ to obtain the fourth property.

  Step 2. We conclude by showing the $u_k$ are precompact.  Since $|A_k|$ is bounded independently of $k$, the Harnack inequality implies that $$v_k(x) = u_k(x) - \tfrac12 x \cdot A_k x - b_k \cdot x$$ satisfies $$|v_k(x)| \leq C(1 + |x|^2).$$ Combined with the boundedness of $|A_k|$ and $|b_k|$, we obtain $$|u_k(x)| \leq C(1 + |x|^2).$$ In particular, for fixed $x \in \Z^2$, there are only finitely many possibilities for $u_k(x)$.  By compactness, we may assume that $u_k \to u$ pointwise as $k \to \infty$.  Since the $u_k$ are integer-valued, this means that $k \mapsto u_k$ is eventually constant on any finite set $X \subseteq \Z^2$.  Since $\Delta u(x)$ only depends on nearest neighbors, we obtain in the limit $$\Delta u(x) \leq 1.$$ By compactness, we may assume $b_k \to b$ and obtain in the limit $$u(x) \geq \tfrac12 x \cdot A x + b \cdot x.$$  It follows that $A \in \Gamma$.
\end{proof}

We are now ready to derive \tref{superharmonic} from \tref{odometer}.
\begin{proof}[Proof of \tref{superharmonic}]
The downward closure of integer superharmonic matrices now makes the implication $(A\leq A_C\implies A\text{ superharmonic})$ from \tref{superharmonic} an immediate consequence of \tref{odometer}.  The implication $(\trace(A)\leq 0\implies A\text{ superharmonic})$ follows from \lref{closure}, since the circle packing $\B$ is dense in the points exterior to its circles.

For the ``only if'' direction, observe that points of tangency are dense on each circle of the band packing $\B$.
In particular, any circle in the plane is either enclosed by some circle of $\B$ or 
strictly encloses some circle of $\B$.  
Therefore any matrix $A \in S_2$ with $\trace(A)>0$ satisfies either $A\leq A_C$ for some $C\in \B$ or $A_C<A$ (that is, $A-A_C$ is positive definite) for some $C \in \B$.  In the latter case, the existence of an integer superharmonic representative $g$ for $A$ would contradict the maximality of $g_{C}$: by \eqref{e.super}, some constant offset $g+c$ of $g$ would satisfy $g+c\geq g_{C}$, but $A_{C}<A$ implies that $g+c-g_{C}$ is unbounded.
\end{proof}

\section{Two Degenerate Cases}
\label{s.degenerate}

\subsection{Explicit constructions}

Before defining odometers for general circles, we consider two subfamilies of $\B$, the Ford and Diamond circles.  The odometers for these subfamilies have simpler structure than the general case, and this allows us to give a more explicit description of their construction.  The purpose of carrying this out is twofold.  First, these cases, particularly the Ford circles, serve as a concrete illustration of our general construction that avoids most of the technical complexity.  Second, by taking these subfamilies as our base-case in the general construction, we avoid several tedious degeneracies.  In particular, the general construction operates quite smoothly when the circle under consideration has parents with distinct non-zero curvatures; this condition fails precisely when the circle is either Ford or Diamond.

\subsection{The Ford circles}
\label{s.ford}

It is a classical fact that the reduced fractions, namely, the pairs of integers $(p,q) \in \Z^2$ such that $q \geq 1$ and $\gcd(q,p) = 1$, are in algebraic bijection with the circles in $\mathcal{B}_0$ that are tangent to the vertical line through the origin.  In curvature coordinates, this bijection is given by
\begin{equation*}
(p,q) \mapsto (q^2, 1, 2 pq)
\end{equation*}
and the circles are called the Ford circles.  We write $C_{pq}$ for the Ford circle associated to the reduced fraction $(p,q)$.

The tangency structure of the Ford circles is famously simple.  If $q \geq 2$, then the parents of the Ford circle $C_{pq}$ in the packing $\B_0$ (\fref{band}) are the vertical line through the origin, and the unique \emph{Ford parents} $C^1_{pq} := C_{p_1 q_1}$ and $C^2_{pq} := C_{p_2 q_2}$ determined by the constraints $p_1 q - q_1 p = 1$, $p_2 q - q_2 p = -1$, $0 \leq q_1 < q$ and $0 < q_2 \leq q$.  The Ford parents are the only Ford circles tangent to $C_{p q}$ having smaller curvature (except that when $q=1$, we have  $q_1=0$ and thus the Ford parents are not both Ford circles). We observe that the definition of parents implies $(p,q) = (p_1,q_1) + (p_2,q_2)$ and $p_1 / q_1 > p/q > p_2/q_2$, which gives the classical relationship between the Ford circles and the Farey fractions.  This connection between the circles and Farey fractions
was used by Ford \cite{Ford} (and later, by Nichols \cite{Nicholls}) to prove results about Diophantine approximation.

The odometer $g_{pq} : \Z^2 \to \Z$ associated to the Ford circle $C_{pq}$ also enjoys a simple description.  To understand the structure of $g_{pq}$, we consider the Ford circle $C_{3,8}$ and its parents $C_{2,5}$ and $C_{1,3}$. These circles have the tangency structure and periodic odometer Laplacians $\Delta g_{pq}$ displayed in \fref{egford}.

\begin{figure}
\raisebox{-0.5\height}{
\begin{tikzpicture}[scale=10,thick,font=\small]
\draw (0,2/3-1/9) -- (0,4/5+1/25);
\draw (1/25,4/5) circle (1/25) node {$C_{2,5}$};
\draw (1/9,2/3) circle (1/9) node {$C_{1,3}$};
\draw (1/64,6/8) circle (1/64);
\draw[thin] (-1/64,6/8) -- (1/64,6/8);
\draw (-3/64,6/8) node {$C_{3,8}$};
\end{tikzpicture}}
\hfill
\raisebox{-0.5\height}{\includegraphics[scale=5]{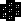}}
\hfill
\raisebox{-0.5\height}{\includegraphics[scale=5]{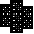}}
\hfill
\raisebox{-0.5\height}{\includegraphics[scale=5]{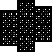}}
\caption{The Ford circles $C_{1,3}$, $C_{2,5}$, and $C_{3,8}$ and several periods of the Laplacian of their odometers $g_{pq} : \Z^2 \to \Z$.  }
\label{f.egford}
\end{figure}
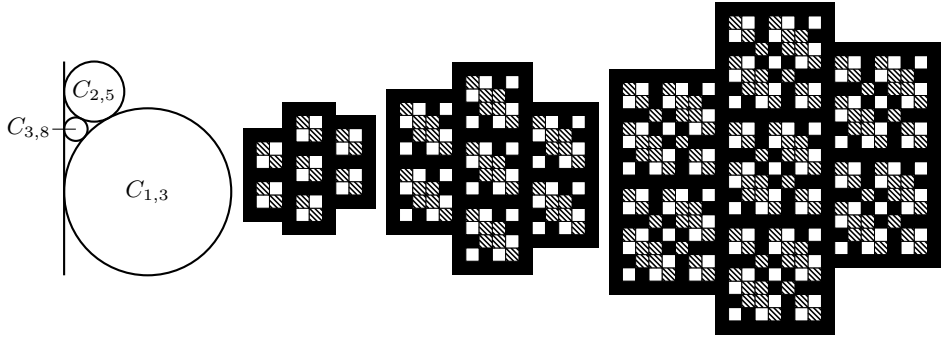

Examining the patterns in \fref{egford}, we see that $[1,q]^2$ is a fundamental domain for the periodic Laplacian $\Delta g_{pq}$.  Moreover, the fundamental tile of $\Delta g_{3,8}$ decomposes (with a few errors on points of overlap) into two copies each of the fundamental tiles of $\Delta g_{1,3}$ and $\Delta g_{2,5}$.

These observations lead to the following construction.  For a general Ford circle $C_{pq}$, we define an odometer $g_{pq} : \Z^2 \to \Z$ by specifying its values on the domain $[0,q]^2 \subseteq \Z^2$ and then extending periodically.  The values of $g_{pq}$ on $[0,q]^2$ are determined recursively by copying data from the parent odometers $g_{p_1 q_1}$ and $g_{p_2 q_2}$ onto the subdomains $E_1 := [0,q_1]^2$, $E_2 := [q_2,q]^2$, $E_3 := [q_1,q] \times [0,q_2]$, and $E_4 := [0,q_2] \times [q_1,q]$.  These subdomains always have one of the three overlapping structures displayed in \fref{fordEi}, depending on the relative sizes of $q_1$ and $q_2$.  This construction is encoded precisely in the following lemma.  \change{Note that our conditions below are redundant, since one only needs two translations to generate a lattice.}

\begin{figure}
\begin{tikzpicture}[scale=.3]
\draw[lightgray,dashed,ultra thin] (0,0) grid (9,9);
\draw (0,0) -- (6,0) -- (6,6) -- (0,6) -- (0,0); \draw (2,2) node {$E_1$};
\draw (9,9) -- (3,9) -- (3,3) -- (9,3) -- (9,9); \draw (7,7) node {$E_2$};
\draw (5,0) -- (5,4) -- (9,4) -- (9,0) -- (5,0); \draw (7,2) node {$E_3$};
\draw (0,5) -- (4,5) -- (4,9) -- (0,9) -- (0,5); \draw (2,7) node {$E_4$};
\end{tikzpicture}
\quad
\begin{tikzpicture}[scale=.3,rotate=90]
\draw[lightgray,dashed,ultra thin] (0,0) grid (9,9);
\draw (0,0) -- (6,0) -- (6,6) -- (0,6) -- (0,0); \draw (2,2) node {$E_3$};
\draw (9,9) -- (3,9) -- (3,3) -- (9,3) -- (9,9); \draw (7,7) node {$E_4$};
\draw (5,0) -- (5,4) -- (9,4) -- (9,0) -- (5,0); \draw (7,2) node {$E_2$};
\draw (0,5) -- (4,5) -- (4,9) -- (0,9) -- (0,5); \draw (2,7) node {$E_1$};
\end{tikzpicture}
\quad
\begin{tikzpicture}[scale=.3]
\draw[lightgray,dashed,ultra thin] (0,0) grid (9,9);
\draw (0,0) -- (5,0) -- (5,5) -- (0,5) -- (0,0); \draw (2,2) node {$E_1$};
\draw (9,9) -- (4,9) -- (4,4) -- (9,4) -- (9,9); \draw (7,7) node {$E_2$};
\draw (4,0) -- (4,5) -- (9,5) -- (9,0) -- (4,0); \draw (7,2) node {$E_3$};
\draw (0,4) -- (5,4) -- (5,9) -- (0,9) -- (0,4); \draw (2,7) node {$E_4$};
\end{tikzpicture}
\caption{The three possible overlapping structures of the subtiles $E_i$ of a Ford tile $[0,q]^2$.  (\change{We are drawing outlines in the dual lattice; so for example, $E_1\cap E_3$ is a just a line of points in $\Z^2$.})}  
\label{f.fordEi}
\end{figure}
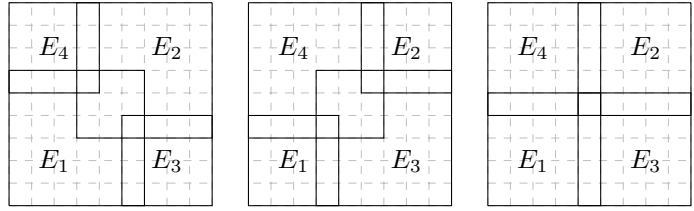

\begin{lemma}
\label{l.fordodo}
There is a unique family of functions $g_{pq} : \Z^2 \to \Z$, indexed by the Ford circles $C_{pq}$, satisfying
\begin{align*}
& g_{pq}(x) = \lceil \tfrac{p}{q} x_1 x_2 \rceil & & \mbox{for } x \in [0,q]^2 \setminus [2,q-2]^2 \\
& g_{pq}(x+(0,-q)) = g_{pq}(x) + (-p,0) \cdot x & & \mbox{for } x \in \Z^2 \\
& g_{pq}(x+(q,q_1)) = g_{pq}(x) + (p_1,p) \cdot x + q_1 p & & \mbox{for } x \in \Z^2 \\
& g_{pq}(x+(-q,q_2)) = g_{pq}(x) + (p_2,-p) \cdot x - q p_2 & & \mbox{for } x \in \Z^2
\intertext{and, when $q \geq 2$,}
& g_{pq}(x) = g_{p_1 q_1}(x) & & \mbox{for } x \in E_1 \\
& g_{pq}(x) = g_{p_1 q_1}(x - (q_2,q_2)) + (p_2,p_2) \cdot x - p_2 q_2 + 1 & & \mbox{for } x \in E_2 \\
& g_{pq}(x) = g_{p_2 q_2}(x - (q_1,0)) + (0,p_1) \cdot x & & \mbox{for } x \in E_3 \\
& g_{pq}(x) = g_{p_2 q_2}(x - (0,q_1)) + (p_1,0) \cdot x & & \mbox{for } x \in E_4,
\end{align*}
where $E_1 := [0,q_1]^2$, $E_2 := [q_2,q]^2$, $E_3 := [q_1,q] \times [0,q_2]$, and $E_4 := [0,q_2] \times [q_1,q]$.
\end{lemma}

\begin{proof}
We prove by induction on $q \geq 1$ that there is a function $g_{pq} : \Z^2 \to \Z$ satisfying the above conditions.  For $p \in \Z$, we observe that
\begin{equation*}
g_{p,1}(x) := \tfrac{1}{2} x_1 (x_1 - 1) + p x_1 x_2
\end{equation*}
satisfies the first four rules.  Thus, we may assume that $q \geq 2$ and the induction hypothesis holds when $q' < q$.

We first observe that any map $g_{pq} : [0,q]^2 \to \Z$ that satisfies the first condition has a unique extension to all of $\Z^2$ that satisfies the next three conditions.  Since the translations of $[0,q]^2$ by the lattice $L_{pq}$ generated by $\{ (0,-q), (q,q_1), (-q,q_2) \}$ cover $\Z^2$, the three translation conditions prescribe values for $g_{pq}$ on the rest of $\Z^2$.  Since $[0,q-1]^2 + L_{pq}$ is a partition of $\Z^2$ and the composition of the three translations is the identity, it suffices to check the consistency of the translation conditions on the set $[0,q]^2 \setminus [1,q-1]^2$.  This follows from the first condition.

Thus, to construct $g_{pq}$, it suffices to check the consistency of the first and the last four conditions, which specify the values of $g_{pq}$ on $[0,q]^2$.  We observe that the first condition prescribes values on the doubled boundary set $B := [0,q]^2 \setminus [2,q-2]^2$ and the last four conditions prescribe values on the sets $E_i$ described above. We must check that the consistency of these five prescriptions.

{\em Case 1: consistency for the intersections $E_1 \cap E_3$, $E_1 \cap E_4$, $E_2 \cap E_3$, and $E_2 \cap E_4$.}  To check the first intersection, we must verify that
\begin{equation*}
g_{p_1 q_1}(x) = g_{p_2 q_2}(x - (q_1,0)) + (0,p_1) \cdot x,
\end{equation*}
for $x \in E_1 \cap E_3 = \{ q_1 \} \cap [0, \min \{ q_1, q_2 \}]$.  Applying the inductive version of the first condition, this reduces to
\begin{equation*}
\frac{p_1}{q_1} x_1 x_2 = \frac{p_2}{q_2} (x_1 - q_1) x_2 + p_1 x_2.
\end{equation*}
Since $x_1 = q_1$, this is easily seen to be true.  The other three intersections can by checked by symmetric arguments.

{\em Case 2: consistency for the intersections $E_1 \cap E_2$ and $E_3 \cap E_4$.}  To check the first intersection, we must verify that
\begin{equation*}
g_{p_1 q_1}(x) = g_{p_1 q_1}(x - (q_2,q_2)) + (p_2,p_2) \cdot x - p_2 q_2 + 1,
\end{equation*}
for all $x \in E_1 \cap E_2 = [q_2,q_1]^2$.  This is non-trivial if and only if $q_1 \geq q_2$, in which case $C_{p_2 q_2}$ is a parent of $C_{p_1 q_1}$.  Let $C_{p_3 q_3}$ denote the other parent. Since $p_3/q_3 > p_1/q_1 > p_2/q_2$, the inductive version of the last four conditions gives
\begin{equation*}
g_{p_1 q_1}(x) = g_{p_3 q_3}(x - (q_2,q_2)) + (p_2,p_2) \cdot x - p_2 q_2 + 1
\end{equation*}
and, since $q_3 = q_1-q_2$,
\begin{equation*}
g_{p_1 q_1}(x - (q_2,q_2)) = g_{p_3 q_3}(x - (q_2,q_2)),
\end{equation*}
for all $x \in [q_2,q_1]^2$.  Substituting this into the above equation, we easily see that equality holds.  The other intersection can be checked by a symmetric argument.

{\em Case 3: consistency for the intersections $B \cap E_i$}.  When $i = 1$, this amounts to verifying
\begin{equation*}
g_{p_1 q_1}(x) = \lceil \tfrac{p}{q} x_1 x_2 \rceil
\end{equation*}
for $x \in [0,q_1] \times [0,1]$.  By induction, this reduces to checking
\begin{equation*}
\lceil \tfrac{p_1}{q_1} x_1 x_2 \rceil = \lceil \tfrac{p}{q} x_1 x_2 \rceil,
\end{equation*}
for $x \in [0,q_1] \times [0,1]$.  When $x_1 x_2 = 0$, this is trivial, so we may assume $x_1 > 0$ and $x_2 = 1$.  Since $\gcd(p,q) = 1$, we see that the distance between $\tfrac{p}{q} x_1$ and the nearest integer is at least $\tfrac{1}{q}$.  Using the relation $q p_1 - p q_1 = 1$, we see that $|\tfrac{p}{q} x_1 - \tfrac{p_1}{q_1} x_1| = \tfrac{1}{q q_1} |x_1| \leq \tfrac{1}{q}$ and therefore $\lceil \tfrac{p_1}{q_1} x_1 x_2 \rceil = \lceil \tfrac{p}{q} x_1 x_2 \rceil$, as desired.  The other intersections can be checked by symmetric arguments.
\end{proof}

In contrast to our construction of general circle odometers in later sections, the description of the Ford circle odometers $g_{pq}$ in \lref{fordodo} has two enormous advantages: the fundamental tiles are square shaped and, more importantly, there is a closed formula for the odometer on the outer two layers of the fundamental tile.  This renders geometric questions about the fundamental tile moot and makes it possible to compute the Laplacian $\Delta g_{pq}$ in a relatively straightforward manner.

Define the {\em boundary} of a subset $X \subseteq \Z^2$ to be
\begin{equation*}
\partial X = \{ x \in X : |x - y| < 2 \mbox{ for some } y \in \Z^2 \setminus X \},
\end{equation*}
and call $X \setminus \partial X$ the {\em interior} of $X$. \change{Thus, $\partial X$ consists of the points in $X$ whose distance to $\Z^2 \setminus X$ is $1$ or $\sqrt{2}$.}  In the case of a square $[0,q]^2$, we have $\partial [0,q]^2 = [0,q]^2 \setminus [1,q-1]^2$.

\begin{lemma}
\label{l.fordlap}
The Laplacian $\Delta g_{pq} : \Z^2 \to \Z$ satisfies
\begin{equation*}
\Delta g_{pq}(x) = \Delta g_{pq}(x + (0,q)) = \Delta g_{pq}(x + (q,q_1)) \quad \mbox{for } x \in \Z^2
\end{equation*}
and
\begin{equation*}
\Delta g_{pq}(x) = 1 \quad \mbox{for } x \in \change{\partial [0,q]^2}.
\end{equation*}
Moreover, if $q \geq 2$ and $x \in [1,q-1]^2$ lies in $2$, $3$, or $4$ of the boundaries \change{$\partial E_i$}, then $\Delta g_{pq}(x) = 1$, $0$, or $-2$, respectively.
\end{lemma}

\begin{proof}
To obtain the periodicity conditions, we simply compute the Laplacian of the periodicity conditions in \lref{fordodo}, recalling that the Laplacian of an affine function vanishes.

We next check $\Delta g_{pq}(x) = 1$ for $x \in [0,q]^2 \setminus [1,q-1]^2$.  We claim that, for $0 \leq a \leq q$,
\begin{equation*}
g_{pq}(a,0) = g_{pq}(0,a) = 0,
\end{equation*}
\begin{equation*}
g_{pq}(a,1) = g_{pq}(1,a) = \lceil \tfrac{p}{q} a \rceil,
\end{equation*}
and
\begin{equation*}
g_{pq}(a,-1) = g_{pq}(-1,a) = - \lfloor \tfrac{p}{q} a \rfloor.
\end{equation*}
The first two are immediate from \lref{fordodo}, so we check the third.  Using the periodicity, we compute
\begin{equation*}
g_{pq}(a,-1) = g_{pq}(a,q-1) - (p,0) \cdot (a,-1) = \lceil \tfrac{p}{q} a (q-1) \rceil - p a = - \lfloor \tfrac{p}{q} a \rfloor.
\end{equation*}
Similarly, we compute
\begin{equation*}
\begin{aligned}
g_{pq}(-1,a) & = g_{pq}(q-1,a+q_1) - (p_1,p) \cdot (-1,a) - q_1 p \\
& = \lceil \tfrac{p}{q} (q-1)(a+q_1) \rceil + p_1 - p a - q_1 p \\
& = - \lfloor \tfrac{p}{q} (a + q_1) \rfloor + p_1 \\
& = - \lfloor \tfrac{p}{q} a - \tfrac{1}{q} \rfloor \\
& = - \lfloor \tfrac{p}{q} a \rfloor, \\
\end{aligned}
\end{equation*}
for $0 \leq a \leq q_2$ and
\begin{equation*}
\begin{aligned}
g_{pq}(-1,a) & = g_{pq}(q - 1,a - q_2) - (-p_2,p) \cdot (-1,a) + p_2 q \\
& = \lceil \tfrac{p}{q} (q-1)(q-q_2) \rceil + p_2 - p a + p_2 q \\
& = - \lfloor \tfrac{p}{q} (a - q_2) \rfloor + p_2 \\
& = - \lfloor \tfrac{p}{q} a - \tfrac{1}{q} \rfloor \\
& = - \lfloor \tfrac{p}{q} a \rfloor, \\
\end{aligned}
\end{equation*}
for $q_2 \leq a \leq q$.  Using $\gcd(p,q) = 1$, we then obtain
\begin{equation*}
\Delta g_{pq}(0,a) = \Delta g_{pq}(a,0) = \lceil \tfrac{p}{q} a \rceil - \lfloor \tfrac{p}{q} a \rfloor = 1
\end{equation*}
for $0 \leq a < q$.  That $\Delta g_{pq} = 1$ on the rest of $[0,q]^2 \setminus [1,q-1]^2$ follows by periodicity.

The moreover clause can be handled similarly, using \lref{fordodo} to explicitly evaluate the Laplacian at the intersections of the boundaries of the sets.  For example, suppose that $x \in (E_1 \cap E_3) \setminus (E_2 \cup E_4)$.  In this case, $x = (q_1, h)$ for some $0 < h < \operatorname{min}\{ q_1, q_2 \}$.  Using $g_{pq} = g_{p_1 q_1}$ on $E_1$, we compute
\begin{equation*}
g_{pq}(x + (0,k)) = \lceil \tfrac{p_1}{q_1} q_1 (h + k) \rceil = p_1 (h+k) \quad \mbox{for } k = -1, 0, 1,
\end{equation*}
and
\begin{equation*}
g_{pq}(x - (1,0) ) = \lceil \tfrac{p_1}{q_1} (q_1 - 1) h \rceil  = p_1 h - \lfloor \tfrac{p_1}{q_1} h \rfloor.
\end{equation*}
Using $q_{pq}(y) = g_{p_2 q_2}(y - (q_1,0)) + (0,p_1) \cdot x$ for $y \in E_3$, we compute
\begin{equation*}
\begin{aligned}
g_{pq}(x + (1,0)) & = g_{p_2 q_2}(1,h) + p_1 h \\
& = \lceil \tfrac{p_2}{q_2} h \rceil + p_1 h \\
& = \lceil \tfrac{p_1}{q_1} h - \tfrac{1}{q_1 q_2} h \rceil + p_1 h\\
& = \lceil \tfrac{p_1}{q_1} h \rceil + p_1 h.
\end{aligned}
\end{equation*}
Putting these together, we obtain $\Delta g_{pq}(x) = 1$.
\end{proof}

The above lemma tells us how to compute $\Delta g_{pq}$.  Indeed, suppose we wish to compute $\Delta g_{pq}(x)$ for some $x \in \Z^2$.  We first reduce to the case $x \in [0,q]^2$ using the periodicity of $\Delta g_{pq}$.  Now, if $x$ lies on the boundary of $[0,q]^2$ or least two of the boundaries of the $E_i$, then we can read off $\Delta g_{pq}(x)$ from \lref{fordlap}.  Otherwise, $x$ lies in the interior of one of the $E_i$, and we can proceed recursively to the corresponding parent $\Delta g_{p_i q_i}(x) = \Delta g_{pq}(x)$.

Following this line of reasoning, we prove \tref{odometer} in the special case of the Ford circles. Observe that the peak associated to the Ford circle $(p,q)$ is the matrix
\begin{equation*}
A_{pq} := \frac{1}{q^2} \left[ \begin{matrix} 1 & pq \\ pq & 0 \end{matrix} \right],
\end{equation*}
which has lattice $L_{pq} := \{ x \in \Z^2 : A_{pq} x \in \Z^2 \}$ generated by $\{ (0,-q), (q,q_1) \}$.

\change{Call $X \subset \Z^2$ \emph{connected} if for all $x,y\in X$ there is a sequence $x=x_0,\ldots,x_k=y$ with $x_i \in X$ and $|x_i-x_{i-1}| = 1$ for all $i=1,\ldots,k$. Call a bounded set $X$ \emph{simply connected} if $\Z^2\setminus X$ is connected.}

\begin{proposition}
\label{p.ford}
For every Ford circle $C_{pq}$, the odometer $g_{pq} : \Z^2 \to \Z$ satisfies
\begin{equation*}
\Delta g_{pq}(x) \leq 1 \quad \mbox{and} \quad g_{pq}(x + v) = g_{pq}(x) + x \cdot A_{pq} v + g_{pq}(v),
\end{equation*}
for all $x \in \Z^2$ and $v \in L_{pq}$.  Moreover, the only infinite connected subset $X \subseteq \Z^2$ such that $\Delta (g_{pq} + \mathbf{1}_X) \leq 1$ is $X = \Z^2$.
\end{proposition}

\begin{proof}
The periodicity condition is immediate from \lref{fordodo}, after we observe $A_{pq} (0,q) = (p,0)$, $A_{pq} (q,q_1) = (p_1,p)$, $g_{pq}(0,q) = 0$, and $g_{pq}(q,q_1) = q_1 p$.  That $\Delta g_{pq} \leq 1$ is immediate from \lref{fordlap}.

To check the moreover clause, we suppose $X \subseteq \Z^2$ is infinite and connected and let $X^C = \Z^2 \setminus X$ denote its complement.  If $\tilde X$ is the complement of any connected component of $X^C$, then $\Delta \mathbf{1}_{\tilde X} \leq \Delta \mathbf{1}_{X}$.  In particular, we may assume that both $X$ and its complement $X^C$ are connected.

If $X^C$ intersects an $L_{pq}$ translation of the boundary of $[0,q]^2$, then, since $X$ is infinite and connected, the set $\{ x \in X^C : \Delta \mathbf{1}_X(x) > 0 \}$ must intersect an $L_{pq}$ translation of the boundary of $[0,q]^2$.  By \lref{fordlap}, $\Delta g_{pq}(x) = 1$ at any such point.  Thus, we may assume that $X^C$ is contained in the interior of $[0,q]^2$.  In particular, it suffices to prove the following claim.

{\em Claim.} Suppose $Y \subseteq [0,q]^2$ is simply connected, $Y \cap [1,q-1]^2$ is non-empty, $Y \setminus [1,q-1]^2$ is connected, and $(Y \setminus [1,q-1]^2) \cap E_i$ is non-empty for at most one $E_i$.  Then there is a point $x \in Y \cap [1,q-1]^2$ such that $\Delta g_{pq}(x) - \Delta \mathbf{1}_Y(x) > 1$.

We proceed by induction on $q$.  This is trivial when $q = 1$, since $[1,q-1]^2$ is empty.  When $q = 2$ \change{we have $[1,q-1]=\{(1,1)\}$ and each of the four nearest neighbors of $(1,1)$ belongs to two of the $E_i$, so} $Y = \{ (1,1) \}$ and, by \lref{fordlap}, we must have $\Delta g_{pq}(1,1) = -2$. We may therefore assume $q > 2$, \change{ which in particular implies $q_1 \neq q_2$}.  By symmetry, we may assume $q_1 > q_2$, so the $E_i$ enjoy the intersection structure displayed on the left of \fref{fordEi}.   We divide the analysis into several cases.

{\em Case 1.} Suppose $Y$ is contained in the interior of some $E_i$. We apply the induction hypothesis to find $x$.

{\em Case 2.} Suppose $x \in \partial Y \cap [1,q-1]^2$ lies in the boundary of exactly two $E_i$. Since \lref{fordlap} gives $\Delta g_{pq}(x) = 1$, we have $\Delta (g_{pq} - \1_Y)(x) > 1$.

{\em Case 3.} Suppose $x \in \partial Y \cap [1,q-1]^2$ lies in the boundary of exactly three $E_i$. Since we are not in the previous case, there are least two neighbors not in $Y$ and thus $\Delta \mathbf{1}_Y(x) < -2$.  By \lref{fordlap}, $\Delta g_{pq}(x) = 0$.

{\em Case 4.} Suppose none of the above cases hold.  Assuming $q_1 > q_2$, then $Y$ does not intersect $E_3$ or $E_4$.  Moreover, it must \change{fail to} intersect either $E_1 \setminus [1,q-1]^2$ or $\change{E_2} \setminus [1,q-1]^2$.  By symmetry, we may assume the former.  If $Y \subseteq E_1$, then we can apply the induction hypothesis to $Y$.  Otherwise, we can apply the induction hypothesis to $Y \cap E_2$. 
\end{proof}

The proof of \tref{odometer} for the Ford circles is completed by the following lemma, which shows that the moreover clause of \pref{ford} suffices to show maximality.

\begin{lemma}
\label{l.suffmax}
If $g : \Z^2 \to \Z$ is superharmonic, $\Delta g$ is doubly periodic, and the only infinite connected set $X \subseteq \Z^2$ such that $\Delta (g + \1_X) \leq 1$ is $X = \Z^2$, then $g$ is maximal in the sense of \tref{odometer}.
\end{lemma}

\begin{proof}
Suppose for contradiction that $h : \Z^2 \to \Z$ satisfies $h \geq g$ and $\Delta h \leq 1$, and that $h - g$ is unbounded.  Replacing $h$ with $h - \min_{x \in \Z^2} (h - g)$, we may assume that $\{ h - g = 0 \} \subseteq \Z^2$ is non-empty.  Using the monotonicity of the Laplacian, we see that $f = \min \{ h, g + 1 \}$ satisfies $\Delta f \leq 1$ and $f - g \in \{ 0 , 1 \}$.  Since $h - g$ is unbounded and $\Delta (h-g)$ is bounded above (here we use the periodicity of $\Delta g$), the set $\{ f - g = 1 \} \subseteq \Z^2$ must contain connected components of arbitrarily large size.  Note that, if $X$ is any connected component of $\{ f - g = 1 \}$, then $\Delta (g + \1_X) \leq 1$.  Using the periodicity of $\Delta g$ and compactness, we can select an infinite connected $X \subsetneq \Z^2$ such that $\Delta (g + \1_X) \leq 1$.  Indeed, let $X_n$ be a sequence of connected components of $\{ f - g = 1 \}$ such that $|X_n| \to \infty$.  Since $\{ f - g = 0 \} \subseteq \Z^2$ is non-empty, we may also select $x_n \in \{ f - g = 0 \}$ which is adjacent to $X_n$.  Using the fact that $\Delta g$ is periodic, we can translate so that $x_n$ always lies in the same period of $\Delta g$.  Now, we can pass to a subsequence $n_k$ such that $x_{n_k} = x^*$ is \change{a constant sequence. Now we can pass to a further subsequence 
such that for each $r \in \N$ the set $X_{n_k} \cap \{ y \in \Z^2 : |y - x^*| \leq r \}$ does not depend on $k$ for $k \geq r$.  Then $\lim_{k \to \infty} X_{n_k}$ exists and has an infinite connected component $X^*$ adjacent to $x^*$. Note that $X^* \neq \Z^2$ (since $x^* \notin X^*$) and} $\Delta g + \Delta 1_{X^*} \leq 1$, contradicting our hypothesis.
\end{proof}

\subsection{The diamond circles}
\label{s.diamond}

The {\em diamond circles} are the circles $(c, c x_1, c x_2) \in \B_0$ which are tangent to $(1,1,0)$ and $(1,1,2)$ and satisfy $0 < x_1 < 1$.  In curvature coordinates, the diamond circles can be parameterized via
\begin{equation*}
C_k := ( 2 k (k+1), 2 k^2-1, 2 k (k+1)),
\end{equation*}
for $k \in \Z^+$.  The peak matrix associated to the diamond circle $C_k$ is
\begin{equation*}
A_k := \frac{1}{2} \left[ \begin{matrix} \frac{k}{k+1} & 1 \\ 1 & \frac{1 - k}{k} \end{matrix} \right],
\end{equation*}
which has lattice $L_k := \{ x \in \Z^2 : A_k x \in \Z^2 \}$ generated by $\{ (0,-2k), (k+1,k) \}$.

To better understand the structure of the diamond circle odometers $g_k : \Z^2 \to \Z$, we examine their Laplacians.  \fref{diamond} displays the patterns associated to the first four diamond circles.  We observe that the periodicity of each pattern is exactly $L_k$.  Moreover, the internal structure of the fundamental tiles are similar enough that we can immediately conjecture what the general case should be.  In fact, for the diamond circles, we have a closed formula for the odometer on each fundamental tile, making this family even simpler than the Ford circles.

\begin{figure}
\raisebox{-0.5\height}{
\begin{tikzpicture}[scale=2,thick,font=\tiny]
\clip (-1/6,0) rectangle (1,2);
\draw (0,-1) -- (0,3);
\draw (1,0) circle (1);
\draw (1,2) circle (1);
\draw (1/4,1) circle (1/4) node {$C_1$};
\draw (7/12,1) circle (1/12);
\draw[thin] (7/12,1) -- (7/12,2/3); \draw (7/12,2/3-.1) node {$C_2$};
\draw (17/24,1) circle (1/24);
\draw[thin] (17/24,1) -- (17/24,4/3); \draw (17/24,4/3+.1) node {$C_3$};
\draw (31/40,1) circle (1/40);
\draw[thin] (31/40,1) -- (31/40,2/3); \draw (31/40,2/3-.1) node {$C_4$};
\end{tikzpicture}}
\raisebox{-0.5\height}{\includegraphics[scale=5]{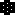}}
\raisebox{-0.5\height}{\includegraphics[scale=5]{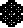}}
\raisebox{-0.5\height}{\includegraphics[scale=5]{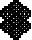}}
\raisebox{-0.5\height}{\includegraphics[scale=5]{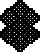}}
\caption{Several periods of the Laplacian of the odometers $g_k : \Z^2 \to \Z$ of the first four diamond circles.  Here, the black, dark patterned, light patterned, and white cells correspond to Laplacian values $1$, $0$, $-1$, and $-2$, respectively.}
\label{f.diamond}
\end{figure}
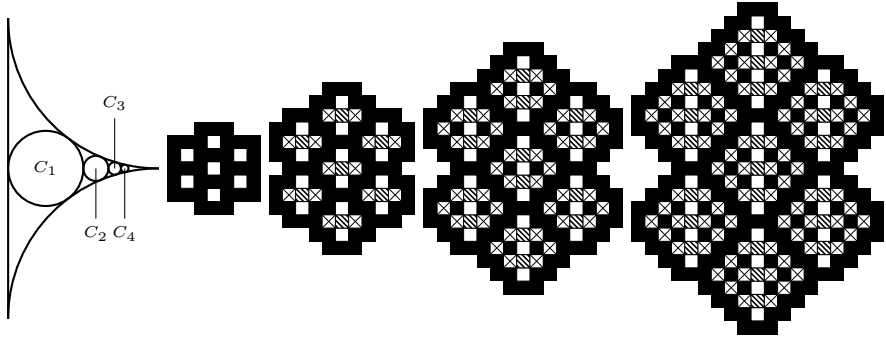

\begin{lemma}
\label{l.diamondodo}
For each $k \in \Z_+$, there is a unique function $g_k : \Z^2 \to \Z$ such that
\begin{align*}
& g_k(x + (0,-2k)) = g_k(x) + (-k,k-1) \cdot x - k (k-2) & & \mbox{for } x \in \Z^2, \\
& g_k(x + (k+1,k)) = g_k(x) + (k,1) \cdot x + \tfrac{1}{2} k(k+1) & & \mbox{for } x \in \Z^2, \\
& g_k(x + (-k-1,k)) = g_k(x) + (0,-k) \cdot x - \tfrac{1}{2} k(k+1) & & \mbox{for } x \in \Z^2,
\intertext{and}
& g_k(x) = \tfrac{1}{2} |x_1| (|x_1| - 1) - \lfloor \tfrac{1}{4} (x_1 - x_2)^2 \rfloor & & \mbox{for } x \in T_k,
\end{align*}
where
\begin{equation*}
T_k := \{ x \in \Z^2 : \max \{ |x_1|, |x_2 - k |, |x_1| + |x_2 - k | - 1 \} \leq k \}.
\end{equation*}
\end{lemma}

\begin{proof}
This is similar to the corresponding proof for the Ford circles above.  In particular, it suffices to check the consistency of the translation conditions on the set $T_k \setminus T_k'$, where
\begin{equation*}
T_k' := \{ x \in \Z^2 : |x_1| + |x_2 - k|  \leq k - 1  \}
\end{equation*}
is the interior of $T_k$.  We claim that this follows from the fourth condition.  If $x, x + (k+1,k) \in T_k$, then either $x = (-j,j)$ for $j = 1, ..., k-1$ or $x = (-1-j,j)$ for $j = 0, ..., k-1$.  We then compute
\begin{equation*}
g_k(x + (k+1,k)) = g_k(x) + (k,1) \cdot x + \tfrac{1}{2} k (k+1)
\end{equation*}
in either case.  This implies the consistency of the second condition.  The first and third conditions follow similarly.
\end{proof}

\begin{lemma}
\label{l.diamondlap}
For each $k \in \Z_+$, the Laplacian $\Delta g_k$ satisfies
\begin{align*}
& \Delta g_k(x + (0,2k)) = \Delta g_k(x + (k+1,k)) = \Delta g_k(x) & & \mbox{for } x \in \Z^2 \\
& \Delta g_k(x) = 1 & & \mbox{for } x \in T_k \setminus T_k' \\
& \Delta g_k(x) = (-1)^{x_1 + x_2} - \mathbf{1}_{\{ 0 \}}(x_1) & & \mbox{for } x \in T_k'.
\end{align*}
\end{lemma}

\begin{proof}
The first two conditions follow by computations analogous to those in the proof of \lref{fordlap}. The formula in the third condition follows by inspection of the formula for $g_k$ on $T_k$.
\end{proof}

\begin{proposition}
\label{p.diamond}
For every diamond circle $C_k$, the odometer $g_k : \Z^2 \to \Z$ satisfies
\begin{equation*}
\Delta g_k(x) \leq 1 \quad \mbox{and} \quad g_k(x + v) = g_k(x) + x \cdot A_k v + g_k(v)
\end{equation*}
for all $x \in \Z^2$ and $v \in L_k$.  Moreover, the only infinite connected subset $X \subseteq \Z^2$ such that $\Delta (g_k + \mathbf{1}_X) \leq 1$ is $X = \Z^2$.
\end{proposition}

\begin{proof}
This follows from the above lemma and the first part of the argument of \pref{ford}.  We can avoid the recursive argument in this case because of the explicit computation of $\Delta g_k$ on $T_k'$ in \lref{diamondlap}.
\end{proof}

\section{Circles and Lattices}
\label{s.lattice}

\subsection{Periodicity conditions}
When defining odometers for the Ford and diamond circles above, we extended a finite amount of data to all of $\Z^2$ via periodicity conditions of the form
\begin{equation*}
g(x + v) = g(x) + a \cdot x + k \quad \mbox{for } x \in \Z^2,
\end{equation*}
where $v, a \in \Z^2$ and $k \in \Z$.  In each of \lref{fordodo} and \lref{diamondodo}, three pairs of vectors $(v_i, a_i)$ appear in such conditions, and these vectors have several suggestive properties.  If we view each vector as a Gaussian integer $\Z[\I]$ via the usual identification $x \mapsto x_1 + \I x_2$, then we have
\begin{equation*}
v_1 + v_2 + v_3 = 0\quad\mbox{and}\quad a_1 + a_2 + a_3 = 0.
\end{equation*}
In this section, we generalize these vectors, associating pairs $(v_i, a_i)$ of vectors to each circle in $\B$. Our calculations largely follow those of Stange \cite{Stange}, who, motivated by data on our lattices $L_C$ for $C\in \B$ and Conway's association of lattices to quadratic forms \cite{Conway}, studied ways to associate lattices to circles in an Apollonian circle packing.

\subsection{Action of the Apollonian group}
For the rest of this paper, we identify $\C$ with $\R^2$ and $\Ga$ with $\Z^2$ in the usual way.  Curvature coordinates are made complex by writing $C=(c,cz)\in \R\times \C$ for a circle with radius $c^{-1}$ and center $z\in \C$.  Part of the band packing in complex curvature coordinates is shown in \fref{band}.  As noted in the introduction, a pairwise-tangent triple $(C_1,C_2,C_3)$ is related linearly to its Soddy circles $C_0,C_4$ in curvature coordinates by
$$C_0+C_4=2(C_1+C_2+C_3).$$
This relation works also for lines, with the convention that the curvature coordinates of a line $\ell$ are $(0,z)$ where $z$ is the unit normal vector to the line, oriented away from the component of $\R^2\stm \ell$ containing the other circles in the triple.  In particular, all lines in the circle packing $\B$ have coordinates $(0,-1)$ or $(0,1)$.

A {\em Descartes quadruple} is a list of four circles such that any three form a pairwise-tangent triple.  As any pairwise tangent triple of circles has exactly two Soddy circles, any pairwise tangent triple of circles can likewise be completed to exactly two Descartes quadruples, up to permutation. We call a Descartes quadruple $(C_0,C_1,C_2,C_3) \in \B^4$ {\em proper} if the curvatures satisfy $c_0 > \max \{ c_1, c_2, c_3 \}$ and points of tangency between $C_0$ and $C_1$, $C_2$, $C_3$ are clockwise around $C_0$.

Note that, if $(C_0,C_1,C_2,C_3) \in \B^4$ is a proper Descartes quadruple, then so is the {\em parent rotation} $(C_0,C_2,C_3,C_1)$ and the {\em successor} $(2(C_0 + C_2 + C_3) - C_1, C_0, C_2, C_3)$.  Each circle $C_0\in \B$ determines a Descartes quadruple $(C_0,C_1,C_2,C_3)$ up to parent rotation. Moreover, any pairwise tangent triple $(C_1,C_2,C_3) \in \B^3$ can be completed to a proper Descartes quadruple $(C_0,C_1,C_2,C_3)$ in at most one way.

\begin{figure}
\begin{center}
\nofig{\input{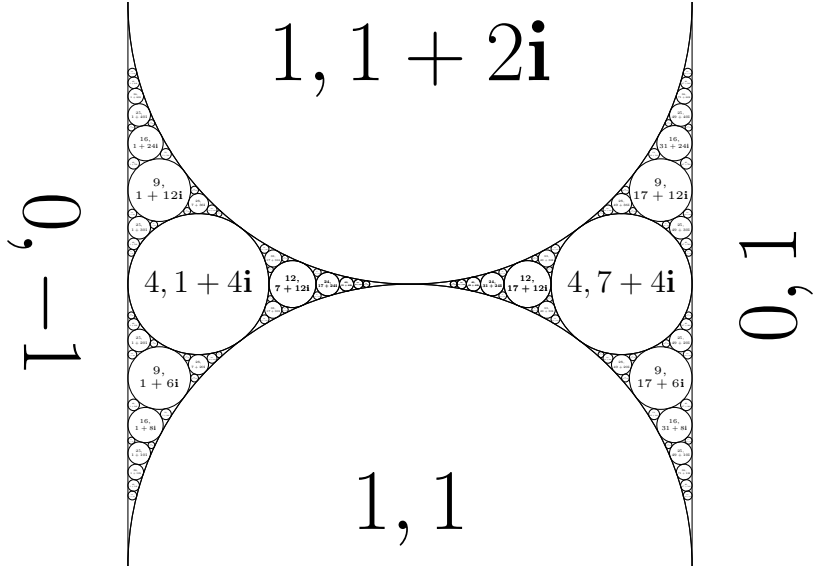}}
\end{center}
\caption{\label{f.band} The Apollonian Band packing $\B_0$ with complex curvature coordinates, in the region $0<\Im(z)<2$.  The circles tangent to the line $(0,-1)$ are Ford circles.}
\end{figure}

\subsection{Lattice vectors}
\label{s.latticevectors}

We now assign vectors $v(C,C'), a(C,C') \in \Z[\I]$ to each pair of tangent circles $C, C' \in \B$.  In analogy to the two degenerate cases, the vectors $v(C,C')$ and $a(C,C')$ will generate periodicity conditions of the odometers associated to $C$ and $C'$.

To describe our recursive construction, let us call a Descartes quadruple {\em semi-proper} if it is either proper or a parent rotation of
\[
(C_0,C_1,C_2,C_3)=(\,(1,1+2z),\,(1,1+2z+2\I),\,(0,1),\,(0,-1)\,)
\]
for some $z \in \Ga$.  These latter quadruples are associated to the large circles in $\B$ and fail to be proper because one of the parents has the same curvature as the child.  Note that, modulo parent rotations, the set of semi-proper quadruples is a directed forest in which each node has exactly one parent and three children.  We induct along this tree structure to define our vectors.

For each semi-proper Descartes quadruple $(C_0,C_1,C_2,C_3)$ we define $v(C_j,C_i)$ and $a(C_j,C_i)$ for all rotations $(i,j,k)$ of $(1,2,3)$.  For the base case, we consider the clockwise Descartes quadruple
\[
(C_0,C_1,C_2,C_3)=(\,(1,1+2z),\,(1,1+2z+2\I),\,(0,1),\,(0,-1)\,)\quad\mbox{for}\quad z\in \Ga,
\]
and set
\begin{equation*}
\begin{aligned}
v(C_3,C_2) & = 0 \\
v(C_2,C_1) & = 1 \\
v(C_1,C_3) & = -1
\end{aligned}
\qquad
\begin{aligned}
a(C_3,C_2) & = 1 \\
a(C_2,C_1) & = z\\
a(C_1,C_3) & = -1-z.
\end{aligned}
\end{equation*}
For the induction step, we fix a semi-proper Descartes quadruple $(C_0,C_1,C_2,C_3)$ and suppose that $v(C_j,C_i)$ and $a(C_j,C_i)$ have been defined for all rotations $(i,j,k)$ of $(1,2,3)$.  For the successor quadruple $(2(C_0+C_2+C_3)-C_1,C_0,C_2,C_3)$, we define
\begin{equation*}
\label{e.latticeinductdefin}
\begin{aligned}
v(C_2,C_0) & = v(C_2,C_1) - \I v(C_3,C_2) \\
v(C_0,C_3) & = v(C_1,C_3) + \I v(C_3,C_2) \\
a(C_2,C_0) & = a(C_2,C_1) + \I a(C_3,C_2) \\
a(C_0,C_3) & = a(C_1,C_3) - \I a(C_3,C_2),
\end{aligned}
\end{equation*}
so that $v(C_j,C_i)$ and $a(C_j,C_i)$ are defined for all rotations $(i,j,k)$ of $(0,2,3)$.  Since every pair of tangent circles $C,C' \in \B$ can be completed to a proper Descartes quadruple of the form $(C_0,C,C',C_3)$, this recursive construction generates vectors for all pairs of tangent circles in $\B$.  Since the quadruple of the big circle $(1,1+2 z)$ is a successor of the quadruple of $(1,1+2z + 2\I)$, we must also check everything is well-defined, but this is immediate from the definition.

\begin{lemma}
\label{l.lattice}
If $(C_0,C_1,C_2,C_3) \in \B^4$ is a proper Descartes quadruple and we write $C_i = (c_i, c_i z_i)$, $v_{ij} = v(C_i,C_j)$, and $a_{ij} = a(C_i,C_j)$, then the following hold.
\begin{subequations}
\label{e.lattice}
\begin{align}
& v_{32} + v_{13} + v_{21} = 0 & &a_{32} + a_{13} + a_{21} = 0 \label{e.lattice.superbasis}\\
& v_{10} = v_{13} - \I v_{21} & &a_{10} = a_{13} + \I a_{21} \label{e.lattice.induct}\\
& v_{01} = \I v_{10} & &a_{01} = -\I a_{10} \label{e.lattice.flip}\\
& v_{32}^2 = c_3 c_2 (z_3 - z_2) & &2 v_{32} a_{32} = c_3 z_3 + c_2 z_2 \label{e.lattice.matrix}\\
& \bar v_{13} v_{21} + v_{13} \bar v_{21} =  -2c_1 \label{e.lattice.determinant} && 
\end{align}
\end{subequations}
\end{lemma}
The properties \eref{lattice.superbasis}, \eref{lattice.induct}, and \eref{lattice.flip} give inductive relationships among the vectors. 
The property \eref{lattice.matrix} expresses the vectors $v(C,C')$ and $a(C,C')$ up to sign in terms of the circles $C,C'\in \B$.  
Finally, \eref{lattice.determinant} implies that the determinant of the lattice generated by $v_{31}$ and $v_{21}$ is $c_1$.  Observe that since each statement holds also for any rotation of the Descartes quadruple, we have, e.g., that $v_{02}=\I v_{20}$, etc.

\begin{proof}[Proof of \lref{lattice}]
The inductive construction gives \eref{lattice.superbasis}--\eref{lattice.flip} immediately.  To check \eref{lattice.matrix}, we first observe that it holds for the base case and its immediate successors.  By induction, it suffices to assume \eref{lattice.matrix} holds for the proper Descartes quadruples $(C_0,C_1,C_2,C_3), (C_1,C_4,C_2,C_3) \in \B^4$ and check the conditions for each of the three successors of $(C_0,C_1,C_2,C_3)$. We write $w_i = c_i z_i$ and compute
\begin{equation*}
\begin{aligned}
v_{10}^2 & = (v_{13} - \I v_{21})^2 \\
& = (v_{13} - v_{12})^2 \\
& = 2 v_{13}^2 + 2 v_{12}^2 - (v_{13} + v_{12})^2 \\
& = 2 v_{13}^2 + 2 v_{12}^2 - v_{14}^2 \\
& = 2 (c_3 w_1 - c_1 w_3) + 2(c_2 w_1- c_1 w_2) - (c_4 w_1 - c_1 w_4) \\
& = (2(c_1 + c_2 + c_3) - c_4) w_1 - c_1 (2(w_1+w_2+w_3) - w_4) \\
& = c_0 w_1 - c_1 w_0,
\end{aligned}
\end{equation*}
using the induction hypotheses and the linear Soddy relation. We also compute
\begin{equation*}
\begin{aligned}
v_{10} a_{10} & = (v_{13} - \I v_{21})(a_{13} + \I a_{21}) \\
& = v_{13} a_{13} + v_{21} a_{21} + \I (v_{13} a_{21} - v_{21} a_{13}) \\
& = v_{13} a_{13} + v_{21} a_{21} - v_{13} a_{12} - v_{12} a_{13} \\
& = 2 v_{13} a_{13} + 2 v_{21} a_{21} - (v_{13} + v_{12})(a_{13}+a_{12}) \\
& = 2 v_{13} a_{13} + 2 v_{21} a_{21} - v_{14} a_{14} \\
& = \tfrac{1}{2} (w_1 + 2 (w_1 + w_2 + w_3) - w_4) \\
& = \tfrac{1}{2}(w_1 + w_0).
\end{aligned}
\end{equation*}
The relations $v_{i0}^2 = c_i c_0 (z_i - z_0)$ and $2 v_{i0} a_{i0} = c_i z_i + c_0 z_0$ for $i = 2, 3$ follow by analogous computations.

The relation \eref{lattice.determinant} can be obtained as follows.  Observe that $v_{ij}^2 = c_i c_j (z_i - z_j)$ implies $|v_{ij}|^2 = c_i + c_j$ and therefore
\begin{equation*}
\begin{aligned}
v_{13} v_{21} (\bar v_{13} v_{21} + v_{13} \bar v_{21}) & = (c_1 + c_3) v_{21}^2 + (c_2 + c_1) v_{13}^2 \\
& = c_1 (v_{21}^2 + v_{13}^2) + c_3(c_1 w_2 - c_2 w_1) + c_2(c_3 w_1 - c_1 w_3) \\
& = c_1 (v_{21}^2 + v_{13}^2) - c_1 (c_3 w_2 - c_2 w_3) \\
& = c_1 (v_{21}^2 + v_{13}^2 - v_{32}^2) \\
& = - 2 c_1 v_{21} v_{13}.
\end{aligned}
\end{equation*}
Since at most one of the $C_i$ is a line when we are not in the base case, we have $v_{ij} \neq 0$ and thus obtain the next relation.  
\end{proof}

Given a circle $C_0$ in the band packing $\B$ with the  proper Descartes quadruple $(C_0,C_1,C_2,C_3)$, we define the lattice
	\[ \Lambda_{C_0} = \Z v(C_1,C_0) + \Z v(C_2,C_0) + \Z v(C_3,C_0) \]
(any two of the three summands suffice since $\sum_{i=1}^3 v(C_i,C_0) = 0$).
Next we compare $\Lambda_C$ with the lattice $L_C$ of \eref{LC}.

\begin{lemma}
\label{l.lattice1}
$\Lambda_C \subseteq L_{C}$ for all $C \in \B$.
\end{lemma}

\begin{proof}

If $C_0 \in {\B}$ has radius $1$, then $\Lambda_{C_0} = L_{C_0}=\Z^2$ and there is nothing to prove.  Otherwise, we may assume that $C_0$ is part of a proper Descartes quadruple $(C_0,C_1,C_2,C_3) \in {\B}^4$.  Using \lref{lattice}, we compute for $i=1,2,3$
\begin{equation}
\begin{aligned}
\label{e.ai0vi0}
a_{i0} & = \tfrac{1}{2} (c_i z_i + c_0 z_0) v_{i0}^{-1} \\
& = \tfrac{1}{2} (c_0^{-1} v_{i0}^2 + (c_i + c_0) z_0) v_{i0}^{-1} \\
& = \tfrac{1}{2} c_0^{-1} v_{i0} + \tfrac{1}{2} z_0  \overline{v_{i0}} \\
& = A_{C_0} v_{i0}.
\end{aligned}
\end{equation}
Since $A_{C_0}v_{i0}=a_{i0} \in \Z[\I]$ we conclude that $v_{i0} \in L_{C_0}$ for $i=1,2,3$.
\end{proof}

In fact $\Lambda_{C}=L_C$ for all $C\in \B$, but our proof of the inclusion $L_{C}\subseteq \Lambda_{C}$ uses our construction of the odometers $g_C$, and is thus postponed to the end of \sref{maximality}.
\smallskip

Observe that \eref{lattice.induct} can be rewritten as $v_{0i}= v_{ji} \pm v_{ki}$, where $(i,j,k)$ is a rotation of $(1,2,3)$ and sign of $v_{ki}$ depends on whether or not $C_i$ is a parent of $C_k$. Inductively, this implies that a vector $v(C,C_0)$ lives in the lattice of the circle $C_0$:

\begin{lemma}
\label{l.biglatticevectors}
Given a proper Descartes quadruple $(C_0,C_1,C_2,C_3) \in \B^4$ and any $C\in\B$ which is tangent to and smaller than $C_0$, we have that $v(C,C_0)\in \Lambda_{C_0}$.\qed
\end{lemma}

Finally, we check that our general lattice vector construction agrees with the degenerate cases.

\begin{proposition}
\label{p.fordlattice}
Every Ford circle $C_0 = (q^2, 1 + 2 pq \I)$ is part of a proper quadruple $(C_0,C_1,C_2,C_3)$, where $C_1 = (q_1^2, 1 + 2 p_1 q_1 \I)$, $C_2 = (q_2^2, 1 + 2 p_2 q_2 \I)$, and $C_3 = (0,-1)$.  Moreover, we have
\begin{equation*}
\begin{aligned}
v(C_1,C_0) & = q + q_1 \I \\
v(C_2,C_0) & = - q + q_2 \I \\
v(C_3,C_0) & = - q \I \\
\end{aligned}
\qquad
\begin{aligned}
a(C_1,C_0) & = p_1 + p \I \\
a(C_2,C_0) & = p_2 - p \I \\
a(C_3,C_0) & = -p \I, \\
\end{aligned}
\end{equation*}
which are exactly the vectors appearing in \lref{fordodo}.
\end{proposition}

\begin{proof}
That $(C_0,C_1,C_2,C_3)$ is proper follows from the discussion in \sref{ford}.  Since \eref{lattice.superbasis} and \eref{lattice.matrix} determines the tuple $(v_{10},v_{20},v_{30},a_{10},a_{20},a_{30})$ up to sign, it is enough to show $v_{30} = - q \I$.  Using \eref{lattice.induct} and \eref{lattice.flip}, we inductively compute
\begin{equation*}
v_{30} = v_{32} - \I v_{13} = v_{32} + \I v_{13} = - \I q_2 - \I q_1 = - \I q.
\end{equation*}
We conclude after reading off the base case $C_0 = (1,1+2)$ from the beginning of this section.
\end{proof}

\begin{proposition}
\label{p.diamondlattice}
Every diamond circle $C_0 = (2 k (k+1), 2 k^2 - 1 + 2 k (k+1) \I)$ is part of a proper quadruple $(C_0,C_1,C_2,C_3)$, where $C_1 = (1,1+2\I)$, $C_2 = (1,1)$, and $C_3 = (2 (k-1) k, 2 (k-1)^2 - 1 + 2 (k-1) k \I)$.  Moreover, we have
\begin{equation*}
\begin{aligned}
v(C_1,C_0) & = k+1+k\I \\
v(C_2,C_0) & = - k - 1 + k\I \\
v(C_3,C_0) & = - 2 k \I \\
\end{aligned}
\qquad
\begin{aligned}
a(C_1,C_0) & = k+\I \\
a(C_2,C_0) & = -k \I \\
a(C_3,C_0) & = -k + (k-1)\I, \\
\end{aligned}
\end{equation*}
which are exactly the vectors appearing in \lref{diamondodo}.
\end{proposition}

\begin{proof}
That $(C_0,C_1,C_2,C_3)$ is proper follows from the discussion in \sref{diamond}.  As in the previous proposition, it is enough to show $v(C_1,C_0) = k + 1 + k \I$.  Using \eref{lattice.induct} and \eref{lattice.flip}, we inductively compute
\begin{equation*}
v_{10} = v_{13} - \I v_{21} = (k + (k-1) \I) + (1 + \I) = k + 1 + k \I.
\end{equation*}
We conclude after reading off the base case $C_0 = (4,1 + 4 \I)$ from the previous proposition.
\end{proof}

\subsection{Symmetry reduction}
\label{s.lattice-sym}

The band packing $\B$ is invariant under the operations
\begin{equation*}
\begin{aligned}
& (c,z) \mapsto (c,-z) \\
& (c,z) \mapsto (c,\bar z) \\
& (c,z) \mapsto (c,z+2c \I) \\
& (c,z) \mapsto (c,z+2c).
\end{aligned}
\end{equation*}
These operations can be extended to the vectors $v(C,C')$ and $a(C,C')$ in the obvious way.  For example, if we apply the shift $(c,z) \mapsto (c,z+2cw)$ for some $w \in \Ga$, then $v(C,C')$ is unchanged while $a(C,C')$ is replaced by $a(C,C') + c c' v(C,C')^{-1}$.

We can extend the results of \sref{degenerate} to the orbit of the Ford and diamond circles under these symmetries.  For example, suppose that $C = (c,z)$ is a Ford circle and we want to construct an odometer for the shifted circle $C' = (c,z+2c w)$ for $w = a + b \I \in \Ga$.  Observe that
\begin{equation*}
A_{C'} - A_C = \mat{ a & b \\ b & - a }
\end{equation*}
and that the function $h_w : \Z^2 \to \Z$ given by
\begin{equation*}
h_w(x) = \frac{a}{2} x_1 (x_1 + 1) - \frac{a}{2} x_2 (x_2 + 1) + b x_1 x_2
\end{equation*}
satisfies $\Delta h_w \equiv 0$.  In particular, $g_{C'} = g_C + h_w$ is an odometer for $C'$.

\section{A topological lemma}
\label{s.topology}
In the course of our general tile construction, it is necessary to translate local knowledge of tile compatibility to global knowledge regarding the intersection structure of the collection of tiles.  This section concerns a technical lemma which is our primary tool for checking that certain collections of tiles form tilings using 
local conditions.

We continue to identify $\Ga$ with $\Z^2$ in the natural way.  Thus a tile $T$ is a set of squares $s_x=\{x,x+1,x+\I,x+1+\I\}\sbs \Ga$, where $I(T)$ (adapted from \eref{disc} to $\C$ in the obvious way) is a topological disc.  We let $\Ga$ inherit the standard degree-4 square lattice graph of $\Z^2$.    For any tile, define the \emph{footprint} $\foot(T)$ to be the subgraph of $\Ga$ induced by the union of squares $s_x\in T$.  For two tiles $T_1, T_2$ write $\foot(T_1) \cap \foot(T_2)$ for the induced subgraph on the intersection of the vertex sets of $\foot(T_1)$ and $\foot(T_2)$.
In an abuse of terminology, we say that tiles $T_1$ and $T_2$ \emph{intersect} if $\foot(T_1)\cap \foot(T_2)\neq \varnothing$, and \emph{overlap} if $T_1\cap T_2\neq \varnothing$.  If $T_1\cap T_2=\varnothing$, they are \emph{non-overlapping}.

\change{Recall from Section~\ref{s.hextilings} that we define a \emph{tiling} as a collection of tiles such that every square $s_x$ of $\Ga$ belongs to a unique tile. Thus, the tiles in a tiling are permitted to intersect, but not to overlap.}

\begin{lemma}
\label{l.topological}
Suppose $\mathcal{T}$ is an (infinite) collection of tiles and $\mathcal{E}$ is a set of two-element subsets of $\mathcal{T}$ satisfying the following four hypotheses.
\begin{enumerate}
\item \label{P.graph} The graph $G=(\mathcal{T},\mathcal{E})$ is a 3-connected planar triangulation.
\item \label{P.periodic} $\mathcal{T}$ and $\mathcal{E}$ are invariant under translation by some full-rank lattice $L \subseteq \Z[\I]$, and $\sum_{T \in \mathcal{T}/L} |T| = |\det L|$.
\item \label{P.ints} If $\{ T_1, T_2 \}\in \eee$ then the intersection $\foot(T_1) \cap \foot(T_2)$ contains at least 2 vertices.  
\item \label{P.threes} We can select for each face $F=\{T_1,T_2,T_3\}$ of $G$ a point $\rho(F)\in \foot(T_1)\cap \foot(T_2)\cap \foot(T_3)$ such that, for the face $F'=\{T_1,T_2,T_4\}$ of $G$, $\rho(F)$ and $\rho(F')$ lie on some path contained in $\foot(T_1) \cap \foot(T_2)$.
\end{enumerate}
Then $\ttt$ is a tiling.  Moreover, if $T_1, T_2 \in \mathcal{T}$ and $T_1$ and $T_2$ intersect, then $\{ T_1, T_2 \} \in \mathcal{E}$, and $\foot(T_1) \cap \foot(T_2)$ is a path joining $\rho(F)$ and $\rho(F')$ for the faces $F,F'$ with $F\cap F'=\{T_1,T_2\}$. 
\end{lemma}

The proof is essentially a homotopy argument.   The task is complicated slightly by the fact that we assume in \eqref{P.threes} merely that each intersection $\foot(T_1)\cap \foot(T_2)$ \emph{contains} a path, and not that it is equal to a path.  Although this weaker assumption prevents us from deducing the Lemma from standard topological facts, it will significantly simplify our inductive argument in \sref{tiles} (see Remark \ref{r.topologyvscases}).

Before commencing, recall that the \emph{winding number} of a closed walk in $\Ga$ about a point $z^*\in \Ga^*$ is the number of times the walk circles the point $z^*$ counterclockwise (here, $\Ga^*$ is the dual lattice to $\Ga$, whose vertices are centered in the faces of $\Ga$).   Thus in terms of the vertices $z_1,z_2,\dots,z_t=z_1$ of the walk, the winding number is given by
\newcommand{\wind}{\mathrm{wind}}
\[
\wind(z_1,\dots,z_t;z^*):=\frac{1}{2\pi} \sum_{i=1}^{t-1}\arg\left(\frac{z_{i+1}-z^*}{z_i-z^*}\right).
\]
We say that the closed walk $W$ \emph{encloses} a point $z^*$ if the winding number of $W$ about $z^*$ is nonzero; similarly, we say that $W$ encloses a square $s_x=\{x,x+1,x+\I,x+1+\I\} \sbs \Ga$ if the winding number of the walk about  $s_x^*=x+\tfrac 1 2 +\tfrac \I 2$ is nonzero. Note that the definition of a tile (via \eref{disc}) implies that any tile $T$ is a set of squares enclosed by a simple cycle $\partial T$ in $\Z^2$, which is the set of points and edges of $T$ which each also lie in a square $s_x\notin T$.   

\begin{proof}[Proof of \lref{topological}]
Our first goal is to show that $\ttt$ is a tiling.  Note that it suffices to show that every square $s_x=\{x,x+1,x+\I,x+1+\I\}$ lies in some tile in $\ttt$, since \hyref{periodic} implies that the average number of tiles a square lies in is 1. The idea is to use the periodicity to draw a large cycle that surrounds a given square $s_x$, and then use the graph structure to contract this cycle to the boundary of a single tile.

We fix an $L$-periodic drawing of $G=(\ttt,\eee)$, and work with the dual graph $G^*$, whose vertex set $\fff$ is the set of triangles in the graph $(\ttt,\eee)$.  For $F\in \fff$, we have that $\rho(F)$ is a point in the intersection of the three tiles in $F$, and for adjacent $F,F'$, let $\rho(F,F')$ be a choice of path from $\rho(F)$ to $\rho(F')$ in the intersection of the two tiles in $F\cap F'$, guaranteed to exist by \hyref{threes}.  Given a sequence $F^0,F^1,\dots,F^n=F^0$ with $\abs{F^i\cap F^{i+1}}=2$ (indices evaluated modulo $n$), we define the closed walk $\ell(F^0,\dots,F^n)$ in $\Z^2$ as the consecutive concatenation of the paths $\rho(F^i,F^{i+1})$ $(0\leq i<n)$.

 Since the graph $(\ttt,\eee)$ is connected and periodic under a nontrivial lattice, we can, given any finite set $Z\sbs \Ga^*$, find a cycle $F^0, \dots, F^n=F^0$ in $G^*$ which wraps around each point $z^*\in Z$, just in the sense that  $\wind(\rho(F^1),\dots,\rho(F^n);z^*)=1$ for all $z^*\in Z$.  Since \hyref{periodic} also implies that the tiles $T\in \ttt$ are of bounded size, this means that given the point $s_x^*\in \Ga^*$, we can (by choosing $Z$ to be the set of all points within some sufficiently large distance of $s_x^*$) find a cycle $C$ given by $F^0,F^1,\dots,F^n=F^0$ such that $\ell(C)$ encloses $s_x^*$ and so $s_x$.  Letting $\kkk$ denote the set of faces of $G^*$  in the region of the plane bounded by $C$ in our fixed embedding (so, each element of $\kkk$ corresponds to some $T\in \ttt$), we choose this $C$ such that $\abs \kkk$ is as small as possible, subject to the condition that $\ell(C)$ encloses $s_x^*$.

\begin{figure}
\begin{center}
  \begin{pdfpic}
    \psset{radius=1pt}
    \begin{pspicture}(0,0)(5,5)

      \psset{linewidth=.05pt}
      
\SpecialCoor
      \rput(! 0 Rand .4 mul add 0 Rand .4 mul add ){\Cnode*{A1}}
      \rput(! 1 Rand .4 mul add 0 Rand .4 mul add ){\Cnode*{B1}}
      \rput(! 2 Rand .4 mul add 0 Rand .4 mul add ){\Cnode*{C1}}
      \rput(! 3 Rand .4 mul add 0 Rand .4 mul add ){\Cnode*{D1}}
      \rput(! 4 Rand .4 mul add 0 Rand .4 mul add ){\Cnode*{E1}}
      \rput(! 5 Rand .4 mul add 0 Rand .4 mul add ){\Cnode*{F1}}

      \psline(A1)(B1)
      \psline(B1)(C1)
      \psline(C1)(D1)
      \psline(D1)(E1)
      \psline(E1)(F1)

      \rput(! 0.5 Rand .4 mul add .866 Rand .4 mul add ){\Cnode*{A2}}
      \rput(! 1.5 Rand .4 mul add .866 Rand .4 mul add ){\Cnode*{B2}}
      \rput(! 2.5 Rand .4 mul add .866 Rand .4 mul add ){\Cnode*{C2}}
      \rput(! 3.5 Rand .4 mul add .866 Rand .4 mul add ){\Cnode*{D2}}
      \rput(! 4.5 Rand .4 mul add .866 Rand .4 mul add ){\Cnode*{E2}}

      \psline(A1)(A2)
      \psline(B1)(B2)
      \psline(C1)(C2)
      \psline(E1)(C2)
      \psline(E1)(E2)

      \psline(B1)(A2)
      \psline(C1)(B2)
      \psline(D1)(C2)
      \psline(E1)(D2)
      \psline(F1)(E2)

      \psline(A2)(B2)
      \psline(B2)(C2)
      \psline(C2)(D2)
      \psline(D2)(E2)

      \rput(! 0 Rand .4 mul add 1.732 Rand .4 mul add ){\Cnode*{A3}}
      \rput(! 1 Rand .4 mul add 1.732 Rand .4 mul add ){\Cnode*{B3}}
      \rput(! 2 Rand .4 mul add 1.732 Rand .4 mul add ){\Cnode*{C3}}
      \rput(! 3 Rand .4 mul add 1.732 Rand .4 mul add ){\Cnode*{D3}}
      \rput(! 4 Rand .4 mul add 1.732 Rand .4 mul add ){\Cnode*{E3}}
      \rput(! 5 Rand .4 mul add 1.732 Rand .4 mul add ){\Cnode*{F3}}

     \psline(A3)(A2)
      \psline(B3)(B2)
      \psline(C3)(C2)

      \psline(D3)(D2)
      \psline(E3)(E2)

      \psline(B3)(A2)
      \psline(C3)(B2)
      \psline(D3)(C2)
      \psline(E3)(D2)
      \psline(F3)(E2)

      \psline(A3)(B3)
      \psline(B3)(C3)
      \psline(C3)(D3)
      \psline(D3)(E3)
      \psline(E3)(F3)
      
      \rput(! 0.5 Rand .4 mul add 2.598 Rand .4 mul add ){\Cnode*{A4}}
      \rput(! 1.5 Rand .4 mul add 2.598 Rand .4 mul add ){\Cnode*{B4}}
      \rput(! 2.5 Rand .4 mul add 2.598 Rand .4 mul add ){\Cnode*{C4}}
      \rput(! 3.5 Rand .4 mul add 2.598 Rand .4 mul add ){\Cnode*{D4}}
      \rput(! 4.5 Rand .4 mul add 2.598 Rand .4 mul add ){\Cnode*{E4}}

     \psline(A3)(A4)
      \psline(B3)(B4)
      \psline(C3)(C4)
      \psline(D3)(D4)
      \psline(E3)(E4)

      \psline(B3)(A4)
      \psline(C3)(B4)
      \psline(D3)(C4)
      \psline(E3)(D4)
      \psline(F3)(E4)
      
      \psline(A4)(B4)
      \psline(B4)(C4)
      \psline(C4)(D4)
      \psline(D4)(E4)
      
      \rput(! 0 Rand .4 mul add 3.464 Rand .4 mul add ){\Cnode*{A5}}
      \rput(! 1 Rand .4 mul add 3.464 Rand .4 mul add ){\Cnode*{B5}}
      \rput(! 2 Rand .4 mul add 3.464 Rand .4 mul add ){\Cnode*{C5}}
      \rput(! 3 Rand .4 mul add 3.464 Rand .4 mul add ){\Cnode*{D5}}
      \rput(! 4 Rand .4 mul add 3.464 Rand .4 mul add ){\Cnode*{E5}}
      \rput(! 5 Rand .4 mul add 3.464 Rand .4 mul add ){\Cnode*{F5}}

    \psline(A5)(A4)
      \psline(B5)(B4)
      \psline(C5)(C4)
      \psline(D5)(D4)
      \psline(E5)(E4)

      \psline(B5)(A4)
      \psline(C5)(B4)
      \psline(D5)(C4)
      \psline(E5)(D4)
      \psline(F5)(E4)
      
      \psline(A5)(B5)
      \psline(B5)(C5)
      \psline(C5)(D5)
      \psline(D5)(E5)
      \psline(E5)(F5)

      \psset{linewidth=1pt}

      \pscircle(B3){.1}
      \pscircle(B3){.2}

      \pscircle(C4){.1}
      
      \pscircle(C3){.1}
      \pscircle(B2){.1}

      \pscircle(D3){.1}
      \pscircle(E3){.1}
      \pscircle(D4){.1}

      \psset{PointName=none,PointSymbol=o,PointNameSep=3pt}

      \pstCGravABC[PointSymbol=square*]{C3}{B2}{B3}{m0}
      
      \pstCGravABC{A2}{B2}{B3}{c1}
      \pstCGravABC{A2}{A3}{B3}{c2}
      \pstCGravABC{A4}{A3}{B3}{c3}
      \pstCGravABC{A4}{B4}{B3}{c4}
      \pstCGravABC{C3}{B4}{B3}{c5}
      \pstCGravABC{C3}{B4}{C4}{c6}
      \pstCGravABC{C5}{B4}{C4}{c7}
      \pstCGravABC{C5}{D5}{C4}{c8}
      \pstCGravABC{D4}{D5}{C4}{c9}
      \pstCGravABC{D4}{D5}{E5}{c10}
      \pstCGravABC{D4}{E4}{E5}{c11}
      \pstCGravABC{D4}{E4}{E3}{c12}
      \pstCGravABC{F3}{E4}{E3}{c13}
      \pstCGravABC{F3}{E2}{E3}{c14}
      \pstCGravABC{D2}{E2}{E3}{c15}
      \pstCGravABC{D2}{D3}{E3}{c16}
      \pstCGravABC{D2}{D3}{C2}{c17}
      \pstCGravABC{C3}{D3}{C2}{c18}
      \pstCGravABC{C3}{B2}{C2}{c19}
      \pstCGravABC{C1}{B2}{C2}{c20}
      \pstCGravABC{C1}{B2}{B1}{c21}
      \pstCGravABC{A2}{B2}{B1}{c22}
      \pstCGravABC{A2}{B2}{B3}{c23}

      \psccurve(c1)(c2)(c3)(c4)(c5)(c6)(c7)(c8)(c9)(c10)(c11)(c12)(c13)(c14)(c15)(c16)(c17)(c18)(c19)(c20)(c21)(c22)(c23)

      \psline[linestyle=dashed,linewidth=.5pt,dash=2pt 1pt](c5)(m0)(c1)
      
      \end{pspicture}
\end{pdfpic}
\end{center}
\caption{\emph{Pulling across $\kkk$}.   \label{f.topology} The triangulation $G$ is drawn.  The 6 circled vertices correspond to faces of $G^*$ which belong to $\kkk$.    The curve drawn is the cycle $C$; the faces $F^0,\dots,F^{n-1}$ of $G$ are marked by $\circ$'s, and the $E_i$'s (there is only one here) is marked by $\blacksquare$.  The twice circled vertex of $G$ corresponds to the face $C''$ of $G^*$.  }
\end{figure}
Note that if $\abs{\kkk}=1$, then $s_x$ is indeed covered by $\ttt$, since then all vertices in $\ell(C)$ belong to a single tile, which is simply connected and would thus cover $s_x$.  Moreover, any element of $\kkk$ whose intersection with $C$ is discontiguous is a cut-vertex $\{f\in \kkk\mid C\cap f\neq \varnothing\}$, when this collection is viewed as a subgraph of the original graph $G$.   (This occus for the third leftmost member of $\kkk$ in \fref{topology}, for example.) Since not every vertex in a finite graph can be a cut-vertex, we may therefore assume without loss of generality that there is face $C''$ of $G^*$ which is a member of $\kkk$ and whose intersection with $C$ \emph{is} contiguous; in other words, whose boundary cycle is
\[
F^{n},F^{n-1},\dots,F^{\ell},E_1,\dots,E_{t},F^0\quad(t>0)
\]
where no $E_i$ lies on $C$.  We now consider the closed walk $\ell(C')$ in $\Ga$, where $C'$ is the cycle 
\[
F^\ell,F^{\ell-1},F^0,E_t,\dots,E_1,F^{\ell}
\]
in $G^*$.  Note that the region bounded by $C'$ has exactly one fewer face than that bounded by $C$. We will show that if $\ell(C)$ enclosed $s_x$, then so must $\ell(C')$, contradicting minimality of $C$.  

To begin, note that the face $C''$ of the dual graph $G^*$ corresponds to some vertex $T\in \ttt$, which completely contains the walk $\ell(C'')$.  Thus, $\ell(C'')$ does not enclose $s_x$ unless $s_x$ lies in $T$.  Finally, the winding number about any square $s_x$ of the loop $\ell(C')$ is the same as the winding number about the square $s_x$ of the concatenation of $\ell(C)$ with $\ell(C'')$.  In particular, since $s_x$ is not enclosed by $\ell(C')$, the winding number of $\ell(C)$ and $\ell(C'')$ about $s_x$ must be equal.  Thus $\ell(C'')$ also encloses $s_x$, contradicting the minimality of $C$ with this property.

\bigskip

Having shown that $\ttt$ is a tiling, we next wish to show that if $\foot(T_1)\cap \foot(T_2)\neq \varnothing$, then $\foot(T_1)\cap \foot(T_2)$ is a path in $\Ga$ from $\rho(F)$ to $\rho(F')$ for the faces $F,F'$ whose intersection is the pair $\{T_1,T_2\}$.  We begin by showing that the intersection is a path.

If it is not, there are paths $P_1\sbs \partial T_1$ and $P_2\sbs \partial T_2$ with the same pair of endpoints, whose concatenation $C$ is a cycle enclosing a region $S^*$ of ${\Ga}^*$ disjoint from $T_1^*$ and $T_2^*$.  (The dual $T^*$ of $T$ consists of the points of the dual of $\Z^2$ which are centers of squares $s_x\sbs T$.)  Among all possible pairs $T_1,T_2$, we may assume we have chosen such that $\abs {S^*}$ is as large as possible (note that there is some absolute bound on $\abs {S^*}$, since, for example, \hyref{periodic} implies that tiles have bounded size).

We let $\ttt_S$ denote the tiles in the region bounded by $C$.  Since $(\ttt,\eee)$ is 3-connected, there must be at least 3 tiles in $\ttt\stm \ttt_S$ which are adjacent to tiles in $\ttt_S$.  By \hyref{ints}, such a tile $T$ must have the property that it shares two vertices with some tile in $\ttt_S$; however, the only candidate points to be shared between a tile $T\notin \ttt_S$, $T\neq T_1,T_2$, and a tile in $\ttt_S$ are the 2 common endpoints of $P_1$ and $P_2$, and for $C$ to be a (simple) cycle there is, for each of these two endpoints, at most one square of $\Ga$ containing the point and lying outside $C$ and outside of the tiles $T_1,T_2$.  In particular, there must be \emph{exactly} three tiles in $\ttt\stm \ttt_S$ adjacent to tiles in $\ttt_S$; namely, $T_1,T_2,$ and a third tile $T_3$ which includes both endpoints of the paths $P_1,P_2$.  But now either the pair $\{T_3,T_2\}$ or the pair $\{T_3,T_1\}$ contradict the maximality of the choice of $S^*$.  Thus $\foot(T_1)\cap \foot(T_2)$ is indeed a path.

$\rho(F)$ and $\rho(F')$ both lie in $\foot(T_1)\cap \foot(T_2)$.  If they are not the endpoints of the path, then, without loss of generality, let $e_1,e_2$ be the two edges of the path $\foot(T_1)\cap \foot(T_2)$ which are incident with $\rho(F)$.  Since $T_1$ and $T_2$ are nonoverlapping, for each $i=1,2$, $e_i$ lies in one square from $T_1$ and one square of $T_2$.  Moreover, their shared endpoint $\rho(F)$ lies also in a third tile $T_3$ (which again, is nonoverlapping with $T_1,T_2$).  In particular, either $T_1$ or $T_2$ must contain two diagonally opposite squares about $\rho(F)$ without containing the other two squares, contradicting the definition of a tile.

\bigskip
Finally, we wish to show that if $\foot(T_1) \cap \foot(T_2) \neq \emptyset$, then $\{ T_1, T_2 \} \in \mathcal{E}$.  We do this by giving a suitable plane drawing of the graph $G=(\ttt,\eee)$.  For each tile $T\in \ttt$, we draw a vertex $v_T$ corresponding to $T$ at some point of the dual $T^*$.   Identifying $\C$ with the Euclidean plane, let $\bar s_x=\{x+s+t\I\st 0\leq s,t\leq 1\}\sbs \C$.  Since we know that tiles $T,T'\in \ttt$ which are adjacent in $G$ must share an edge of $\Ga$, we can draw a curve from $v_T$ to $v_{T'}$ such that every point in the curve lies in the interior of $\bar s_x\cup \bar s_y$ where $s_x$, $s_y$ each lie in $T$ or $T'$; in particular, we can draw all edges of $G$ such that they are pairwise nonintersecting (except at shared endpoints) and such that the edge from $v_T$ to $v_{T'}$ is disjoint from any $\bar s_x$ for an $s_x$ not contained in $T$ or $T'$.   With this drawing the curve $C_T$ in $\C$ corresponding to the cycle through the neighbors of a tile $T$ is disjoint from $T$, and bounds a region containing $T$.  Since $C_T$ and $C_{T'}$ bound disjoint regions of $\C$ when $T,T'$ are nonadjacent in $G$, we see that any nonadjacent tiles are nonintersecting.
\end{proof}

\section{Tiles}
\label{s.tiles}

In this section, we associate to each circle $C \in {\B}$ a tile (unique up to translation) which will be $90^\circ$ symmetric and tile the plane under translation by the lattice $\Lambda_C$.  
Before beginning our construction, we need a few additional definitions regarding tiles.  We let $\cent(T)$ and $|T|$ denote the {\em centroid} and {\em area} of a tile $T$, which are, respectively, the centroid and area of the real subset $I(T)$ from \eref{disc}.  We say that $T_1$ and $T_2$ \emph{touch} if they are non-overlapping and $\partial T_1\cap \partial T_2$ is a simple path of $\Z^2$ with at least two vertices.  We say that three tiles form a \emph{touching triple} of tiles if they are pairwise touching, and share exactly one common boundary vertex.

Recall that a tile is a set of squares $s_x$, and that the \emph{footprint} of a tile $T$ is
\[
\foot(T):=\bigcup_{s_x\in T} s_x.
\]

It will be convenient to allow a degenerate case of our tile definition.  Note that if $T=\varnothing$, then the centroid $\cent(T)$ would be undefined.  We will allow tiles $T=\varnothing$, whose footprints $\foot(T)$ may be any singleton from $\Ga$; this singleton then gives the centroid of $T$.  In particular, note that the degenerate tiles are in bijective correspondence with $\Ga$.   If $T$ is degenerate, we say that $T,T'$ touch if $\foot(T)\sbs \partial T'$.  We emphasize that $T'\stm T=T'$ whenever $T$ is degenerate.

Let a \emph{prototile} $T$ be a set of squares $s_x\sbs \Ga$, and, if empty, have $\foot(T)$ assigned as a singleton in $\Ga$ (compared with the definition of a tile, we are dropping the requirement that $I(T)$ is a topological disk).  We begin by recursively associating a prototile to each $C\in \B$; in \lref{tiling} we will verify that these prototiles are in fact tiles.  

\begin{definition}
\label{d.T0}
If $(C_0,C_1,C_2,C_3) \in {\B}^4$ is a proper Descartes quadruple, then a set $T_0$ of squares $s_x\sbs \Ga$ is a \emph{prototile} for $C_0$ if $T_0$ has the \emph{tile decomposition} 
\begin{equation}
\label{e.tdecomp}
T_0 = T_1^+ \cup T_1^- \cup T_2^+ \cup T_2^- \cup T_3^+ \cup T_3^-,
\end{equation}
with $\foot(T_i^\pm)\sbs \foot(T_0)$ even if $T_i$ is degenerate, where, for each rotation $(i,j,k)$ of $(1,2,3)$, $T_i^\pm$ is a prototile of $C_i$ satisfying
\begin{equation}
\label{e.tdecompconst}
\cent(T_i^\pm) - \cent(T_0)= \pm \tfrac{1}{2} (v_{kj} - \I v_{kj}).
\end{equation}
where $v_{ij}:=v(C_i,C_j)$.

The base cases are those circles in $\B$ which are not the first circle of any proper Descartes quadruple:  $T_0$ is a prototile for $C_0=(0,\pm 1)$ if $T_0=\varnothing$ and $\foot(T_0)=\{x\}$ for any $x\in \Ga$, while $T_0$ is a prototile for $C_0=(1,1+2z)$ if $T_0=\{s_x\}$, for any $x,z\in \Ga$.
\end{definition}

Note that by induction, any circle in $\B$ has at most one prototile up to translation, and any circle's prototiles must necessarily be $180^\circ$ symmetric.   An example of a decomposition as \eref{tdecomp} can be seen in large tile in the center of \fref{righttile}.

When a prototile $T$ for a circle $C$ satisfies the definition of a tile (i.e., $I(T)$ is a topological disk), we say that $T$ is a tile for $C$.  Given a tile $T_0$ for a circle $C_0\in \B$ with $c_0>1$, we say $T$ is a \emph{subtile} of $T_0$ if $T$ is one of the tiles in the decomposition \eref{tdecomp} for $T_0$.  In the proof of \lref{tiling}, below, we will see that the decomposition of $T$ into subtiles is nonoverlapping, except for prescribed overlap between the largest pair of subtiles.

Our construction in \sref{ford} recursively assigns tiles to each Ford circle with decompositions 
\[
T_0=T_1^+\cup T_1^-\cup T_2^+\cup T_2^-
\]
where the $T_i^\pm$'s were constructions for the two Ford parents of $C_0$.   As a general circle in the Apollonian packing, the third parent of a Ford circle is a line $(0,-1)$ with the degenerate tile $\varnothing$; thus to see that $T_0$ can be realized as a tile for $C_0$ via \dref{T0}, it is only necessary to check, via \pref{fordlattice}, that assigning $\foot(T_3^\pm)=\cent(T_0)\pm \tfrac 1 2 (v_{12}-\I v_{12})$ (from \eref{tdecompconst}) gives that $\foot(T_3^\pm)\sbs \foot(T_0)$.

The presence of the coefficient $\tfrac 1 2$ in \eref{tdecompconst} means that even the existence of prototiles for general circles is not quite immediate.

\begin{lemma}
\label{l.prototile}
There is a prototile $T_0$ for every $C_0\in \B$.
\end{lemma}
\begin{proof}
If $C_0$ is not equivalent to a Ford circle under the symmetries discussed in \sref{lattice-sym}, then it is a member of a proper Descartes quadruple $(C_0,C_1,C_2,C_3)$ for $C_i\neq (0,\pm 1)$ for each $i$, and $(C_1,C_4,C_2,C_3)$ is also a proper Descartes quadruple, where $C_4=2(C_1+C_2+C_3)-C_0$ is the Soddy precursor of $C_0$.  By induction, $C_1$ has a prototile $T_1$ falling into one two cases.  Either, it is a translate of $T_1=\{s_0\}$, in which case a direct application of \dref{T0} verifies that $T_0=\{s_{-1-\I},s_{-1},s_{-\I},s_0\}$  is a prototile for $C_0$.  Otherwise it is given as
\[
T_1 = S_4^+ \cup S_4^- \cup S_2^+ \cup S_2^- \cup  S_3^+ \cup S_3^-,
\]
where each $S_i^\pm$ is a prototile for $C_i$ such that
\begin{equation*}
\cent(S_i^\pm) - \cent(T_1) = \pm \tfrac{1}{2} (v_{kj} - \I v_{kj}),
\end{equation*}
for all rotations $(i,j,k)$ of $(4,2,3)$.   In this latter case, we set 
\begin{equation*}
\begin{aligned}
T_1^+ & = T_1 + v_{32} - \I v_{32} & T_2^+ & = S_2^+ + v_{32} & T_3^+ & = S_3^+ - \I v_{32} \\
T_1^- & = T_1 & T_2^- & = S_2^- - \I v_{32}  & T_3^- & = S_3^- + v_{32}.
\end{aligned}
\end{equation*}
See \fref{maketile} for an illustration.  The key point is that $v_{32}\in \Z[\I]$, and thus that each $T_i^\pm$ is a prototile for $C_i$, allowing us to define
\[
T_0:= T_1^+ \cup T_1^- \cup T_2^+ \cup T_2^- \cup T_3^+ \cup T_3^-.
\]
By the lattice rules (\lref{lattice}), we compute that
\[
\cent(T_i^\pm)-p_0=\pm \tfrac{1}{2} (v_{kj} - \I v_{kj}),
\]
if $(i,j,k)$ is a rotation of $(1,2,3)$, and $p_0:=\cent(T_1^-)+\tfrac 1 2 (v_{32} - \I v_{32})$.  Since each $T_i^\pm$ is $180^\circ$ symmetric, we must have $p_0=\cent(T_0).$  Thus $T_0$ is a prototile for $C_0$ according to \dref{T0}.  Note that the condition that $\foot(T_i^\pm)\sbs T_0$ cannot fail, since $C_i\neq (0,\pm 1)$ for any $i$ implies that no $T_i^\pm$ is degenerate.
\end{proof}

\renewcommand{\topfraction}{.8}
\begin{figure}[t!]
\begin{center}
\nofig{\input{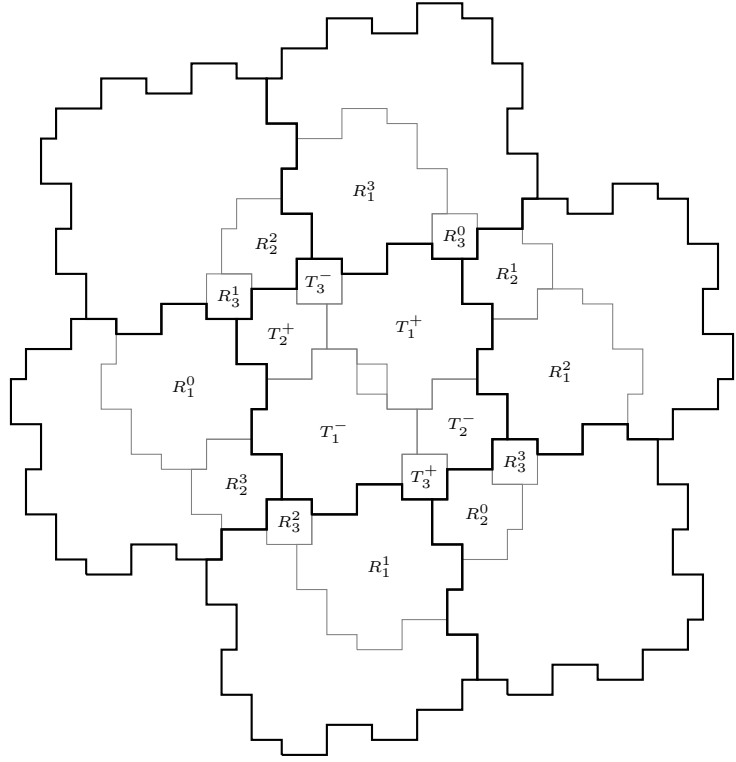}}
\end{center}
\caption{Verifying a tiling by decomposing into parent tiles.  The large tile in the center is $T_0$.} 
\label{f.righttile}
\end{figure}

\lref{tiling}, below, is the heart of our inductive construction, showing that prototiles for circles in $\B$ are in fact $90^\circ$ symmetric tiles.  \lref{tiling} concerns a proper Descartes quadruple $(C_0,C_1,C_2,C_3)$, writing $C_i=(c_i,z_i)$ and $v_{ij}=v(C_i,C_j)$, and where $T_i$ denotes some prototile for $C_i$.  For simplicity of notation, we assume $c_1\geq c_2\geq c_3$, using the symmetries discussed in \sref{lattice-sym}.    A major goal of the lemma is to understand the relationship between $T_0$ and the tiles $T_0 \pm v_{0i}$, for $i = 1, 2, 3$.  We do this by decomposing tiles, to deduce that tiles form a touching triple as a consequence of the fact that some smaller tiles form a touching triple.   In particular, we write $T_0$ as a union of $T_i^\pm$ as in \eref{tdecomp} and then define

\begin{equation}
\label{e.leftdecomp}
\begin{aligned}
R_1^0 & = T_1^+ + v_{03} \\
R_1^1 & = T_1^+ - v_{02} \\
R_1^2 & = T_1^- - v_{03} \\
R_1^3 & = T_1^- + v_{02}
\end{aligned}
\qquad
\begin{aligned}
R_2^0 & = T_2^+ + v_{01} \\
R_2^1 & = T_2^+ - v_{03} \\
R_2^2 & = T_2^- - v_{01} \\
R_2^3 & = T_2^- + v_{03}
\end{aligned}
\qquad
\begin{aligned}
R_3^0 & = T_3^+ + v_{02} \\
R_3^1 & = T_3^+ - v_{01} \\
R_3^2 & = T_3^- - v_{02} \\
R_3^3 & = T_3^- + v_{01}.
\end{aligned}
\end{equation}
(See \fref{righttile}.) When $c_1>1$, \dref{T0} implies that the subtiles $T_1^\pm$ have decompositions as in \eref{tdecomp}, and we write
\begin{align*}
T_1^+&=Q_2^+\cup Q_2^-\cup Q_3^+\cup Q_3^-\cup Q_4^+\cup Q_4^-,\\
T_1^-&=S_2^+\cup S_2^-\cup S_3^+\cup S_3^-\cup S_4^+\cup S_4^-,
\end{align*}
where $Q_i^\pm$ and $S_i^\pm$ are each tiles for $C_i$, and $(C_1,C_4,C_2,C_3)$ is the proper Descartes quadruple given by letting $C_4=2(C_1+C_2+C_3)-C_0\in \B$ be the Soddy precursor of $C_0$.  The \emph{double decomposition} of $T_0$ is then the collection of tiles 
\[
\{T_2^+,T_2^-,T_3^+,T_3^-\}\cup \{S_i^\pm\st i=2,3,4\}\cup \{Q_i^\pm\st i=2,3,4\}
\]
shown in \fref{doubledecomp}.  Note that in the course of proving \lref{tiling}, we will show that $Q_4^-=S_4^+$.

\begin{lemma}
\label{l.tiling}
Let $(C_0,C_1,C_2,C_3)$ be a Descartes quadruple, let $(i,j,k)$ indicate any rotation of $(1,2,3)$, and let $v_{ij} = v(C_i, C_j)$.   Denoting the tiles of the decomposition and double decomposition of $T_0$ as above, the following properties hold:
\begin{enumerate}
\renewcommand{\theenumi}{T\arabic{enumi}}
\item \label{e.tarea}
$T_0$ is a tile, 
and $|T_0| = c_0$.
Moreover, $T_0\stm T_i^\pm$ is a tile, which touches $T_i^\pm$, for each $i=1,2,3$.
\item \label{e.t90}
$T_0$ is $90^\circ$ symmetric.

\item \label{e.ttile}
$T_0$, $T_0 + v_{i0},$ and $T_0 - v_{j0}$ form a touching triple, provided $c_0>1$.

\item \label{e.ttouch1}
$T_0,$ $T_0 + v_{i0},$ and $T_i$ form a touching triple whenever $\cent(T_i)-\cent(T_0)= \tfrac{1}{2} (v_{i0} + v_{0i})$ and $c_0>1$.

\item \label{e.ttouch2}
$T_0,$ $T_i$, and $T_j$ form a touching triple whenever $\cent(T_i)-\cent(T_0)= \tfrac{1}{2} (v_{i0} + v_{0i}),$ and $\cent(T_j)- \cent(T_0)= \tfrac{1}{2} (v_{j0} - v_{0j})$.

\item\label{e.tgooddecomp}
If $c_1\geq c_2>1$, then among other labeled tiles from \fref{righttile}, the subtile $T_i^\pm$ intersects only those which are drawn adjacent to it or with overlap. Moreover, 
$\foot(T_1^+)\cap \foot(T_2^+)\sbs \foot(Q_2^+)\cap \foot(T_2^+)$, and $\foot(T_1^+)\cap \foot(T_1^-)\sbs (\foot(Q_2^+)\cap \foot(S_3^-))\cup \foot(Q_4^-)\cup (\foot(Q_3^+)\cap \foot(S_2^-)).$
\end{enumerate}
\end{lemma}

Note that \eref{tgooddecomp} could be removed from \lref{tiling} without compromising the induction, but this technical information will be necessary for our use of tiles in the construction of integer superharmonic representatives.

Before commencing with our inductive proof of \lref{tiling} in the general case, we use it to prove the following version of the tiling theorem from the introduction. This will give a very simple example of an application of \lref{topological}.

\begin{theorem}
\label{t.Lambdatiling}
For every circle $C\in \B$, 
there is a tile $T_C \sbs \Z^2$ with $90^\circ$ rotational symmetry, such that $T_C+\Lambda_C$ is a tiling.  Moreover, except when $C$ has radius 1, each tile in $T_C+\Lambda_C$ borders exactly 6 other tiles.
\end{theorem}
Note that \tref{tiling} follows from this and our confirmation in \tref{lattice2} that $\Lambda_C=L_C$.
\begin{proof}[Proof of \tref{Lambdatiling}]
In the case where $C\in \B$ has curvature 1, the lattice $\Lambda_C$ is \change{$\Z^2$} 
and the corresponding tile $T$ is simply a single square $s_x$.  Thus we let $T=T_0$ be the tile for a circle $C_0$ in a Descartes quadruple $(C_0,C_1,C_2,C_3)$ with $c_0>1$, and write $v_{ij}$ for the vectors $v(C_i,C_j)$.  We consider a $\Z^2$-periodic planar graph $G$ whose vertices correspond to the tiles $T_0+\Lambda_{C_0}$, where two vertices are adjacent if the corresponding tiles differ by a vector $\pm v_{0i}$.  \eref{ttile} and the area condition from \eref{tarea} now allow us to verify all the hypotheses of \lref{topological}.  \lref{topological} implies that each square of $\Z^2$ lies in exactly one tile among $T_0+\Lambda_{C_0}$ (indeed, $T_0+\Lambda_{C_0}$ is a tiling) and that tile intersections are in one-to-one correspondence to edges in $G$, which has degree 6.
\end{proof}

\begin{figure}[t!]
\centering
\begin{minipage}[t]{.49\textwidth}
\centering{\begin{tikzpicture}[scale=1/5,font=\tiny]
\begin{scope}[draw=gray]
\draw (27.5,27.5) -- (28.5,27.5) -- (29.5,27.5) -- (29.5,28.5) -- (29.5,29.5) -- (28.5,29.5) -- (27.5,29.5) -- (27.5,28.5) -- (27.5,27.5);
\draw (28.500000,28.500000) node {$S_4^+$};
\draw (22.5,22.5) -- (23.5,22.5) -- (24.5,22.5) -- (25.5,22.5) -- (25.5,23.5) -- (25.5,24.5) -- (26.5,24.5) -- (26.5,25.5) -- (26.5,26.5) -- (26.5,27.5) -- (25.5,27.5) -- (24.5,27.5) -- (24.5,28.5) -- (23.5,28.5) -- (22.5,28.5) -- (21.5,28.5) -- (21.5,27.5) -- (21.5,26.5) -- (20.5,26.5) -- (20.5,25.5) -- (20.5,24.5) -- (20.5,23.5) -- (21.5,23.5) -- (22.5,23.5) -- (22.5,22.5);
\draw (23.500000,25.500000) node {$S_2^+$};
\draw (24.5,19.5) -- (25.5,19.5) -- (26.5,19.5) -- (27.5,19.5) -- (27.5,20.5) -- (27.5,21.5) -- (27.5,22.5) -- (26.5,22.5) -- (25.5,22.5) -- (24.5,22.5) -- (24.5,21.5) -- (24.5,20.5) -- (24.5,19.5);
\draw (26.000000,21.000000) node {$S_3^+$};
\draw (22.5,20.5) -- (23.5,20.5) -- (24.5,20.5) -- (24.5,21.5) -- (24.5,22.5) -- (23.5,22.5) -- (22.5,22.5) -- (22.5,21.5) -- (22.5,20.5);
\draw (23.500000,21.500000) node {$S_4^-$};
\draw (27.5,21.5) -- (28.5,21.5) -- (29.5,21.5) -- (30.5,21.5) -- (30.5,22.5) -- (30.5,23.5) -- (31.5,23.5) -- (31.5,24.5) -- (31.5,25.5) -- (31.5,26.5) -- (30.5,26.5) -- (29.5,26.5) -- (29.5,27.5) -- (28.5,27.5) -- (27.5,27.5) -- (26.5,27.5) -- (26.5,26.5) -- (26.5,25.5) -- (25.5,25.5) -- (25.5,24.5) -- (25.5,23.5) -- (25.5,22.5) -- (26.5,22.5) -- (27.5,22.5) -- (27.5,21.5);
\draw (28.500000,24.500000) node {$S_2^-$};
\draw (24.5,27.5) -- (25.5,27.5) -- (26.5,27.5) -- (27.5,27.5) -- (27.5,28.5) -- (27.5,29.5) -- (27.5,30.5) -- (26.5,30.5) -- (25.5,30.5) -- (24.5,30.5) -- (24.5,29.5) -- (24.5,28.5) -- (24.5,27.5);
\draw (26.000000,29.000000) node {$S_3^-$};
\end{scope}
\begin{scope}[thick]
\draw (29.5,26.5) -- (30.5,26.5) -- (31.5,26.5) -- (32.5,26.5) -- (32.5,27.5) -- (32.5,28.5) -- (33.5,28.5) -- (34.5,28.5) -- (35.5,28.5) -- (35.5,29.5) -- (35.5,30.5) -- (36.5,30.5) -- (36.5,31.5) -- (36.5,32.5) -- (36.5,33.5) -- (35.5,33.5) -- (34.5,33.5) -- (34.5,34.5) -- (34.5,35.5) -- (34.5,36.5) -- (33.5,36.5) -- (32.5,36.5) -- (32.5,37.5) -- (31.5,37.5) -- (30.5,37.5) -- (29.5,37.5) -- (29.5,36.5) -- (29.5,35.5) -- (28.5,35.5) -- (27.5,35.5) -- (26.5,35.5) -- (26.5,34.5) -- (26.5,33.5) -- (25.5,33.5) -- (25.5,32.5) -- (25.5,31.5) -- (25.5,30.5) -- (26.5,30.5) -- (27.5,30.5) -- (27.5,29.5) -- (27.5,28.5) -- (27.5,27.5) -- (28.5,27.5) -- (29.5,27.5) -- (29.5,26.5);
\draw (31.000000,32.000000) node {$T_1^+ \setminus S_4^+$};
\draw (21.5,28.5) -- (22.5,28.5) -- (23.5,28.5) -- (24.5,28.5) -- (24.5,29.5) -- (24.5,30.5) -- (25.5,30.5) -- (25.5,31.5) -- (25.5,32.5) -- (25.5,33.5) -- (24.5,33.5) -- (23.5,33.5) -- (23.5,34.5) -- (22.5,34.5) -- (21.5,34.5) -- (20.5,34.5) -- (20.5,33.5) -- (20.5,32.5) -- (19.5,32.5) -- (19.5,31.5) -- (19.5,30.5) -- (19.5,29.5) -- (20.5,29.5) -- (21.5,29.5) -- (21.5,28.5);
\draw (22.500000,31.500000) node {$T_2^+$};
\draw (30.5,20.5) -- (31.5,20.5) -- (32.5,20.5) -- (33.5,20.5) -- (33.5,21.5) -- (33.5,22.5) -- (33.5,23.5) -- (32.5,23.5) -- (31.5,23.5) -- (30.5,23.5) -- (30.5,22.5) -- (30.5,21.5) -- (30.5,20.5);
\draw (32.000000,22.000000) node {$T_3^+$};
\draw (24.5,19.5) -- (25.5,19.5) -- (26.5,19.5) -- (27.5,19.5) -- (27.5,20.5) -- (27.5,21.5) -- (28.5,21.5) -- (29.5,21.5) -- (30.5,21.5) -- (30.5,22.5) -- (30.5,23.5) -- (31.5,23.5) -- (31.5,24.5) -- (31.5,25.5) -- (31.5,26.5) -- (30.5,26.5) -- (29.5,26.5) -- (29.5,27.5) -- (29.5,28.5) -- (29.5,29.5) -- (28.5,29.5) -- (27.5,29.5) -- (27.5,30.5) -- (26.5,30.5) -- (25.5,30.5) -- (24.5,30.5) -- (24.5,29.5) -- (24.5,28.5) -- (23.5,28.5) -- (22.5,28.5) -- (21.5,28.5) -- (21.5,27.5) -- (21.5,26.5) -- (20.5,26.5) -- (20.5,25.5) -- (20.5,24.5) -- (20.5,23.5) -- (21.5,23.5) -- (22.5,23.5) -- (22.5,22.5) -- (22.5,21.5) -- (22.5,20.5) -- (23.5,20.5) -- (24.5,20.5) -- (24.5,19.5);

\draw (33.5,22.5) -- (34.5,22.5) -- (35.5,22.5) -- (36.5,22.5) -- (36.5,23.5) -- (36.5,24.5) -- (37.5,24.5) -- (37.5,25.5) -- (37.5,26.5) -- (37.5,27.5) -- (36.5,27.5) -- (35.5,27.5) -- (35.5,28.5) -- (34.5,28.5) -- (33.5,28.5) -- (32.5,28.5) -- (32.5,27.5) -- (32.5,26.5) -- (31.5,26.5) -- (31.5,25.5) -- (31.5,24.5) -- (31.5,23.5) -- (32.5,23.5) -- (33.5,23.5) -- (33.5,22.5);
\draw (34.500000,25.500000) node {$T_2^-$};
\draw (23.5,33.5) -- (24.5,33.5) -- (25.5,33.5) -- (26.5,33.5) -- (26.5,34.5) -- (26.5,35.5) -- (26.5,36.5) -- (25.5,36.5) -- (24.5,36.5) -- (23.5,36.5) -- (23.5,35.5) -- (23.5,34.5) -- (23.5,33.5);
\draw (25.000000,35.000000) node {$T_3^-$};
\end{scope}

\end{tikzpicture}}%
\begin{minipage}[c]{1.2\textwidth}
\caption{Constructing a prototile from the decomposition of its largest parent tile.}%
\label{f.maketile}%
\end{minipage}%
\end{minipage}%
\begin{minipage}[t]{.49\textwidth}
\centering
\begin{tikzpicture}[scale=1/5,font=\tiny]
\begin{scope}[draw=gray]
\draw (27.5,27.5) -- (28.5,27.5) -- (29.5,27.5) -- (29.5,28.5) -- (29.5,29.5) -- (28.5,29.5) -- (27.5,29.5) -- (27.5,28.5) -- (27.5,27.5);
\draw (28.500000,28.500000) node {$S_4^+$};
\draw (22.5,22.5) -- (23.5,22.5) -- (24.5,22.5) -- (25.5,22.5) -- (25.5,23.5) -- (25.5,24.5) -- (26.5,24.5) -- (26.5,25.5) -- (26.5,26.5) -- (26.5,27.5) -- (25.5,27.5) -- (24.5,27.5) -- (24.5,28.5) -- (23.5,28.5) -- (22.5,28.5) -- (21.5,28.5) -- (21.5,27.5) -- (21.5,26.5) -- (20.5,26.5) -- (20.5,25.5) -- (20.5,24.5) -- (20.5,23.5) -- (21.5,23.5) -- (22.5,23.5) -- (22.5,22.5);
\draw (23.500000,25.500000) node {$S_2^+$};
\draw (24.5,19.5) -- (25.5,19.5) -- (26.5,19.5) -- (27.5,19.5) -- (27.5,20.5) -- (27.5,21.5) -- (27.5,22.5) -- (26.5,22.5) -- (25.5,22.5) -- (24.5,22.5) -- (24.5,21.5) -- (24.5,20.5) -- (24.5,19.5);
\draw (26.000000,21.000000) node {$S_3^+$};
\draw (22.5,20.5) -- (23.5,20.5) -- (24.5,20.5) -- (24.5,21.5) -- (24.5,22.5) -- (23.5,22.5) -- (22.5,22.5) -- (22.5,21.5) -- (22.5,20.5);
\draw (23.500000,21.500000) node {$S_4^-$};
\draw (27.5,21.5) -- (28.5,21.5) -- (29.5,21.5) -- (30.5,21.5) -- (30.5,22.5) -- (30.5,23.5) -- (31.5,23.5) -- (31.5,24.5) -- (31.5,25.5) -- (31.5,26.5) -- (30.5,26.5) -- (29.5,26.5) -- (29.5,27.5) -- (28.5,27.5) -- (27.5,27.5) -- (26.5,27.5) -- (26.5,26.5) -- (26.5,25.5) -- (25.5,25.5) -- (25.5,24.5) -- (25.5,23.5) -- (25.5,22.5) -- (26.5,22.5) -- (27.5,22.5) -- (27.5,21.5);
\draw (28.500000,24.500000) node {$S_2^-$};
\draw (24.5,27.5) -- (25.5,27.5) -- (26.5,27.5) -- (27.5,27.5) -- (27.5,28.5) -- (27.5,29.5) -- (27.5,30.5) -- (26.5,30.5) -- (25.5,30.5) -- (24.5,30.5) -- (24.5,29.5) -- (24.5,28.5) -- (24.5,27.5);
\draw (26.000000,29.000000) node {$S_3^-$};
\draw (32.5,34.5) -- (33.5,34.5) -- (34.5,34.5) -- (34.5,35.5) -- (34.5,36.5) -- (33.5,36.5) -- (32.5,36.5) -- (32.5,35.5) -- (32.5,34.5);
\draw (33.500000,35.500000) node {$Q_4^+$};
\draw (27.5,29.5) -- (28.5,29.5) -- (29.5,29.5) -- (30.5,29.5) -- (30.5,30.5) -- (30.5,31.5) -- (31.5,31.5) -- (31.5,32.5) -- (31.5,33.5) -- (31.5,34.5) -- (30.5,34.5) -- (29.5,34.5) -- (29.5,35.5) -- (28.5,35.5) -- (27.5,35.5) -- (26.5,35.5) -- (26.5,34.5) -- (26.5,33.5) -- (25.5,33.5) -- (25.5,32.5) -- (25.5,31.5) -- (25.5,30.5) -- (26.5,30.5) -- (27.5,30.5) -- (27.5,29.5);
\draw (28.500000,32.500000) node {$Q_2^+$};
\draw (29.5,26.5) -- (30.5,26.5) -- (31.5,26.5) -- (32.5,26.5) -- (32.5,27.5) -- (32.5,28.5) -- (32.5,29.5) -- (31.5,29.5) -- (30.5,29.5) -- (29.5,29.5) -- (29.5,28.5) -- (29.5,27.5) -- (29.5,26.5);
\draw (31.000000,28.000000) node {$Q_3^+$};
\draw (32.5,28.5) -- (33.5,28.5) -- (34.5,28.5) -- (35.5,28.5) -- (35.5,29.5) -- (35.5,30.5) -- (36.5,30.5) -- (36.5,31.5) -- (36.5,32.5) -- (36.5,33.5) -- (35.5,33.5) -- (34.5,33.5) -- (34.5,34.5) -- (33.5,34.5) -- (32.5,34.5) -- (31.5,34.5) -- (31.5,33.5) -- (31.5,32.5) -- (30.5,32.5) -- (30.5,31.5) -- (30.5,30.5) -- (30.5,29.5) -- (31.5,29.5) -- (32.5,29.5) -- (32.5,28.5);
\draw (33.500000,31.500000) node {$Q_2^-$};
\draw (29.5,34.5) -- (30.5,34.5) -- (31.5,34.5) -- (32.5,34.5) -- (32.5,35.5) -- (32.5,36.5) -- (32.5,37.5) -- (31.5,37.5) -- (30.5,37.5) -- (29.5,37.5) -- (29.5,36.5) -- (29.5,35.5) -- (29.5,34.5);
\draw (31.000000,36.000000) node {$Q_3^-$};
\end{scope}\begin{scope}[thick]
\draw (29.5,26.5) -- (30.5,26.5) -- (31.5,26.5) -- (32.5,26.5) -- (32.5,27.5) -- (32.5,28.5) -- (33.5,28.5) -- (34.5,28.5) -- (35.5,28.5) -- (35.5,29.5) -- (35.5,30.5) -- (36.5,30.5) -- (36.5,31.5) -- (36.5,32.5) -- (36.5,33.5) -- (35.5,33.5) -- (34.5,33.5) -- (34.5,34.5) -- (34.5,35.5) -- (34.5,36.5) -- (33.5,36.5) -- (32.5,36.5) -- (32.5,37.5) -- (31.5,37.5) -- (30.5,37.5) -- (29.5,37.5) -- (29.5,36.5) -- (29.5,35.5) -- (28.5,35.5) -- (27.5,35.5) -- (26.5,35.5) -- (26.5,34.5) -- (26.5,33.5) -- (25.5,33.5) -- (25.5,32.5) -- (25.5,31.5) -- (25.5,30.5) -- (26.5,30.5) -- (27.5,30.5) -- (27.5,29.5) -- (27.5,28.5) -- (27.5,27.5) -- (28.5,27.5) -- (29.5,27.5) -- (29.5,26.5);
\draw (21.5,28.5) -- (22.5,28.5) -- (23.5,28.5) -- (24.5,28.5) -- (24.5,29.5) -- (24.5,30.5) -- (25.5,30.5) -- (25.5,31.5) -- (25.5,32.5) -- (25.5,33.5) -- (24.5,33.5) -- (23.5,33.5) -- (23.5,34.5) -- (22.5,34.5) -- (21.5,34.5) -- (20.5,34.5) -- (20.5,33.5) -- (20.5,32.5) -- (19.5,32.5) -- (19.5,31.5) -- (19.5,30.5) -- (19.5,29.5) -- (20.5,29.5) -- (21.5,29.5) -- (21.5,28.5);
\draw (22.500000,31.500000) node {$T_2^+$};
\draw (30.5,20.5) -- (31.5,20.5) -- (32.5,20.5) -- (33.5,20.5) -- (33.5,21.5) -- (33.5,22.5) -- (33.5,23.5) -- (32.5,23.5) -- (31.5,23.5) -- (30.5,23.5) -- (30.5,22.5) -- (30.5,21.5) -- (30.5,20.5);
\draw (32.000000,22.000000) node {$T_3^+$};
\draw (24.5,19.5) -- (25.5,19.5) -- (26.5,19.5) -- (27.5,19.5) -- (27.5,20.5) -- (27.5,21.5) -- (28.5,21.5) -- (29.5,21.5) -- (30.5,21.5) -- (30.5,22.5) -- (30.5,23.5) -- (31.5,23.5) -- (31.5,24.5) -- (31.5,25.5) -- (31.5,26.5) -- (30.5,26.5) -- (29.5,26.5) -- (29.5,27.5) -- (29.5,28.5) -- (29.5,29.5) -- (28.5,29.5) -- (27.5,29.5) -- (27.5,30.5) -- (26.5,30.5) -- (25.5,30.5) -- (24.5,30.5) -- (24.5,29.5) -- (24.5,28.5) -- (23.5,28.5) -- (22.5,28.5) -- (21.5,28.5) -- (21.5,27.5) -- (21.5,26.5) -- (20.5,26.5) -- (20.5,25.5) -- (20.5,24.5) -- (20.5,23.5) -- (21.5,23.5) -- (22.5,23.5) -- (22.5,22.5) -- (22.5,21.5) -- (22.5,20.5) -- (23.5,20.5) -- (24.5,20.5) -- (24.5,19.5);
\draw (33.5,22.5) -- (34.5,22.5) -- (35.5,22.5) -- (36.5,22.5) -- (36.5,23.5) -- (36.5,24.5) -- (37.5,24.5) -- (37.5,25.5) -- (37.5,26.5) -- (37.5,27.5) -- (36.5,27.5) -- (35.5,27.5) -- (35.5,28.5) -- (34.5,28.5) -- (33.5,28.5) -- (32.5,28.5) -- (32.5,27.5) -- (32.5,26.5) -- (31.5,26.5) -- (31.5,25.5) -- (31.5,24.5) -- (31.5,23.5) -- (32.5,23.5) -- (33.5,23.5) -- (33.5,22.5);
\draw (34.500000,25.500000) node {$T_2^-$};
\draw (23.5,33.5) -- (24.5,33.5) -- (25.5,33.5) -- (26.5,33.5) -- (26.5,34.5) -- (26.5,35.5) -- (26.5,36.5) -- (25.5,36.5) -- (24.5,36.5) -- (23.5,36.5) -- (23.5,35.5) -- (23.5,34.5) -- (23.5,33.5);
\draw (25.000000,35.000000) node {$T_3^-$};
\end{scope}
\end{tikzpicture}%
\begin{minipage}[c]{1.2\textwidth}
\caption{The double decomposition of a tile, for the proof of \lref{tiling}.}
\label{f.doubledecomp}
\end{minipage}
\end{minipage}
\end{figure}

The rest of this section is devoted to the proof of \lref{tiling}.  Since we have already addressed the case of Ford circles and diamond circles, we may assume that $(C_0,C_1,C_2,C_3)\in \B$ is a proper Descartes quadruple, such that no $C_i$ is a line, and at most one $C_i$ has curvature $c_i=1$.  In particular, $c_0>4$, and, rotating the parents $(C_1,C_2,C_3)$ and possibly conjugating the original tuple using the symmetries in \sref{lattice-sym}, we may assume $C_3$ is a parent of $C_2$ and $C_2$ is a parent of $C_1$ so that $c_1> c_2\geq 4$.  In particular, we also have that $(C_1,C_4,C_2,C_3)$ and $(C_2,C_3,C_5, C_6)$ are both proper Descartes quadruples for some $C_5, C_6 \in {\B}$ where $C_4$ is the Soddy twin of $C_0$, and by induction, we may assume \lref{tiling} holds for both of these quadruples.   

For the sake of clarity and brevity, Claims \ref{c.outtouch} and \ref{c.doubledecomp} below are stated with the aid of \fref{righttile} and \fref{doubledecomp}, respectively.  For the purposes of the statements of the claims, tiles $S_1,S_2,S_3$ (from among the $T_i^\pm$ and $R_i^j$, or $T_i^\pm$, $Q_i^\pm$ and  $S_i^\pm$ labeled in the figure) are considered to be drawn adjacently if their corresponding regions in the figure overlap or share some portion of their boundaries as drawn. ($T_1^-$ and $T_1^+$ in \fref{righttile} are the only tiles drawn with overlap.)  In the course of our proof of \lref{tiling}, we will be verifying that the basic tile layout in the figure is correct.
\begin{claim}
\label{c.outtouch}
If $(S_1,S_2,S_3)$ is a triple of adjacent labelled tiles in \fref{righttile}, not all of which are contained in $T_0$, then the $S_i$ form a touching triple, unless $\abs{S_1}=\abs{S_2}=\abs{S_3}=1$, in which case the $S_i$'s are pairwise intersecting squares.
\end{claim}

\begin{proof}[Proof of claim]
To prove this claim, we use the lattice rules \eref{lattice} to simplify the differences between centers of adjacent tiles.  One can check that adjacent tiles for $C_i$ are related by $\pm v_{ji}$ for some $j < i$.  One can also check that, when $i \neq j$, adjacent tiles for $C_i$ and $C_j$  are related by $\tfrac{1}{2} \I^s (v_{ij} + v_{ji})$ for $s = 0, 1, 2, 3$.  From this, we see that all of the triples $(S_1, S_2, S_3)$ from the claim fall into one of the inductive versions of \eref{ttouch1} or \eref{ttouch2} for $(C_1,C_4,C_2,C_3)$ or $(C_2,C_3,C_5,C_6)$ so long as $\abs{S_i}\neq 1$ for some $i$.  The case where each $S_i$ is a tile for a circle of curvature $1$ is easily checked by hand.
\end{proof}

To analyze touching tiles within $T_0$, we need to make use of the double decomposition. \fref{doubledecomp}'s depiction of the double decomposition is a bit deceptive, however, as it shows $S_4^\pm$ smaller than $S_3^\pm$ and $S_2^\pm$, even though the size of $S_4^\pm$ relative to $S_2^\pm,S_3^\pm$ is not constrainted by the hypotheses of the Lemma (it depends on the relative size of $C_4$ to $C_2,C_3$).  In particular, this is why the scope of Claim \ref{c.outtouch} is limited to triples of tiles not all lying inside $T_0$.   To emphasize this point, examples of the six generally possible tangency structures (from the six possible relative size orders of the circles $C_1,C_2,C_3$) are illustrated in \fref{sixdecompositions}. 

The following claim is sufficient for our purposes, however, and does not depend on the relative size of $C_4$ to $C_2$ and $C_3$:

\begin{claim}
\label{c.doubledecomp}
$S_4^+=Q_4^-$, and if $(R_1,R_2,R_3)$ is a triple of adjacent labelled tiles in \fref{doubledecomp}, not all of which are contained in a single $T_1^\pm$, then the $R_i$ form a touching triple of tiles.  
\end{claim}
\begin{proof}[Proof of claim]
Again we use the lattice rules \eref{lattice} to simplify the differences between centers of adjacent tiles.  In particular, we verify that $S_4^+=Q_4^-$, and that each triple covered by the claim is a case of \eref{ttouch1} or \eref{ttouch2} for the quadruple $(C_2,C_3,C_5,C_6)$.
\end{proof}

The information from the double decomposition is not enough for us to claim yet that, e.g., $T_1^+$ and $T_2^+$ touch (in particular, that they intersect only on their boundaries), as we have not analyzed the topological relationships among all the tiles in the double decomposition from \fref{doubledecomp}, which would be quite cumbersome in light of the unknown relative size of the circle $C_4$.  To proceed further at this point will require \lref{topological}.   

\medskip

In principle, we would like to apply some topology to rule out the presence of extraneous tile relationships.  However, \lref{topological} is designed to apply to tilings, so the 
overlap of tiles $T_1^+$ and $T_1^-$ precludes its direct use for this purpose.

To work around this, we will modify some of our tiles so that we actually have a tiling.   In particular, we define $\tilde \ttt_L$ by replacing each translate of $T_1^+$ by the corresponding translate of $T_1^+\stm Q_4^-$, which, by induction, is a tile by \eref{tarea}.  We consider the graph $G$ whose vertex set is in correspondence with $\ttt_L$, and, also, therefore, $\tilde \ttt_L$.  For any vertex $v\in V(G)$, we write $T(v)$ for the corresponding tile in $\ttt_L$, and $\tilde T(v)$ for the corresponding tile in $\tilde \ttt_L$. (Note that unless $T(v)$ is a $L$-translate of $T_1^+$, we have that $T(v)=\tilde T(v)$.)  A pair of vertices $\{u,v\}$ is joined by an edge in $G$ if $T(u)$ and $T(v)$ are translates by a common vector in $L$ of a tile pair drawn adjacently or with overlap in \fref{righttile}.  
Viewed as an abstract graph, $G$ is easily seen to have a drawing as a $\Z^2$-periodic planar triangulation.  

We let $\eee$ and $\tilde \eee$ be the sets of pairs $\{T(u),T(v)\}$ and $\{\tilde T(u),\tilde T(v)\}$, respectively, for $\{u,v\}\in E(G)$.
\lref{topological} will allow us to prove the following claim:
\begin{claim}
  \label{c.applytop}
The tiles in $\tilde \ttt$ form a tiling, the only intersecting pairs of tiles in $\tilde\ttt$ are those in $\eee$, and all nonempty intersections of tile-pairs are paths.
\end{claim}

\bigskip

\noindent Assuming the claim, we can now verify \eref{tarea}-\eref{ttouch2} for the quadruple $(C_0,C_1,C_2,C_3)$.

\ipart{\eref{tarea}}{$T_0$ is a tile, 
and $|T_0| = c_0$.
Moreover, $T_0\stm T_i^\pm$ is a tile, which touches $T_i^\pm$, for each $i=1,2,3$.}

\begin{proof}
Since $T_0$ is a set of squares whose centers form a connected subgraph of the dual lattice ${\Z^2}^*$, the fact that $T_0$ is a tile follows from the fact that $T_0+L$ covers the plane without overlap (so $T_0$ has no ``holes''), by \lref{topological}.  Similarly, since any 5 tiles from among
\[
T_1^+\stm Q_4^-, T_1^-, T_2^+, T_2^-, T_3^+, T_3^-
\]
induce a connected subgraph of $(\tilde \ttt, \tilde \eee)$ whose complement is also connected, we get that $T_0\stm T_i^\pm$ is a tile.
It also follows that $\abs{T_0}=c_0$, since $c_0$ is the determinant of $L$.  
\end{proof}

\ipart{\eref{t90}}{$T_0$ is $90^\circ$ symmetric.}
\begin{proof}
Claim \ref{c.outtouch} implies that the tiles $R_i^k$, for $i = 1, 2, 3$ and $k = 1, 2, 3, 4$, surround $T_0$.  By induction, each of the $R_i^k$ is $90^\circ$ symmetric.  Moreover, using the lattice rules \eref{lattice}, we see that the union $S = \cup R_i^k$ is a $90^\circ$ symmetric union of squares.  Since $T_0$ is the bounded component of the complement of $S \setminus \partial S$, it too must be $90^\circ$ symmetric.
\end{proof}

\ipart{\eref{ttile}}{$T_0$, $T_0 + v_{i0},$ and $T_0 - v_{j0}$ form a touching triple, provided $c_0 > 1$.}

\begin{proof}
Our use of \lref{topological} tells us that the only pairs of tiles in $\tilde \ttt$ which intersect are those in $\tilde \eee$.  Thus the pairwise intersections of the tiles $T_0$, $T_0 + v_{0i},$ and $T_0 - v_{0j}$ can be written as the union of the subtiles from the respective tiles which are drawn adjacently in \fref{righttile}.  Each of these intersections is a simple path between points $\rho(F), \rho(F')$ which can be concatenated to a simple path.  (The resulting path is simple since the subtiles are all nonoverlapping.)  We are done by \eref{t90}.
\end{proof}

\ipart{\eref{ttouch1}}{$T_0,$ $T_0 + v_{i0},$ and $T_i$ form a touching triple, if $\cent(T_i)-\cent(T_0)=\tfrac{1}{2} (v_{i0} + v_{0i})$, provided $c_0 > 1$.}

\begin{proof}
Using the lattice rules \eref{lattice}, one can verify for each $i$ in $(1,2,3)$ that $T_i=R_i^1$, as defined in \eref{leftdecomp}.  In particular, as above, \lref{topological} implies that $T_0,$ $T_0+v_{0i}$, and $R_i^1$ form a touching triple.  The statement follows now from \eref{t90}.
\end{proof}

\ipart{\eref{ttouch2}}{$T_0,$ $T_i$, and $T_j$ form a touching triple, if $\cent(T_i)-\cent(T_0)=\tfrac{1}{2} (v_{i0} + v_{0i}),$ and $\cent(T_j)- \cent(T_0)=\tfrac{1}{2} (v_{j0} - v_{0j})$.}

\begin{proof}
Using the lattice rules \eref{lattice}, we can verify that $T_j=R_j^2$
\lref{topological}  gives that $T_0,$ $R_i^1$, and $R_j^2$ form a touching triple.  Thus, since we also had that $T_i=R_i^1$, the statement follows now from \eref{t90}.
\end{proof}

\ipart{\eref{tgooddecomp}}{
If $c_1\geq c_2>1$, then among other labeled tiles from \fref{righttile}, the subtile $T_i^\pm$ intersects only those which are drawn adjacent to it or with overlap. Moreover, 
$\foot(T_1^+)\cap \foot(T_2^+)\sbs \foot(Q_2^+)\cap \foot(T_2^+)$, and $\foot(T_1^+)\cap \foot(T_1^-)\sbs (\foot(Q_2^+)\cap \foot(S_3^-))\cup \foot(Q_4^-)\cup (\foot(Q_3^+)\cap \foot(S_2^-)).$
}

\begin{proof}
These are both consequences of the application of \lref{topological}.  The first follows from the fact that $\foot(T_1)\cap \foot(T_2)$ is a path between $\rho(F)$ and $\rho(F')$ for the faces $F,F'$ which contain $T_1,T_2$.  Similarly, the second follows from the fact that $\foot(T_1^+\stm Q_4^-)\cap \foot(T_1^-)$ is a path in $(\foot(Q_2^+)\cap \foot(S_3^-))\cup \foot(Q_4^-)\cup (\foot(Q_3^+)\cap \foot(S_2^-)).$
\end{proof}

\begin{remark}\label{rem.90}
  An essential feature of \eref{ttile}, \eref{ttouch1}, \eref{ttouch2} is that their proofs make use of 90 degree symmetry.  In particular, each induction step essentially rotates known relationships by 90 degrees, which is then fixed by using \eref{t90}.  The lack of 90 degree symmetry will force a significant increase in the complexity of our odometer construction in comparison with this tile induction.
\end{remark}

\afterpage{
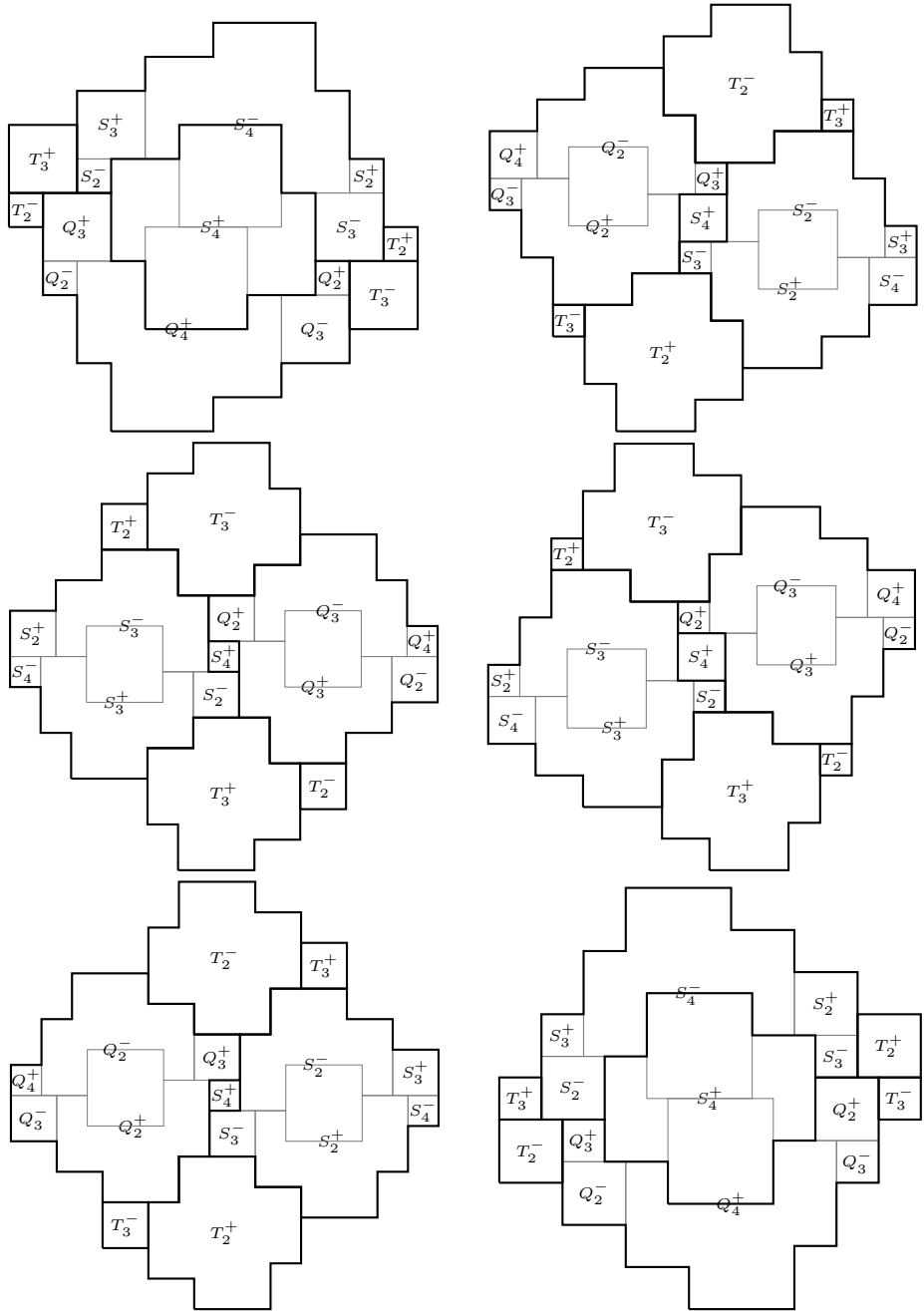
\begin{figure}[t]
  \begin{tabular}{lr}
    \begin{tikzpicture}[x=0.035897\textwidth,y=0.035897\textwidth,font=\tiny]
\begin{scope}[draw=gray]
\draw (17.5,16.5) -- (18.5,16.5) -- (19.5,16.5) -- (20.5,16.5) -- (20.5,17.5) -- (21.5,17.5) -- (22.5,17.5) -- (22.5,18.5) -- (22.5,19.5) -- (22.5,20.5) -- (21.5,20.5) -- (21.5,21.5) -- (21.5,22.5) -- (20.5,22.5) -- (19.5,22.5) -- (18.5,22.5) -- (18.5,21.5) -- (17.5,21.5) -- (16.5,21.5) -- (16.5,20.5) -- (16.5,19.5) -- (16.5,18.5) -- (17.5,18.5) -- (17.5,17.5) -- (17.5,16.5);
\draw (19.500000,19.500000) node {$S_4^+$};
\draw (23.5,20.5) -- (24.5,20.5) -- (24.5,21.5) -- (23.5,21.5) -- (23.5,20.5);
\draw (24.000000,21.000000) node {$S_2^+$};
\draw (15.5,21.5) -- (16.5,21.5) -- (17.5,21.5) -- (17.5,22.5) -- (17.5,23.5) -- (16.5,23.5) -- (15.5,23.5) -- (15.5,22.5) -- (15.5,21.5);
\draw (16.500000,22.500000) node {$S_3^+$};
\draw (18.5,19.5) -- (19.5,19.5) -- (20.5,19.5) -- (21.5,19.5) -- (21.5,20.5) -- (22.5,20.5) -- (23.5,20.5) -- (23.5,21.5) -- (23.5,22.5) -- (23.5,23.5) -- (22.5,23.5) -- (22.5,24.5) -- (22.5,25.5) -- (21.5,25.5) -- (20.5,25.5) -- (19.5,25.5) -- (19.5,24.5) -- (18.5,24.5) -- (17.5,24.5) -- (17.5,23.5) -- (17.5,22.5) -- (17.5,21.5) -- (18.5,21.5) -- (18.5,20.5) -- (18.5,19.5);
\draw (20.500000,22.500000) node {$S_4^-$};
\draw (15.5,20.5) -- (16.5,20.5) -- (16.5,21.5) -- (15.5,21.5) -- (15.5,20.5);
\draw (16.000000,21.000000) node {$S_2^-$};
\draw (22.5,18.5) -- (23.5,18.5) -- (24.5,18.5) -- (24.5,19.5) -- (24.5,20.5) -- (23.5,20.5) -- (22.5,20.5) -- (22.5,19.5) -- (22.5,18.5);
\draw (23.500000,19.500000) node {$S_3^-$};
\draw (16.5,13.5) -- (17.5,13.5) -- (18.5,13.5) -- (19.5,13.5) -- (19.5,14.5) -- (20.5,14.5) -- (21.5,14.5) -- (21.5,15.5) -- (21.5,16.5) -- (21.5,17.5) -- (20.5,17.5) -- (20.5,18.5) -- (20.5,19.5) -- (19.5,19.5) -- (18.5,19.5) -- (17.5,19.5) -- (17.5,18.5) -- (16.5,18.5) -- (15.5,18.5) -- (15.5,17.5) -- (15.5,16.5) -- (15.5,15.5) -- (16.5,15.5) -- (16.5,14.5) -- (16.5,13.5);
\draw (18.500000,16.500000) node {$Q_4^+$};
\draw (22.5,17.5) -- (23.5,17.5) -- (23.5,18.5) -- (22.5,18.5) -- (22.5,17.5);
\draw (23.000000,18.000000) node {$Q_2^+$};
\draw (14.5,18.5) -- (15.5,18.5) -- (16.5,18.5) -- (16.5,19.5) -- (16.5,20.5) -- (15.5,20.5) -- (14.5,20.5) -- (14.5,19.5) -- (14.5,18.5);
\draw (15.500000,19.500000) node {$Q_3^+$};
\draw (14.5,17.5) -- (15.5,17.5) -- (15.5,18.5) -- (14.5,18.5) -- (14.5,17.5);
\draw (15.000000,18.000000) node {$Q_2^-$};
\draw (21.5,15.5) -- (22.5,15.5) -- (23.5,15.5) -- (23.5,16.5) -- (23.5,17.5) -- (22.5,17.5) -- (21.5,17.5) -- (21.5,16.5) -- (21.5,15.5);
\draw (22.500000,16.500000) node {$Q_3^-$};
\end{scope}\begin{scope}[thick]
\draw (16.5,13.5) -- (17.5,13.5) -- (18.5,13.5) -- (19.5,13.5) -- (19.5,14.5) -- (20.5,14.5) -- (21.5,14.5) -- (21.5,15.5) -- (22.5,15.5) -- (23.5,15.5) -- (23.5,16.5) -- (23.5,17.5) -- (23.5,18.5) -- (22.5,18.5) -- (22.5,19.5) -- (22.5,20.5) -- (21.5,20.5) -- (21.5,21.5) -- (21.5,22.5) -- (20.5,22.5) -- (19.5,22.5) -- (18.5,22.5) -- (18.5,21.5) -- (17.5,21.5) -- (16.5,21.5) -- (16.5,20.5) -- (15.5,20.5) -- (14.5,20.5) -- (14.5,19.5) -- (14.5,18.5) -- (14.5,17.5) -- (15.5,17.5) -- (15.5,16.5) -- (15.5,15.5) -- (16.5,15.5) -- (16.5,14.5) -- (16.5,13.5);
\draw (24.5,18.5) -- (25.5,18.5) -- (25.5,19.5) -- (24.5,19.5) -- (24.5,18.5);
\draw (25.000000,19.000000) node {$T_2^+$};
\draw (13.5,20.5) -- (14.5,20.5) -- (15.5,20.5) -- (15.5,21.5) -- (15.5,22.5) -- (14.5,22.5) -- (13.5,22.5) -- (13.5,21.5) -- (13.5,20.5);
\draw (14.500000,21.500000) node {$T_3^+$};
\draw (17.5,16.5) -- (18.5,16.5) -- (19.5,16.5) -- (20.5,16.5) -- (20.5,17.5) -- (21.5,17.5) -- (22.5,17.5) -- (22.5,18.5) -- (23.5,18.5) -- (24.5,18.5) -- (24.5,19.5) -- (24.5,20.5) -- (24.5,21.5) -- (23.5,21.5) -- (23.5,22.5) -- (23.5,23.5) -- (22.5,23.5) -- (22.5,24.5) -- (22.5,25.5) -- (21.5,25.5) -- (20.5,25.5) -- (19.5,25.5) -- (19.5,24.5) -- (18.5,24.5) -- (17.5,24.5) -- (17.5,23.5) -- (16.5,23.5) -- (15.5,23.5) -- (15.5,22.5) -- (15.5,21.5) -- (15.5,20.5) -- (16.5,20.5) -- (16.5,19.5) -- (16.5,18.5) -- (17.5,18.5) -- (17.5,17.5) -- (17.5,16.5);
\draw (13.5,19.5) -- (14.5,19.5) -- (14.5,20.5) -- (13.5,20.5) -- (13.5,19.5);
\draw (14.000000,20.000000) node {$T_2^-$};
\draw (23.5,16.5) -- (24.5,16.5) -- (25.5,16.5) -- (25.5,17.5) -- (25.5,18.5) -- (24.5,18.5) -- (23.5,18.5) -- (23.5,17.5) -- (23.5,16.5);
\draw (24.500000,17.500000) node {$T_3^-$};
\end{scope}
\end{tikzpicture}
    &
    \begin{tikzpicture}[x=0.016667\textwidth,y=0.016667\textwidth,font=\tiny]
\begin{scope}[draw=gray]
\draw (40.5,40.5) -- (41.5,40.5) -- (42.5,40.5) -- (43.5,40.5) -- (43.5,41.5) -- (43.5,42.5) -- (43.5,43.5) -- (42.5,43.5) -- (41.5,43.5) -- (40.5,43.5) -- (40.5,42.5) -- (40.5,41.5) -- (40.5,40.5);
\draw (42.000000,42.000000) node {$S_4^+$};
\draw (44.5,32.5) -- (45.5,32.5) -- (46.5,32.5) -- (47.5,32.5) -- (48.5,32.5) -- (49.5,32.5) -- (49.5,33.5) -- (49.5,34.5) -- (50.5,34.5) -- (51.5,34.5) -- (52.5,34.5) -- (52.5,35.5) -- (52.5,36.5) -- (52.5,37.5) -- (52.5,38.5) -- (52.5,39.5) -- (51.5,39.5) -- (50.5,39.5) -- (50.5,40.5) -- (50.5,41.5) -- (50.5,42.5) -- (49.5,42.5) -- (48.5,42.5) -- (47.5,42.5) -- (46.5,42.5) -- (45.5,42.5) -- (45.5,41.5) -- (45.5,40.5) -- (44.5,40.5) -- (43.5,40.5) -- (42.5,40.5) -- (42.5,39.5) -- (42.5,38.5) -- (42.5,37.5) -- (42.5,36.5) -- (42.5,35.5) -- (43.5,35.5) -- (44.5,35.5) -- (44.5,34.5) -- (44.5,33.5) -- (44.5,32.5);
\draw (47.500000,37.500000) node {$S_2^+$};
\draw (53.5,39.5) -- (54.5,39.5) -- (55.5,39.5) -- (55.5,40.5) -- (55.5,41.5) -- (54.5,41.5) -- (53.5,41.5) -- (53.5,40.5) -- (53.5,39.5);
\draw (54.500000,40.500000) node {$S_3^+$};
\draw (52.5,36.5) -- (53.5,36.5) -- (54.5,36.5) -- (55.5,36.5) -- (55.5,37.5) -- (55.5,38.5) -- (55.5,39.5) -- (54.5,39.5) -- (53.5,39.5) -- (52.5,39.5) -- (52.5,38.5) -- (52.5,37.5) -- (52.5,36.5);
\draw (54.000000,38.000000) node {$S_4^-$};
\draw (45.5,37.5) -- (46.5,37.5) -- (47.5,37.5) -- (48.5,37.5) -- (49.5,37.5) -- (50.5,37.5) -- (50.5,38.5) -- (50.5,39.5) -- (51.5,39.5) -- (52.5,39.5) -- (53.5,39.5) -- (53.5,40.5) -- (53.5,41.5) -- (53.5,42.5) -- (53.5,43.5) -- (53.5,44.5) -- (52.5,44.5) -- (51.5,44.5) -- (51.5,45.5) -- (51.5,46.5) -- (51.5,47.5) -- (50.5,47.5) -- (49.5,47.5) -- (48.5,47.5) -- (47.5,47.5) -- (46.5,47.5) -- (46.5,46.5) -- (46.5,45.5) -- (45.5,45.5) -- (44.5,45.5) -- (43.5,45.5) -- (43.5,44.5) -- (43.5,43.5) -- (43.5,42.5) -- (43.5,41.5) -- (43.5,40.5) -- (44.5,40.5) -- (45.5,40.5) -- (45.5,39.5) -- (45.5,38.5) -- (45.5,37.5);
\draw (48.500000,42.500000) node {$S_2^-$};
\draw (40.5,38.5) -- (41.5,38.5) -- (42.5,38.5) -- (42.5,39.5) -- (42.5,40.5) -- (41.5,40.5) -- (40.5,40.5) -- (40.5,39.5) -- (40.5,38.5);
\draw (41.500000,39.500000) node {$S_3^-$};
\draw (28.5,44.5) -- (29.5,44.5) -- (30.5,44.5) -- (31.5,44.5) -- (31.5,45.5) -- (31.5,46.5) -- (31.5,47.5) -- (30.5,47.5) -- (29.5,47.5) -- (28.5,47.5) -- (28.5,46.5) -- (28.5,45.5) -- (28.5,44.5);
\draw (30.000000,46.000000) node {$Q_4^+$};
\draw (32.5,36.5) -- (33.5,36.5) -- (34.5,36.5) -- (35.5,36.5) -- (36.5,36.5) -- (37.5,36.5) -- (37.5,37.5) -- (37.5,38.5) -- (38.5,38.5) -- (39.5,38.5) -- (40.5,38.5) -- (40.5,39.5) -- (40.5,40.5) -- (40.5,41.5) -- (40.5,42.5) -- (40.5,43.5) -- (39.5,43.5) -- (38.5,43.5) -- (38.5,44.5) -- (38.5,45.5) -- (38.5,46.5) -- (37.5,46.5) -- (36.5,46.5) -- (35.5,46.5) -- (34.5,46.5) -- (33.5,46.5) -- (33.5,45.5) -- (33.5,44.5) -- (32.5,44.5) -- (31.5,44.5) -- (30.5,44.5) -- (30.5,43.5) -- (30.5,42.5) -- (30.5,41.5) -- (30.5,40.5) -- (30.5,39.5) -- (31.5,39.5) -- (32.5,39.5) -- (32.5,38.5) -- (32.5,37.5) -- (32.5,36.5);
\draw (35.500000,41.500000) node {$Q_2^+$};
\draw (41.5,43.5) -- (42.5,43.5) -- (43.5,43.5) -- (43.5,44.5) -- (43.5,45.5) -- (42.5,45.5) -- (41.5,45.5) -- (41.5,44.5) -- (41.5,43.5);
\draw (42.500000,44.500000) node {$Q_3^+$};
\draw (33.5,41.5) -- (34.5,41.5) -- (35.5,41.5) -- (36.5,41.5) -- (37.5,41.5) -- (38.5,41.5) -- (38.5,42.5) -- (38.5,43.5) -- (39.5,43.5) -- (40.5,43.5) -- (41.5,43.5) -- (41.5,44.5) -- (41.5,45.5) -- (41.5,46.5) -- (41.5,47.5) -- (41.5,48.5) -- (40.5,48.5) -- (39.5,48.5) -- (39.5,49.5) -- (39.5,50.5) -- (39.5,51.5) -- (38.5,51.5) -- (37.5,51.5) -- (36.5,51.5) -- (35.5,51.5) -- (34.5,51.5) -- (34.5,50.5) -- (34.5,49.5) -- (33.5,49.5) -- (32.5,49.5) -- (31.5,49.5) -- (31.5,48.5) -- (31.5,47.5) -- (31.5,46.5) -- (31.5,45.5) -- (31.5,44.5) -- (32.5,44.5) -- (33.5,44.5) -- (33.5,43.5) -- (33.5,42.5) -- (33.5,41.5);
\draw (36.500000,46.500000) node {$Q_2^-$};
\draw (28.5,42.5) -- (29.5,42.5) -- (30.5,42.5) -- (30.5,43.5) -- (30.5,44.5) -- (29.5,44.5) -- (28.5,44.5) -- (28.5,43.5) -- (28.5,42.5);
\draw (29.500000,43.500000) node {$Q_3^-$};
\end{scope}\begin{scope}[thick]
\draw (32.5,36.5) -- (33.5,36.5) -- (34.5,36.5) -- (35.5,36.5) -- (36.5,36.5) -- (37.5,36.5) -- (37.5,37.5) -- (37.5,38.5) -- (38.5,38.5) -- (39.5,38.5) -- (40.5,38.5) -- (40.5,39.5) -- (40.5,40.5) -- (41.5,40.5) -- (42.5,40.5) -- (43.5,40.5) -- (43.5,41.5) -- (43.5,42.5) -- (43.5,43.5) -- (43.5,44.5) -- (43.5,45.5) -- (42.5,45.5) -- (41.5,45.5) -- (41.5,46.5) -- (41.5,47.5) -- (41.5,48.5) -- (40.5,48.5) -- (39.5,48.5) -- (39.5,49.5) -- (39.5,50.5) -- (39.5,51.5) -- (38.5,51.5) -- (37.5,51.5) -- (36.5,51.5) -- (35.5,51.5) -- (34.5,51.5) -- (34.5,50.5) -- (34.5,49.5) -- (33.5,49.5) -- (32.5,49.5) -- (31.5,49.5) -- (31.5,48.5) -- (31.5,47.5) -- (30.5,47.5) -- (29.5,47.5) -- (28.5,47.5) -- (28.5,46.5) -- (28.5,45.5) -- (28.5,44.5) -- (28.5,43.5) -- (28.5,42.5) -- (29.5,42.5) -- (30.5,42.5) -- (30.5,41.5) -- (30.5,40.5) -- (30.5,39.5) -- (31.5,39.5) -- (32.5,39.5) -- (32.5,38.5) -- (32.5,37.5) -- (32.5,36.5);
\draw (36.5,28.5) -- (37.5,28.5) -- (38.5,28.5) -- (39.5,28.5) -- (40.5,28.5) -- (41.5,28.5) -- (41.5,29.5) -- (41.5,30.5) -- (42.5,30.5) -- (43.5,30.5) -- (44.5,30.5) -- (44.5,31.5) -- (44.5,32.5) -- (44.5,33.5) -- (44.5,34.5) -- (44.5,35.5) -- (43.5,35.5) -- (42.5,35.5) -- (42.5,36.5) -- (42.5,37.5) -- (42.5,38.5) -- (41.5,38.5) -- (40.5,38.5) -- (39.5,38.5) -- (38.5,38.5) -- (37.5,38.5) -- (37.5,37.5) -- (37.5,36.5) -- (36.5,36.5) -- (35.5,36.5) -- (34.5,36.5) -- (34.5,35.5) -- (34.5,34.5) -- (34.5,33.5) -- (34.5,32.5) -- (34.5,31.5) -- (35.5,31.5) -- (36.5,31.5) -- (36.5,30.5) -- (36.5,29.5) -- (36.5,28.5);
\draw (39.500000,33.500000) node {$T_2^+$};
\draw (49.5,47.5) -- (50.5,47.5) -- (51.5,47.5) -- (51.5,48.5) -- (51.5,49.5) -- (50.5,49.5) -- (49.5,49.5) -- (49.5,48.5) -- (49.5,47.5);
\draw (50.500000,48.500000) node {$T_3^+$};
\draw (44.5,32.5) -- (45.5,32.5) -- (46.5,32.5) -- (47.5,32.5) -- (48.5,32.5) -- (49.5,32.5) -- (49.5,33.5) -- (49.5,34.5) -- (50.5,34.5) -- (51.5,34.5) -- (52.5,34.5) -- (52.5,35.5) -- (52.5,36.5) -- (53.5,36.5) -- (54.5,36.5) -- (55.5,36.5) -- (55.5,37.5) -- (55.5,38.5) -- (55.5,39.5) -- (55.5,40.5) -- (55.5,41.5) -- (54.5,41.5) -- (53.5,41.5) -- (53.5,42.5) -- (53.5,43.5) -- (53.5,44.5) -- (52.5,44.5) -- (51.5,44.5) -- (51.5,45.5) -- (51.5,46.5) -- (51.5,47.5) -- (50.5,47.5) -- (49.5,47.5) -- (48.5,47.5) -- (47.5,47.5) -- (46.5,47.5) -- (46.5,46.5) -- (46.5,45.5) -- (45.5,45.5) -- (44.5,45.5) -- (43.5,45.5) -- (43.5,44.5) -- (43.5,43.5) -- (42.5,43.5) -- (41.5,43.5) -- (40.5,43.5) -- (40.5,42.5) -- (40.5,41.5) -- (40.5,40.5) -- (40.5,39.5) -- (40.5,38.5) -- (41.5,38.5) -- (42.5,38.5) -- (42.5,37.5) -- (42.5,36.5) -- (42.5,35.5) -- (43.5,35.5) -- (44.5,35.5) -- (44.5,34.5) -- (44.5,33.5) -- (44.5,32.5);
\draw (41.5,45.5) -- (42.5,45.5) -- (43.5,45.5) -- (44.5,45.5) -- (45.5,45.5) -- (46.5,45.5) -- (46.5,46.5) -- (46.5,47.5) -- (47.5,47.5) -- (48.5,47.5) -- (49.5,47.5) -- (49.5,48.5) -- (49.5,49.5) -- (49.5,50.5) -- (49.5,51.5) -- (49.5,52.5) -- (48.5,52.5) -- (47.5,52.5) -- (47.5,53.5) -- (47.5,54.5) -- (47.5,55.5) -- (46.5,55.5) -- (45.5,55.5) -- (44.5,55.5) -- (43.5,55.5) -- (42.5,55.5) -- (42.5,54.5) -- (42.5,53.5) -- (41.5,53.5) -- (40.5,53.5) -- (39.5,53.5) -- (39.5,52.5) -- (39.5,51.5) -- (39.5,50.5) -- (39.5,49.5) -- (39.5,48.5) -- (40.5,48.5) -- (41.5,48.5) -- (41.5,47.5) -- (41.5,46.5) -- (41.5,45.5);
\draw (44.500000,50.500000) node {$T_2^-$};
\draw (32.5,34.5) -- (33.5,34.5) -- (34.5,34.5) -- (34.5,35.5) -- (34.5,36.5) -- (33.5,36.5) -- (32.5,36.5) -- (32.5,35.5) -- (32.5,34.5);
\draw (33.500000,35.500000) node {$T_3^-$};
\end{scope}
\end{tikzpicture} \\
    \begin{tikzpicture}[x=0.016092\textwidth,y=0.016092\textwidth,font=\tiny]
\begin{scope}[draw=gray]
\draw (42.5,42.5) -- (43.5,42.5) -- (44.5,42.5) -- (44.5,43.5) -- (44.5,44.5) -- (43.5,44.5) -- (42.5,44.5) -- (42.5,43.5) -- (42.5,42.5);
\draw (43.500000,43.500000) node {$S_4^+$};
\draw (29.5,43.5) -- (30.5,43.5) -- (31.5,43.5) -- (32.5,43.5) -- (32.5,44.5) -- (32.5,45.5) -- (32.5,46.5) -- (31.5,46.5) -- (30.5,46.5) -- (29.5,46.5) -- (29.5,45.5) -- (29.5,44.5) -- (29.5,43.5);
\draw (31.000000,45.000000) node {$S_2^+$};
\draw (33.5,35.5) -- (34.5,35.5) -- (35.5,35.5) -- (36.5,35.5) -- (37.5,35.5) -- (38.5,35.5) -- (38.5,36.5) -- (38.5,37.5) -- (39.5,37.5) -- (40.5,37.5) -- (41.5,37.5) -- (41.5,38.5) -- (41.5,39.5) -- (41.5,40.5) -- (41.5,41.5) -- (41.5,42.5) -- (40.5,42.5) -- (39.5,42.5) -- (39.5,43.5) -- (39.5,44.5) -- (39.5,45.5) -- (38.5,45.5) -- (37.5,45.5) -- (36.5,45.5) -- (35.5,45.5) -- (34.5,45.5) -- (34.5,44.5) -- (34.5,43.5) -- (33.5,43.5) -- (32.5,43.5) -- (31.5,43.5) -- (31.5,42.5) -- (31.5,41.5) -- (31.5,40.5) -- (31.5,39.5) -- (31.5,38.5) -- (32.5,38.5) -- (33.5,38.5) -- (33.5,37.5) -- (33.5,36.5) -- (33.5,35.5);
\draw (36.500000,40.500000) node {$S_3^+$};
\draw (29.5,41.5) -- (30.5,41.5) -- (31.5,41.5) -- (31.5,42.5) -- (31.5,43.5) -- (30.5,43.5) -- (29.5,43.5) -- (29.5,42.5) -- (29.5,41.5);
\draw (30.500000,42.500000) node {$S_4^-$};
\draw (41.5,39.5) -- (42.5,39.5) -- (43.5,39.5) -- (44.5,39.5) -- (44.5,40.5) -- (44.5,41.5) -- (44.5,42.5) -- (43.5,42.5) -- (42.5,42.5) -- (41.5,42.5) -- (41.5,41.5) -- (41.5,40.5) -- (41.5,39.5);
\draw (43.000000,41.000000) node {$S_2^-$};
\draw (34.5,40.5) -- (35.5,40.5) -- (36.5,40.5) -- (37.5,40.5) -- (38.5,40.5) -- (39.5,40.5) -- (39.5,41.5) -- (39.5,42.5) -- (40.5,42.5) -- (41.5,42.5) -- (42.5,42.5) -- (42.5,43.5) -- (42.5,44.5) -- (42.5,45.5) -- (42.5,46.5) -- (42.5,47.5) -- (41.5,47.5) -- (40.5,47.5) -- (40.5,48.5) -- (40.5,49.5) -- (40.5,50.5) -- (39.5,50.5) -- (38.5,50.5) -- (37.5,50.5) -- (36.5,50.5) -- (35.5,50.5) -- (35.5,49.5) -- (35.5,48.5) -- (34.5,48.5) -- (33.5,48.5) -- (32.5,48.5) -- (32.5,47.5) -- (32.5,46.5) -- (32.5,45.5) -- (32.5,44.5) -- (32.5,43.5) -- (33.5,43.5) -- (34.5,43.5) -- (34.5,42.5) -- (34.5,41.5) -- (34.5,40.5);
\draw (37.500000,45.500000) node {$S_3^-$};
\draw (55.5,43.5) -- (56.5,43.5) -- (57.5,43.5) -- (57.5,44.5) -- (57.5,45.5) -- (56.5,45.5) -- (55.5,45.5) -- (55.5,44.5) -- (55.5,43.5);
\draw (56.500000,44.500000) node {$Q_4^+$};
\draw (42.5,44.5) -- (43.5,44.5) -- (44.5,44.5) -- (45.5,44.5) -- (45.5,45.5) -- (45.5,46.5) -- (45.5,47.5) -- (44.5,47.5) -- (43.5,47.5) -- (42.5,47.5) -- (42.5,46.5) -- (42.5,45.5) -- (42.5,44.5);
\draw (44.000000,46.000000) node {$Q_2^+$};
\draw (46.5,36.5) -- (47.5,36.5) -- (48.5,36.5) -- (49.5,36.5) -- (50.5,36.5) -- (51.5,36.5) -- (51.5,37.5) -- (51.5,38.5) -- (52.5,38.5) -- (53.5,38.5) -- (54.5,38.5) -- (54.5,39.5) -- (54.5,40.5) -- (54.5,41.5) -- (54.5,42.5) -- (54.5,43.5) -- (53.5,43.5) -- (52.5,43.5) -- (52.5,44.5) -- (52.5,45.5) -- (52.5,46.5) -- (51.5,46.5) -- (50.5,46.5) -- (49.5,46.5) -- (48.5,46.5) -- (47.5,46.5) -- (47.5,45.5) -- (47.5,44.5) -- (46.5,44.5) -- (45.5,44.5) -- (44.5,44.5) -- (44.5,43.5) -- (44.5,42.5) -- (44.5,41.5) -- (44.5,40.5) -- (44.5,39.5) -- (45.5,39.5) -- (46.5,39.5) -- (46.5,38.5) -- (46.5,37.5) -- (46.5,36.5);
\draw (49.500000,41.500000) node {$Q_3^+$};
\draw (54.5,40.5) -- (55.5,40.5) -- (56.5,40.5) -- (57.5,40.5) -- (57.5,41.5) -- (57.5,42.5) -- (57.5,43.5) -- (56.5,43.5) -- (55.5,43.5) -- (54.5,43.5) -- (54.5,42.5) -- (54.5,41.5) -- (54.5,40.5);
\draw (56.000000,42.000000) node {$Q_2^-$};
\draw (47.5,41.5) -- (48.5,41.5) -- (49.5,41.5) -- (50.5,41.5) -- (51.5,41.5) -- (52.5,41.5) -- (52.5,42.5) -- (52.5,43.5) -- (53.5,43.5) -- (54.5,43.5) -- (55.5,43.5) -- (55.5,44.5) -- (55.5,45.5) -- (55.5,46.5) -- (55.5,47.5) -- (55.5,48.5) -- (54.5,48.5) -- (53.5,48.5) -- (53.5,49.5) -- (53.5,50.5) -- (53.5,51.5) -- (52.5,51.5) -- (51.5,51.5) -- (50.5,51.5) -- (49.5,51.5) -- (48.5,51.5) -- (48.5,50.5) -- (48.5,49.5) -- (47.5,49.5) -- (46.5,49.5) -- (45.5,49.5) -- (45.5,48.5) -- (45.5,47.5) -- (45.5,46.5) -- (45.5,45.5) -- (45.5,44.5) -- (46.5,44.5) -- (47.5,44.5) -- (47.5,43.5) -- (47.5,42.5) -- (47.5,41.5);
\draw (50.500000,46.500000) node {$Q_3^-$};
\end{scope}\begin{scope}[thick]
\draw (46.5,36.5) -- (47.5,36.5) -- (48.5,36.5) -- (49.5,36.5) -- (50.5,36.5) -- (51.5,36.5) -- (51.5,37.5) -- (51.5,38.5) -- (52.5,38.5) -- (53.5,38.5) -- (54.5,38.5) -- (54.5,39.5) -- (54.5,40.5) -- (55.5,40.5) -- (56.5,40.5) -- (57.5,40.5) -- (57.5,41.5) -- (57.5,42.5) -- (57.5,43.5) -- (57.5,44.5) -- (57.5,45.5) -- (56.5,45.5) -- (55.5,45.5) -- (55.5,46.5) -- (55.5,47.5) -- (55.5,48.5) -- (54.5,48.5) -- (53.5,48.5) -- (53.5,49.5) -- (53.5,50.5) -- (53.5,51.5) -- (52.5,51.5) -- (51.5,51.5) -- (50.5,51.5) -- (49.5,51.5) -- (48.5,51.5) -- (48.5,50.5) -- (48.5,49.5) -- (47.5,49.5) -- (46.5,49.5) -- (45.5,49.5) -- (45.5,48.5) -- (45.5,47.5) -- (44.5,47.5) -- (43.5,47.5) -- (42.5,47.5) -- (42.5,46.5) -- (42.5,45.5) -- (42.5,44.5) -- (42.5,43.5) -- (42.5,42.5) -- (43.5,42.5) -- (44.5,42.5) -- (44.5,41.5) -- (44.5,40.5) -- (44.5,39.5) -- (45.5,39.5) -- (46.5,39.5) -- (46.5,38.5) -- (46.5,37.5) -- (46.5,36.5);
\draw (35.5,50.5) -- (36.5,50.5) -- (37.5,50.5) -- (38.5,50.5) -- (38.5,51.5) -- (38.5,52.5) -- (38.5,53.5) -- (37.5,53.5) -- (36.5,53.5) -- (35.5,53.5) -- (35.5,52.5) -- (35.5,51.5) -- (35.5,50.5);
\draw (37.000000,52.000000) node {$T_2^+$};
\draw (40.5,29.5) -- (41.5,29.5) -- (42.5,29.5) -- (43.5,29.5) -- (44.5,29.5) -- (45.5,29.5) -- (45.5,30.5) -- (45.5,31.5) -- (46.5,31.5) -- (47.5,31.5) -- (48.5,31.5) -- (48.5,32.5) -- (48.5,33.5) -- (48.5,34.5) -- (48.5,35.5) -- (48.5,36.5) -- (47.5,36.5) -- (46.5,36.5) -- (46.5,37.5) -- (46.5,38.5) -- (46.5,39.5) -- (45.5,39.5) -- (44.5,39.5) -- (43.5,39.5) -- (42.5,39.5) -- (41.5,39.5) -- (41.5,38.5) -- (41.5,37.5) -- (40.5,37.5) -- (39.5,37.5) -- (38.5,37.5) -- (38.5,36.5) -- (38.5,35.5) -- (38.5,34.5) -- (38.5,33.5) -- (38.5,32.5) -- (39.5,32.5) -- (40.5,32.5) -- (40.5,31.5) -- (40.5,30.5) -- (40.5,29.5);
\draw (43.500000,34.500000) node {$T_3^+$};
\draw (33.5,35.5) -- (34.5,35.5) -- (35.5,35.5) -- (36.5,35.5) -- (37.5,35.5) -- (38.5,35.5) -- (38.5,36.5) -- (38.5,37.5) -- (39.5,37.5) -- (40.5,37.5) -- (41.5,37.5) -- (41.5,38.5) -- (41.5,39.5) -- (42.5,39.5) -- (43.5,39.5) -- (44.5,39.5) -- (44.5,40.5) -- (44.5,41.5) -- (44.5,42.5) -- (44.5,43.5) -- (44.5,44.5) -- (43.5,44.5) -- (42.5,44.5) -- (42.5,45.5) -- (42.5,46.5) -- (42.5,47.5) -- (41.5,47.5) -- (40.5,47.5) -- (40.5,48.5) -- (40.5,49.5) -- (40.5,50.5) -- (39.5,50.5) -- (38.5,50.5) -- (37.5,50.5) -- (36.5,50.5) -- (35.5,50.5) -- (35.5,49.5) -- (35.5,48.5) -- (34.5,48.5) -- (33.5,48.5) -- (32.5,48.5) -- (32.5,47.5) -- (32.5,46.5) -- (31.5,46.5) -- (30.5,46.5) -- (29.5,46.5) -- (29.5,45.5) -- (29.5,44.5) -- (29.5,43.5) -- (29.5,42.5) -- (29.5,41.5) -- (30.5,41.5) -- (31.5,41.5) -- (31.5,40.5) -- (31.5,39.5) -- (31.5,38.5) -- (32.5,38.5) -- (33.5,38.5) -- (33.5,37.5) -- (33.5,36.5) -- (33.5,35.5);
\draw (48.5,33.5) -- (49.5,33.5) -- (50.5,33.5) -- (51.5,33.5) -- (51.5,34.5) -- (51.5,35.5) -- (51.5,36.5) -- (50.5,36.5) -- (49.5,36.5) -- (48.5,36.5) -- (48.5,35.5) -- (48.5,34.5) -- (48.5,33.5);
\draw (50.000000,35.000000) node {$T_2^-$};
\draw (40.5,47.5) -- (41.5,47.5) -- (42.5,47.5) -- (43.5,47.5) -- (44.5,47.5) -- (45.5,47.5) -- (45.5,48.5) -- (45.5,49.5) -- (46.5,49.5) -- (47.5,49.5) -- (48.5,49.5) -- (48.5,50.5) -- (48.5,51.5) -- (48.5,52.5) -- (48.5,53.5) -- (48.5,54.5) -- (47.5,54.5) -- (46.5,54.5) -- (46.5,55.5) -- (46.5,56.5) -- (46.5,57.5) -- (45.5,57.5) -- (44.5,57.5) -- (43.5,57.5) -- (42.5,57.5) -- (41.5,57.5) -- (41.5,56.5) -- (41.5,55.5) -- (40.5,55.5) -- (39.5,55.5) -- (38.5,55.5) -- (38.5,54.5) -- (38.5,53.5) -- (38.5,52.5) -- (38.5,51.5) -- (38.5,50.5) -- (39.5,50.5) -- (40.5,50.5) -- (40.5,49.5) -- (40.5,48.5) -- (40.5,47.5);
\draw (43.500000,52.500000) node {$T_3^-$};
\end{scope}
\end{tikzpicture}
    &
    \begin{tikzpicture}[x=0.016667\textwidth,y=0.016667\textwidth,font=\tiny]
\begin{scope}[draw=gray]
\draw (40.5,40.5) -- (41.5,40.5) -- (42.5,40.5) -- (43.5,40.5) -- (43.5,41.5) -- (43.5,42.5) -- (43.5,43.5) -- (42.5,43.5) -- (41.5,43.5) -- (40.5,43.5) -- (40.5,42.5) -- (40.5,41.5) -- (40.5,40.5);
\draw (42.000000,42.000000) node {$S_4^+$};
\draw (28.5,39.5) -- (29.5,39.5) -- (30.5,39.5) -- (30.5,40.5) -- (30.5,41.5) -- (29.5,41.5) -- (28.5,41.5) -- (28.5,40.5) -- (28.5,39.5);
\draw (29.500000,40.500000) node {$S_2^+$};
\draw (34.5,32.5) -- (35.5,32.5) -- (36.5,32.5) -- (37.5,32.5) -- (38.5,32.5) -- (39.5,32.5) -- (39.5,33.5) -- (39.5,34.5) -- (39.5,35.5) -- (40.5,35.5) -- (41.5,35.5) -- (41.5,36.5) -- (41.5,37.5) -- (41.5,38.5) -- (41.5,39.5) -- (41.5,40.5) -- (40.5,40.5) -- (39.5,40.5) -- (38.5,40.5) -- (38.5,41.5) -- (38.5,42.5) -- (37.5,42.5) -- (36.5,42.5) -- (35.5,42.5) -- (34.5,42.5) -- (33.5,42.5) -- (33.5,41.5) -- (33.5,40.5) -- (33.5,39.5) -- (32.5,39.5) -- (31.5,39.5) -- (31.5,38.5) -- (31.5,37.5) -- (31.5,36.5) -- (31.5,35.5) -- (31.5,34.5) -- (32.5,34.5) -- (33.5,34.5) -- (34.5,34.5) -- (34.5,33.5) -- (34.5,32.5);
\draw (36.500000,37.500000) node {$S_3^+$};
\draw (28.5,36.5) -- (29.5,36.5) -- (30.5,36.5) -- (31.5,36.5) -- (31.5,37.5) -- (31.5,38.5) -- (31.5,39.5) -- (30.5,39.5) -- (29.5,39.5) -- (28.5,39.5) -- (28.5,38.5) -- (28.5,37.5) -- (28.5,36.5);
\draw (30.000000,38.000000) node {$S_4^-$};
\draw (41.5,38.5) -- (42.5,38.5) -- (43.5,38.5) -- (43.5,39.5) -- (43.5,40.5) -- (42.5,40.5) -- (41.5,40.5) -- (41.5,39.5) -- (41.5,38.5);
\draw (42.500000,39.500000) node {$S_2^-$};
\draw (33.5,37.5) -- (34.5,37.5) -- (35.5,37.5) -- (36.5,37.5) -- (37.5,37.5) -- (38.5,37.5) -- (38.5,38.5) -- (38.5,39.5) -- (38.5,40.5) -- (39.5,40.5) -- (40.5,40.5) -- (40.5,41.5) -- (40.5,42.5) -- (40.5,43.5) -- (40.5,44.5) -- (40.5,45.5) -- (39.5,45.5) -- (38.5,45.5) -- (37.5,45.5) -- (37.5,46.5) -- (37.5,47.5) -- (36.5,47.5) -- (35.5,47.5) -- (34.5,47.5) -- (33.5,47.5) -- (32.5,47.5) -- (32.5,46.5) -- (32.5,45.5) -- (32.5,44.5) -- (31.5,44.5) -- (30.5,44.5) -- (30.5,43.5) -- (30.5,42.5) -- (30.5,41.5) -- (30.5,40.5) -- (30.5,39.5) -- (31.5,39.5) -- (32.5,39.5) -- (33.5,39.5) -- (33.5,38.5) -- (33.5,37.5);
\draw (35.500000,42.500000) node {$S_3^-$};
\draw (52.5,44.5) -- (53.5,44.5) -- (54.5,44.5) -- (55.5,44.5) -- (55.5,45.5) -- (55.5,46.5) -- (55.5,47.5) -- (54.5,47.5) -- (53.5,47.5) -- (52.5,47.5) -- (52.5,46.5) -- (52.5,45.5) -- (52.5,44.5);
\draw (54.000000,46.000000) node {$Q_4^+$};
\draw (40.5,43.5) -- (41.5,43.5) -- (42.5,43.5) -- (42.5,44.5) -- (42.5,45.5) -- (41.5,45.5) -- (40.5,45.5) -- (40.5,44.5) -- (40.5,43.5);
\draw (41.500000,44.500000) node {$Q_2^+$};
\draw (46.5,36.5) -- (47.5,36.5) -- (48.5,36.5) -- (49.5,36.5) -- (50.5,36.5) -- (51.5,36.5) -- (51.5,37.5) -- (51.5,38.5) -- (51.5,39.5) -- (52.5,39.5) -- (53.5,39.5) -- (53.5,40.5) -- (53.5,41.5) -- (53.5,42.5) -- (53.5,43.5) -- (53.5,44.5) -- (52.5,44.5) -- (51.5,44.5) -- (50.5,44.5) -- (50.5,45.5) -- (50.5,46.5) -- (49.5,46.5) -- (48.5,46.5) -- (47.5,46.5) -- (46.5,46.5) -- (45.5,46.5) -- (45.5,45.5) -- (45.5,44.5) -- (45.5,43.5) -- (44.5,43.5) -- (43.5,43.5) -- (43.5,42.5) -- (43.5,41.5) -- (43.5,40.5) -- (43.5,39.5) -- (43.5,38.5) -- (44.5,38.5) -- (45.5,38.5) -- (46.5,38.5) -- (46.5,37.5) -- (46.5,36.5);
\draw (48.500000,41.500000) node {$Q_3^+$};
\draw (53.5,42.5) -- (54.5,42.5) -- (55.5,42.5) -- (55.5,43.5) -- (55.5,44.5) -- (54.5,44.5) -- (53.5,44.5) -- (53.5,43.5) -- (53.5,42.5);
\draw (54.500000,43.500000) node {$Q_2^-$};
\draw (45.5,41.5) -- (46.5,41.5) -- (47.5,41.5) -- (48.5,41.5) -- (49.5,41.5) -- (50.5,41.5) -- (50.5,42.5) -- (50.5,43.5) -- (50.5,44.5) -- (51.5,44.5) -- (52.5,44.5) -- (52.5,45.5) -- (52.5,46.5) -- (52.5,47.5) -- (52.5,48.5) -- (52.5,49.5) -- (51.5,49.5) -- (50.5,49.5) -- (49.5,49.5) -- (49.5,50.5) -- (49.5,51.5) -- (48.5,51.5) -- (47.5,51.5) -- (46.5,51.5) -- (45.5,51.5) -- (44.5,51.5) -- (44.5,50.5) -- (44.5,49.5) -- (44.5,48.5) -- (43.5,48.5) -- (42.5,48.5) -- (42.5,47.5) -- (42.5,46.5) -- (42.5,45.5) -- (42.5,44.5) -- (42.5,43.5) -- (43.5,43.5) -- (44.5,43.5) -- (45.5,43.5) -- (45.5,42.5) -- (45.5,41.5);
\draw (47.500000,46.500000) node {$Q_3^-$};
\end{scope}\begin{scope}[thick]
\draw (46.5,36.5) -- (47.5,36.5) -- (48.5,36.5) -- (49.5,36.5) -- (50.5,36.5) -- (51.5,36.5) -- (51.5,37.5) -- (51.5,38.5) -- (51.5,39.5) -- (52.5,39.5) -- (53.5,39.5) -- (53.5,40.5) -- (53.5,41.5) -- (53.5,42.5) -- (54.5,42.5) -- (55.5,42.5) -- (55.5,43.5) -- (55.5,44.5) -- (55.5,45.5) -- (55.5,46.5) -- (55.5,47.5) -- (54.5,47.5) -- (53.5,47.5) -- (52.5,47.5) -- (52.5,48.5) -- (52.5,49.5) -- (51.5,49.5) -- (50.5,49.5) -- (49.5,49.5) -- (49.5,50.5) -- (49.5,51.5) -- (48.5,51.5) -- (47.5,51.5) -- (46.5,51.5) -- (45.5,51.5) -- (44.5,51.5) -- (44.5,50.5) -- (44.5,49.5) -- (44.5,48.5) -- (43.5,48.5) -- (42.5,48.5) -- (42.5,47.5) -- (42.5,46.5) -- (42.5,45.5) -- (41.5,45.5) -- (40.5,45.5) -- (40.5,44.5) -- (40.5,43.5) -- (40.5,42.5) -- (40.5,41.5) -- (40.5,40.5) -- (41.5,40.5) -- (42.5,40.5) -- (43.5,40.5) -- (43.5,39.5) -- (43.5,38.5) -- (44.5,38.5) -- (45.5,38.5) -- (46.5,38.5) -- (46.5,37.5) -- (46.5,36.5);
\draw (32.5,47.5) -- (33.5,47.5) -- (34.5,47.5) -- (34.5,48.5) -- (34.5,49.5) -- (33.5,49.5) -- (32.5,49.5) -- (32.5,48.5) -- (32.5,47.5);
\draw (33.500000,48.500000) node {$T_2^+$};
\draw (42.5,28.5) -- (43.5,28.5) -- (44.5,28.5) -- (45.5,28.5) -- (46.5,28.5) -- (47.5,28.5) -- (47.5,29.5) -- (47.5,30.5) -- (47.5,31.5) -- (48.5,31.5) -- (49.5,31.5) -- (49.5,32.5) -- (49.5,33.5) -- (49.5,34.5) -- (49.5,35.5) -- (49.5,36.5) -- (48.5,36.5) -- (47.5,36.5) -- (46.5,36.5) -- (46.5,37.5) -- (46.5,38.5) -- (45.5,38.5) -- (44.5,38.5) -- (43.5,38.5) -- (42.5,38.5) -- (41.5,38.5) -- (41.5,37.5) -- (41.5,36.5) -- (41.5,35.5) -- (40.5,35.5) -- (39.5,35.5) -- (39.5,34.5) -- (39.5,33.5) -- (39.5,32.5) -- (39.5,31.5) -- (39.5,30.5) -- (40.5,30.5) -- (41.5,30.5) -- (42.5,30.5) -- (42.5,29.5) -- (42.5,28.5);
\draw (44.500000,33.500000) node {$T_3^+$};
\draw (34.5,32.5) -- (35.5,32.5) -- (36.5,32.5) -- (37.5,32.5) -- (38.5,32.5) -- (39.5,32.5) -- (39.5,33.5) -- (39.5,34.5) -- (39.5,35.5) -- (40.5,35.5) -- (41.5,35.5) -- (41.5,36.5) -- (41.5,37.5) -- (41.5,38.5) -- (42.5,38.5) -- (43.5,38.5) -- (43.5,39.5) -- (43.5,40.5) -- (43.5,41.5) -- (43.5,42.5) -- (43.5,43.5) -- (42.5,43.5) -- (41.5,43.5) -- (40.5,43.5) -- (40.5,44.5) -- (40.5,45.5) -- (39.5,45.5) -- (38.5,45.5) -- (37.5,45.5) -- (37.5,46.5) -- (37.5,47.5) -- (36.5,47.5) -- (35.5,47.5) -- (34.5,47.5) -- (33.5,47.5) -- (32.5,47.5) -- (32.5,46.5) -- (32.5,45.5) -- (32.5,44.5) -- (31.5,44.5) -- (30.5,44.5) -- (30.5,43.5) -- (30.5,42.5) -- (30.5,41.5) -- (29.5,41.5) -- (28.5,41.5) -- (28.5,40.5) -- (28.5,39.5) -- (28.5,38.5) -- (28.5,37.5) -- (28.5,36.5) -- (29.5,36.5) -- (30.5,36.5) -- (31.5,36.5) -- (31.5,35.5) -- (31.5,34.5) -- (32.5,34.5) -- (33.5,34.5) -- (34.5,34.5) -- (34.5,33.5) -- (34.5,32.5);
\draw (49.5,34.5) -- (50.5,34.5) -- (51.5,34.5) -- (51.5,35.5) -- (51.5,36.5) -- (50.5,36.5) -- (49.5,36.5) -- (49.5,35.5) -- (49.5,34.5);
\draw (50.500000,35.500000) node {$T_2^-$};
\draw (37.5,45.5) -- (38.5,45.5) -- (39.5,45.5) -- (40.5,45.5) -- (41.5,45.5) -- (42.5,45.5) -- (42.5,46.5) -- (42.5,47.5) -- (42.5,48.5) -- (43.5,48.5) -- (44.5,48.5) -- (44.5,49.5) -- (44.5,50.5) -- (44.5,51.5) -- (44.5,52.5) -- (44.5,53.5) -- (43.5,53.5) -- (42.5,53.5) -- (41.5,53.5) -- (41.5,54.5) -- (41.5,55.5) -- (40.5,55.5) -- (39.5,55.5) -- (38.5,55.5) -- (37.5,55.5) -- (36.5,55.5) -- (36.5,54.5) -- (36.5,53.5) -- (36.5,52.5) -- (35.5,52.5) -- (34.5,52.5) -- (34.5,51.5) -- (34.5,50.5) -- (34.5,49.5) -- (34.5,48.5) -- (34.5,47.5) -- (35.5,47.5) -- (36.5,47.5) -- (37.5,47.5) -- (37.5,46.5) -- (37.5,45.5);
\draw (39.500000,50.500000) node {$T_3^-$};
\end{scope}
\end{tikzpicture} \\
    \begin{tikzpicture}[x=0.016092\textwidth,y=0.016092\textwidth,font=\tiny]
\begin{scope}[draw=gray]
\draw (42.5,42.5) -- (43.5,42.5) -- (44.5,42.5) -- (44.5,43.5) -- (44.5,44.5) -- (43.5,44.5) -- (42.5,44.5) -- (42.5,43.5) -- (42.5,42.5);
\draw (43.500000,43.500000) node {$S_4^+$};
\draw (48.5,35.5) -- (49.5,35.5) -- (50.5,35.5) -- (51.5,35.5) -- (52.5,35.5) -- (53.5,35.5) -- (53.5,36.5) -- (53.5,37.5) -- (53.5,38.5) -- (54.5,38.5) -- (55.5,38.5) -- (55.5,39.5) -- (55.5,40.5) -- (55.5,41.5) -- (55.5,42.5) -- (55.5,43.5) -- (54.5,43.5) -- (53.5,43.5) -- (52.5,43.5) -- (52.5,44.5) -- (52.5,45.5) -- (51.5,45.5) -- (50.5,45.5) -- (49.5,45.5) -- (48.5,45.5) -- (47.5,45.5) -- (47.5,44.5) -- (47.5,43.5) -- (47.5,42.5) -- (46.5,42.5) -- (45.5,42.5) -- (45.5,41.5) -- (45.5,40.5) -- (45.5,39.5) -- (45.5,38.5) -- (45.5,37.5) -- (46.5,37.5) -- (47.5,37.5) -- (48.5,37.5) -- (48.5,36.5) -- (48.5,35.5);
\draw (50.500000,40.500000) node {$S_2^+$};
\draw (54.5,43.5) -- (55.5,43.5) -- (56.5,43.5) -- (57.5,43.5) -- (57.5,44.5) -- (57.5,45.5) -- (57.5,46.5) -- (56.5,46.5) -- (55.5,46.5) -- (54.5,46.5) -- (54.5,45.5) -- (54.5,44.5) -- (54.5,43.5);
\draw (56.000000,45.000000) node {$S_3^+$};
\draw (55.5,41.5) -- (56.5,41.5) -- (57.5,41.5) -- (57.5,42.5) -- (57.5,43.5) -- (56.5,43.5) -- (55.5,43.5) -- (55.5,42.5) -- (55.5,41.5);
\draw (56.500000,42.500000) node {$S_4^-$};
\draw (47.5,40.5) -- (48.5,40.5) -- (49.5,40.5) -- (50.5,40.5) -- (51.5,40.5) -- (52.5,40.5) -- (52.5,41.5) -- (52.5,42.5) -- (52.5,43.5) -- (53.5,43.5) -- (54.5,43.5) -- (54.5,44.5) -- (54.5,45.5) -- (54.5,46.5) -- (54.5,47.5) -- (54.5,48.5) -- (53.5,48.5) -- (52.5,48.5) -- (51.5,48.5) -- (51.5,49.5) -- (51.5,50.5) -- (50.5,50.5) -- (49.5,50.5) -- (48.5,50.5) -- (47.5,50.5) -- (46.5,50.5) -- (46.5,49.5) -- (46.5,48.5) -- (46.5,47.5) -- (45.5,47.5) -- (44.5,47.5) -- (44.5,46.5) -- (44.5,45.5) -- (44.5,44.5) -- (44.5,43.5) -- (44.5,42.5) -- (45.5,42.5) -- (46.5,42.5) -- (47.5,42.5) -- (47.5,41.5) -- (47.5,40.5);
\draw (49.500000,45.500000) node {$S_2^-$};
\draw (42.5,39.5) -- (43.5,39.5) -- (44.5,39.5) -- (45.5,39.5) -- (45.5,40.5) -- (45.5,41.5) -- (45.5,42.5) -- (44.5,42.5) -- (43.5,42.5) -- (42.5,42.5) -- (42.5,41.5) -- (42.5,40.5) -- (42.5,39.5);
\draw (44.000000,41.000000) node {$S_3^-$};
\draw (29.5,43.5) -- (30.5,43.5) -- (31.5,43.5) -- (31.5,44.5) -- (31.5,45.5) -- (30.5,45.5) -- (29.5,45.5) -- (29.5,44.5) -- (29.5,43.5);
\draw (30.500000,44.500000) node {$Q_4^+$};
\draw (35.5,36.5) -- (36.5,36.5) -- (37.5,36.5) -- (38.5,36.5) -- (39.5,36.5) -- (40.5,36.5) -- (40.5,37.5) -- (40.5,38.5) -- (40.5,39.5) -- (41.5,39.5) -- (42.5,39.5) -- (42.5,40.5) -- (42.5,41.5) -- (42.5,42.5) -- (42.5,43.5) -- (42.5,44.5) -- (41.5,44.5) -- (40.5,44.5) -- (39.5,44.5) -- (39.5,45.5) -- (39.5,46.5) -- (38.5,46.5) -- (37.5,46.5) -- (36.5,46.5) -- (35.5,46.5) -- (34.5,46.5) -- (34.5,45.5) -- (34.5,44.5) -- (34.5,43.5) -- (33.5,43.5) -- (32.5,43.5) -- (32.5,42.5) -- (32.5,41.5) -- (32.5,40.5) -- (32.5,39.5) -- (32.5,38.5) -- (33.5,38.5) -- (34.5,38.5) -- (35.5,38.5) -- (35.5,37.5) -- (35.5,36.5);
\draw (37.500000,41.500000) node {$Q_2^+$};
\draw (41.5,44.5) -- (42.5,44.5) -- (43.5,44.5) -- (44.5,44.5) -- (44.5,45.5) -- (44.5,46.5) -- (44.5,47.5) -- (43.5,47.5) -- (42.5,47.5) -- (41.5,47.5) -- (41.5,46.5) -- (41.5,45.5) -- (41.5,44.5);
\draw (43.000000,46.000000) node {$Q_3^+$};
\draw (34.5,41.5) -- (35.5,41.5) -- (36.5,41.5) -- (37.5,41.5) -- (38.5,41.5) -- (39.5,41.5) -- (39.5,42.5) -- (39.5,43.5) -- (39.5,44.5) -- (40.5,44.5) -- (41.5,44.5) -- (41.5,45.5) -- (41.5,46.5) -- (41.5,47.5) -- (41.5,48.5) -- (41.5,49.5) -- (40.5,49.5) -- (39.5,49.5) -- (38.5,49.5) -- (38.5,50.5) -- (38.5,51.5) -- (37.5,51.5) -- (36.5,51.5) -- (35.5,51.5) -- (34.5,51.5) -- (33.5,51.5) -- (33.5,50.5) -- (33.5,49.5) -- (33.5,48.5) -- (32.5,48.5) -- (31.5,48.5) -- (31.5,47.5) -- (31.5,46.5) -- (31.5,45.5) -- (31.5,44.5) -- (31.5,43.5) -- (32.5,43.5) -- (33.5,43.5) -- (34.5,43.5) -- (34.5,42.5) -- (34.5,41.5);
\draw (36.500000,46.500000) node {$Q_2^-$};
\draw (29.5,40.5) -- (30.5,40.5) -- (31.5,40.5) -- (32.5,40.5) -- (32.5,41.5) -- (32.5,42.5) -- (32.5,43.5) -- (31.5,43.5) -- (30.5,43.5) -- (29.5,43.5) -- (29.5,42.5) -- (29.5,41.5) -- (29.5,40.5);
\draw (31.000000,42.000000) node {$Q_3^-$};
\end{scope}\begin{scope}[thick]
\draw (35.5,36.5) -- (36.5,36.5) -- (37.5,36.5) -- (38.5,36.5) -- (39.5,36.5) -- (40.5,36.5) -- (40.5,37.5) -- (40.5,38.5) -- (40.5,39.5) -- (41.5,39.5) -- (42.5,39.5) -- (42.5,40.5) -- (42.5,41.5) -- (42.5,42.5) -- (43.5,42.5) -- (44.5,42.5) -- (44.5,43.5) -- (44.5,44.5) -- (44.5,45.5) -- (44.5,46.5) -- (44.5,47.5) -- (43.5,47.5) -- (42.5,47.5) -- (41.5,47.5) -- (41.5,48.5) -- (41.5,49.5) -- (40.5,49.5) -- (39.5,49.5) -- (38.5,49.5) -- (38.5,50.5) -- (38.5,51.5) -- (37.5,51.5) -- (36.5,51.5) -- (35.5,51.5) -- (34.5,51.5) -- (33.5,51.5) -- (33.5,50.5) -- (33.5,49.5) -- (33.5,48.5) -- (32.5,48.5) -- (31.5,48.5) -- (31.5,47.5) -- (31.5,46.5) -- (31.5,45.5) -- (30.5,45.5) -- (29.5,45.5) -- (29.5,44.5) -- (29.5,43.5) -- (29.5,42.5) -- (29.5,41.5) -- (29.5,40.5) -- (30.5,40.5) -- (31.5,40.5) -- (32.5,40.5) -- (32.5,39.5) -- (32.5,38.5) -- (33.5,38.5) -- (34.5,38.5) -- (35.5,38.5) -- (35.5,37.5) -- (35.5,36.5);
\draw (41.5,29.5) -- (42.5,29.5) -- (43.5,29.5) -- (44.5,29.5) -- (45.5,29.5) -- (46.5,29.5) -- (46.5,30.5) -- (46.5,31.5) -- (46.5,32.5) -- (47.5,32.5) -- (48.5,32.5) -- (48.5,33.5) -- (48.5,34.5) -- (48.5,35.5) -- (48.5,36.5) -- (48.5,37.5) -- (47.5,37.5) -- (46.5,37.5) -- (45.5,37.5) -- (45.5,38.5) -- (45.5,39.5) -- (44.5,39.5) -- (43.5,39.5) -- (42.5,39.5) -- (41.5,39.5) -- (40.5,39.5) -- (40.5,38.5) -- (40.5,37.5) -- (40.5,36.5) -- (39.5,36.5) -- (38.5,36.5) -- (38.5,35.5) -- (38.5,34.5) -- (38.5,33.5) -- (38.5,32.5) -- (38.5,31.5) -- (39.5,31.5) -- (40.5,31.5) -- (41.5,31.5) -- (41.5,30.5) -- (41.5,29.5);
\draw (43.500000,34.500000) node {$T_2^+$};
\draw (48.5,50.5) -- (49.5,50.5) -- (50.5,50.5) -- (51.5,50.5) -- (51.5,51.5) -- (51.5,52.5) -- (51.5,53.5) -- (50.5,53.5) -- (49.5,53.5) -- (48.5,53.5) -- (48.5,52.5) -- (48.5,51.5) -- (48.5,50.5);
\draw (50.000000,52.000000) node {$T_3^+$};
\draw (48.5,35.5) -- (49.5,35.5) -- (50.5,35.5) -- (51.5,35.5) -- (52.5,35.5) -- (53.5,35.5) -- (53.5,36.5) -- (53.5,37.5) -- (53.5,38.5) -- (54.5,38.5) -- (55.5,38.5) -- (55.5,39.5) -- (55.5,40.5) -- (55.5,41.5) -- (56.5,41.5) -- (57.5,41.5) -- (57.5,42.5) -- (57.5,43.5) -- (57.5,44.5) -- (57.5,45.5) -- (57.5,46.5) -- (56.5,46.5) -- (55.5,46.5) -- (54.5,46.5) -- (54.5,47.5) -- (54.5,48.5) -- (53.5,48.5) -- (52.5,48.5) -- (51.5,48.5) -- (51.5,49.5) -- (51.5,50.5) -- (50.5,50.5) -- (49.5,50.5) -- (48.5,50.5) -- (47.5,50.5) -- (46.5,50.5) -- (46.5,49.5) -- (46.5,48.5) -- (46.5,47.5) -- (45.5,47.5) -- (44.5,47.5) -- (44.5,46.5) -- (44.5,45.5) -- (44.5,44.5) -- (43.5,44.5) -- (42.5,44.5) -- (42.5,43.5) -- (42.5,42.5) -- (42.5,41.5) -- (42.5,40.5) -- (42.5,39.5) -- (43.5,39.5) -- (44.5,39.5) -- (45.5,39.5) -- (45.5,38.5) -- (45.5,37.5) -- (46.5,37.5) -- (47.5,37.5) -- (48.5,37.5) -- (48.5,36.5) -- (48.5,35.5);
\draw (41.5,47.5) -- (42.5,47.5) -- (43.5,47.5) -- (44.5,47.5) -- (45.5,47.5) -- (46.5,47.5) -- (46.5,48.5) -- (46.5,49.5) -- (46.5,50.5) -- (47.5,50.5) -- (48.5,50.5) -- (48.5,51.5) -- (48.5,52.5) -- (48.5,53.5) -- (48.5,54.5) -- (48.5,55.5) -- (47.5,55.5) -- (46.5,55.5) -- (45.5,55.5) -- (45.5,56.5) -- (45.5,57.5) -- (44.5,57.5) -- (43.5,57.5) -- (42.5,57.5) -- (41.5,57.5) -- (40.5,57.5) -- (40.5,56.5) -- (40.5,55.5) -- (40.5,54.5) -- (39.5,54.5) -- (38.5,54.5) -- (38.5,53.5) -- (38.5,52.5) -- (38.5,51.5) -- (38.5,50.5) -- (38.5,49.5) -- (39.5,49.5) -- (40.5,49.5) -- (41.5,49.5) -- (41.5,48.5) -- (41.5,47.5);
\draw (43.500000,52.500000) node {$T_2^-$};
\draw (35.5,33.5) -- (36.5,33.5) -- (37.5,33.5) -- (38.5,33.5) -- (38.5,34.5) -- (38.5,35.5) -- (38.5,36.5) -- (37.5,36.5) -- (36.5,36.5) -- (35.5,36.5) -- (35.5,35.5) -- (35.5,34.5) -- (35.5,33.5);
\draw (37.000000,35.000000) node {$T_3^-$};
\end{scope}
\end{tikzpicture}
    &
    \begin{tikzpicture}[x=0.022222\textwidth,y=0.022222\textwidth,font=\tiny]
\begin{scope}[draw=gray]
\draw (29.5,26.5) -- (30.5,26.5) -- (31.5,26.5) -- (32.5,26.5) -- (33.5,26.5) -- (34.5,26.5) -- (34.5,27.5) -- (34.5,28.5) -- (34.5,29.5) -- (35.5,29.5) -- (36.5,29.5) -- (36.5,30.5) -- (36.5,31.5) -- (36.5,32.5) -- (36.5,33.5) -- (36.5,34.5) -- (35.5,34.5) -- (34.5,34.5) -- (33.5,34.5) -- (33.5,35.5) -- (33.5,36.5) -- (32.5,36.5) -- (31.5,36.5) -- (30.5,36.5) -- (29.5,36.5) -- (28.5,36.5) -- (28.5,35.5) -- (28.5,34.5) -- (28.5,33.5) -- (27.5,33.5) -- (26.5,33.5) -- (26.5,32.5) -- (26.5,31.5) -- (26.5,30.5) -- (26.5,29.5) -- (26.5,28.5) -- (27.5,28.5) -- (28.5,28.5) -- (29.5,28.5) -- (29.5,27.5) -- (29.5,26.5);
\draw (31.500000,31.500000) node {$S_4^+$};
\draw (35.5,34.5) -- (36.5,34.5) -- (37.5,34.5) -- (38.5,34.5) -- (38.5,35.5) -- (38.5,36.5) -- (38.5,37.5) -- (37.5,37.5) -- (36.5,37.5) -- (35.5,37.5) -- (35.5,36.5) -- (35.5,35.5) -- (35.5,34.5);
\draw (37.000000,36.000000) node {$S_2^+$};
\draw (23.5,33.5) -- (24.5,33.5) -- (25.5,33.5) -- (25.5,34.5) -- (25.5,35.5) -- (24.5,35.5) -- (23.5,35.5) -- (23.5,34.5) -- (23.5,33.5);
\draw (24.500000,34.500000) node {$S_3^+$};
\draw (28.5,31.5) -- (29.5,31.5) -- (30.5,31.5) -- (31.5,31.5) -- (32.5,31.5) -- (33.5,31.5) -- (33.5,32.5) -- (33.5,33.5) -- (33.5,34.5) -- (34.5,34.5) -- (35.5,34.5) -- (35.5,35.5) -- (35.5,36.5) -- (35.5,37.5) -- (35.5,38.5) -- (35.5,39.5) -- (34.5,39.5) -- (33.5,39.5) -- (32.5,39.5) -- (32.5,40.5) -- (32.5,41.5) -- (31.5,41.5) -- (30.5,41.5) -- (29.5,41.5) -- (28.5,41.5) -- (27.5,41.5) -- (27.5,40.5) -- (27.5,39.5) -- (27.5,38.5) -- (26.5,38.5) -- (25.5,38.5) -- (25.5,37.5) -- (25.5,36.5) -- (25.5,35.5) -- (25.5,34.5) -- (25.5,33.5) -- (26.5,33.5) -- (27.5,33.5) -- (28.5,33.5) -- (28.5,32.5) -- (28.5,31.5);
\draw (30.500000,36.500000) node {$S_4^-$};
\draw (23.5,30.5) -- (24.5,30.5) -- (25.5,30.5) -- (26.5,30.5) -- (26.5,31.5) -- (26.5,32.5) -- (26.5,33.5) -- (25.5,33.5) -- (24.5,33.5) -- (23.5,33.5) -- (23.5,32.5) -- (23.5,31.5) -- (23.5,30.5);
\draw (25.000000,32.000000) node {$S_2^-$};
\draw (36.5,32.5) -- (37.5,32.5) -- (38.5,32.5) -- (38.5,33.5) -- (38.5,34.5) -- (37.5,34.5) -- (36.5,34.5) -- (36.5,33.5) -- (36.5,32.5);
\draw (37.500000,33.500000) node {$S_3^-$};
\draw (30.5,21.5) -- (31.5,21.5) -- (32.5,21.5) -- (33.5,21.5) -- (34.5,21.5) -- (35.5,21.5) -- (35.5,22.5) -- (35.5,23.5) -- (35.5,24.5) -- (36.5,24.5) -- (37.5,24.5) -- (37.5,25.5) -- (37.5,26.5) -- (37.5,27.5) -- (37.5,28.5) -- (37.5,29.5) -- (36.5,29.5) -- (35.5,29.5) -- (34.5,29.5) -- (34.5,30.5) -- (34.5,31.5) -- (33.5,31.5) -- (32.5,31.5) -- (31.5,31.5) -- (30.5,31.5) -- (29.5,31.5) -- (29.5,30.5) -- (29.5,29.5) -- (29.5,28.5) -- (28.5,28.5) -- (27.5,28.5) -- (27.5,27.5) -- (27.5,26.5) -- (27.5,25.5) -- (27.5,24.5) -- (27.5,23.5) -- (28.5,23.5) -- (29.5,23.5) -- (30.5,23.5) -- (30.5,22.5) -- (30.5,21.5);
\draw (32.500000,26.500000) node {$Q_4^+$};
\draw (36.5,29.5) -- (37.5,29.5) -- (38.5,29.5) -- (39.5,29.5) -- (39.5,30.5) -- (39.5,31.5) -- (39.5,32.5) -- (38.5,32.5) -- (37.5,32.5) -- (36.5,32.5) -- (36.5,31.5) -- (36.5,30.5) -- (36.5,29.5);
\draw (38.000000,31.000000) node {$Q_2^+$};
\draw (24.5,28.5) -- (25.5,28.5) -- (26.5,28.5) -- (26.5,29.5) -- (26.5,30.5) -- (25.5,30.5) -- (24.5,30.5) -- (24.5,29.5) -- (24.5,28.5);
\draw (25.500000,29.500000) node {$Q_3^+$};
\draw (24.5,25.5) -- (25.5,25.5) -- (26.5,25.5) -- (27.5,25.5) -- (27.5,26.5) -- (27.5,27.5) -- (27.5,28.5) -- (26.5,28.5) -- (25.5,28.5) -- (24.5,28.5) -- (24.5,27.5) -- (24.5,26.5) -- (24.5,25.5);
\draw (26.000000,27.000000) node {$Q_2^-$};
\draw (37.5,27.5) -- (38.5,27.5) -- (39.5,27.5) -- (39.5,28.5) -- (39.5,29.5) -- (38.5,29.5) -- (37.5,29.5) -- (37.5,28.5) -- (37.5,27.5);
\draw (38.500000,28.500000) node {$Q_3^-$};
\end{scope}\begin{scope}[thick]
\draw (30.5,21.5) -- (31.5,21.5) -- (32.5,21.5) -- (33.5,21.5) -- (34.5,21.5) -- (35.5,21.5) -- (35.5,22.5) -- (35.5,23.5) -- (35.5,24.5) -- (36.5,24.5) -- (37.5,24.5) -- (37.5,25.5) -- (37.5,26.5) -- (37.5,27.5) -- (38.5,27.5) -- (39.5,27.5) -- (39.5,28.5) -- (39.5,29.5) -- (39.5,30.5) -- (39.5,31.5) -- (39.5,32.5) -- (38.5,32.5) -- (37.5,32.5) -- (36.5,32.5) -- (36.5,33.5) -- (36.5,34.5) -- (35.5,34.5) -- (34.5,34.5) -- (33.5,34.5) -- (33.5,35.5) -- (33.5,36.5) -- (32.5,36.5) -- (31.5,36.5) -- (30.5,36.5) -- (29.5,36.5) -- (28.5,36.5) -- (28.5,35.5) -- (28.5,34.5) -- (28.5,33.5) -- (27.5,33.5) -- (26.5,33.5) -- (26.5,32.5) -- (26.5,31.5) -- (26.5,30.5) -- (25.5,30.5) -- (24.5,30.5) -- (24.5,29.5) -- (24.5,28.5) -- (24.5,27.5) -- (24.5,26.5) -- (24.5,25.5) -- (25.5,25.5) -- (26.5,25.5) -- (27.5,25.5) -- (27.5,24.5) -- (27.5,23.5) -- (28.5,23.5) -- (29.5,23.5) -- (30.5,23.5) -- (30.5,22.5) -- (30.5,21.5);
\draw (38.5,32.5) -- (39.5,32.5) -- (40.5,32.5) -- (41.5,32.5) -- (41.5,33.5) -- (41.5,34.5) -- (41.5,35.5) -- (40.5,35.5) -- (39.5,35.5) -- (38.5,35.5) -- (38.5,34.5) -- (38.5,33.5) -- (38.5,32.5);
\draw (40.000000,34.000000) node {$T_2^+$};
\draw (21.5,30.5) -- (22.5,30.5) -- (23.5,30.5) -- (23.5,31.5) -- (23.5,32.5) -- (22.5,32.5) -- (21.5,32.5) -- (21.5,31.5) -- (21.5,30.5);
\draw (22.500000,31.500000) node {$T_3^+$};
\draw (29.5,26.5) -- (30.5,26.5) -- (31.5,26.5) -- (32.5,26.5) -- (33.5,26.5) -- (34.5,26.5) -- (34.5,27.5) -- (34.5,28.5) -- (34.5,29.5) -- (35.5,29.5) -- (36.5,29.5) -- (36.5,30.5) -- (36.5,31.5) -- (36.5,32.5) -- (37.5,32.5) -- (38.5,32.5) -- (38.5,33.5) -- (38.5,34.5) -- (38.5,35.5) -- (38.5,36.5) -- (38.5,37.5) -- (37.5,37.5) -- (36.5,37.5) -- (35.5,37.5) -- (35.5,38.5) -- (35.5,39.5) -- (34.5,39.5) -- (33.5,39.5) -- (32.5,39.5) -- (32.5,40.5) -- (32.5,41.5) -- (31.5,41.5) -- (30.5,41.5) -- (29.5,41.5) -- (28.5,41.5) -- (27.5,41.5) -- (27.5,40.5) -- (27.5,39.5) -- (27.5,38.5) -- (26.5,38.5) -- (25.5,38.5) -- (25.5,37.5) -- (25.5,36.5) -- (25.5,35.5) -- (24.5,35.5) -- (23.5,35.5) -- (23.5,34.5) -- (23.5,33.5) -- (23.5,32.5) -- (23.5,31.5) -- (23.5,30.5) -- (24.5,30.5) -- (25.5,30.5) -- (26.5,30.5) -- (26.5,29.5) -- (26.5,28.5) -- (27.5,28.5) -- (28.5,28.5) -- (29.5,28.5) -- (29.5,27.5) -- (29.5,26.5);
\draw (21.5,27.5) -- (22.5,27.5) -- (23.5,27.5) -- (24.5,27.5) -- (24.5,28.5) -- (24.5,29.5) -- (24.5,30.5) -- (23.5,30.5) -- (22.5,30.5) -- (21.5,30.5) -- (21.5,29.5) -- (21.5,28.5) -- (21.5,27.5);
\draw (23.000000,29.000000) node {$T_2^-$};
\draw (39.5,30.5) -- (40.5,30.5) -- (41.5,30.5) -- (41.5,31.5) -- (41.5,32.5) -- (40.5,32.5) -- (39.5,32.5) -- (39.5,31.5) -- (39.5,30.5);
\draw (40.500000,31.500000) node {$T_3^-$};
\end{scope}
\end{tikzpicture}
  \end{tabular}
  \caption{The six possible internal tangency structures of the double decomposition.  The relative sizes of the parents of the largest parent determine how the subtiles of the largest parents interact with the other parents. Boldfaced lines are boundaries of the first decomposition, while light lines are boundaries of subtiles in the double decomposition.  Overlap in the decompositions makes it a little tricky to visually parse some of these cases; it is helpful to keep in mind that each subtile is centrally symmetric, and to know that its label is drawn here at its center.}
  \label{f.sixdecompositions}
\end{figure}
\clearpage
}

We now complete the proof of \lref{tiling} with the following.
\begin{proof}[Proof of Claim \ref{c.applytop}]
We claim that the graph $(\tilde \ttt_L,\tilde \eee)$ satisfies the hypotheses of \lref{topological}.  Hypothesis \ref{P.graph} is immediate, as is the first part of hypothesis \ref{P.periodic}.  For the second part  of hypothesis \ref{P.periodic}, observe that (using \eref{soddy}) we have
\[
\sum_{T\in \tilde \ttt_L/L}\abs{T}=2c_1+2c_2+2c_3-c_4=c_0=\abs{\det \Lambda_{C_0}}
\]
It remains to verify hypotheses \ref{P.ints} and \ref{P.threes}.  We will need the following:
\begin{claim}
\label{c.isolatetwin}
If $s_x\notin T_1^-\cup T_1^+$, then $s_x\cap \foot(Q_4^-)=\varnothing$, unless $\abs{T_3^-}=\abs{T_2^+}=1$.
\end{claim}
\begin{proof}
By \eref{tarea}, $T_1^+\stm Q_4^-$ and $T_1^-\stm S_4^+$ are both tiles, which each touch $T_4:=Q_4^-=S_4^+$.  
Thus neither of these tiles shares any squares with $S_4^+$, and unless $\abs{T_3^-}=\abs{T_2^+}=1$, Claim \ref{c.doubledecomp} or inductively, by \lref{tiling}, give that the intersection of $T_1^+\stm Q_4^-$ and $T_1^-\stm S_4^+$ contains edges of $\Z^2$.  Thus, the union $\sss\sbs \Z^2$ of these tiles satisfies that its dual $\sss^*$ is connected; since $\sss$ is centrally symmetric about the center of $S_4^+$, we have as a consequence that the squares of $\sss$ surround $T_4$.  In particular, no square of $S_4^+$ can intersect any square outside of $\sss$.    
\end{proof}

Consider now \hyref{ints} of \lref{topological}.  For any pair $\tilde T(u),\tilde T(v)$ in $\tilde \eee$, we have from the definition of $G$ that the pair $T(u),T(v)$ are drawn adjacently or with overlap in \fref{righttile}.  Claims \ref{c.outtouch} and \ref{c.doubledecomp} now imply that $\foot(T(u))\cap \foot(T(v))$ contains at least two vertices.  Finally, Claim \ref{c.isolatetwin} implies that $\foot(\tilde T(u))\cap \foot(\tilde T(v))=\foot(T(u))\cap \foot(T(v))$, unless, up to symmetry, we have $T(u)=T_1^+$, $T(v)=T_1^-$.  In this case, however, Claim \ref{c.doubledecomp} implies directly that $\foot(\tilde T(u))\cap \foot(\tilde T(v))$ contains at least 2 vertices.

To check hypothesis \ref{P.threes} for the graph $(\tilde \ttt,\tilde \eee)$, we warm up by examining this for the graph $(\ttt,\eee)$, where it also holds.   We choose an assignment of the points $\rho(F)$ for faces $F$ of $G$.  For any face $F=\{u,v,w\}$ whose three corresponding tiles form a triple covered by Claim \ref{c.outtouch}, we must choose $\rho(F)\in \Ga$ to be the unique point in the three-way intersection of the footprints of the tiles.  Remaining faces $F$ are those whose corresponding tile triple lies entirely within $T_0$: i.e., the triple is either $(T_1^\pm,T_2^\pm,T_3^\mp)$ or $(T_1^\pm,T_1^\mp,T_2^\pm)$.  In the first case, we use Claim \ref{c.doubledecomp} to choose (without loss of generality) $\rho(\{T_1^+,T_2^+,T_3^-\})$ to be the unique point in the intersection of the touching triple $Q_2^+,T_2^+,T_3^-$.  In the second case, we use Claim \ref{c.doubledecomp} to choose (without loss of generality) $\rho(\{T_1^+,T_1^-,T_2^+\})$ to be the unique point in the intersection of the touching triple $Q_2^+,T_2^+,S_3^-$.

With this choice for the $\rho(F)$'s, we see that for any adjacent $F,F'$ corresponding to tiles from \fref{righttile} \emph{other} than the pair of faces corresponding to the tile triples
\begin{equation}
\label{e.overlapFs}
\{T_1^+,T_1^-,T_2^+\},\quad \{T_1^+,T_1^-,T_2^-\},
\end{equation}
that the points $\rho(F),\rho(F')$ are joined by a simple path given as the intersection of a single pair of tiles $\foot(U^1)\cap \foot(U^2)$ known to touch either by Claim \ref{c.outtouch} or Claim \ref{c.doubledecomp}. This would verify hypothesis \ref{P.threes} of \lref{topological} in the graph $(\ttt,\eee)$ for these pairs of faces.  Note now that Claim \ref{c.isolatetwin} implies that the points $\rho(F)$ selected above also lie in the three way intersections $\foot(\tilde T(u))\cap \foot(\tilde T(v))\cap \foot(\tilde T(w))$ for $\{u,v,w\}=F$, and, moreover, that the simple paths in the intersections of the footprints of pairs of tiles $T(u)$, $T(v)$ used above remain in the intersection $\foot(\tilde T(u))\cap \foot(\tilde T(v))$.  In particular, we have verified that with this assignment of $\rho(F)$'s, hypothesis \ref{P.threes} holds also in the graph $(\tilde \ttt,\tilde \eee)$ for all faces \emph{other} than those corresponding to translates of triples 
\[
\{T_1^+\stm Q_4^-,T_1^-,T_2^+\},\quad \{T_1^+\stm Q_4^-,T_1^-,T_2^-\},
\]
and we deal with this final case separately, now.  

Let $L^+$ be the set of edges of $\partial Q_4^-$ whose two incident squares both lie in $T_1^+$, and define $L^-$ similarly.  Note that $L^-$ is the central reflection through $\cent(Q_4^-)$ of $L^+$, and that Claim \ref{c.isolatetwin} implies that $L^+\cup L^-$ equals the edge-set of $\partial Q_4^-$.  We claim that $L^+$ is the edge-set of a path, which suffices for us, since then, letting $V(L^+)$ denote the set of vertices incident with edges in $L^+$, we have that $\foot(S_3^-)\cap \foot(S_4^+)\cap \foot(Q_2^+)\sbs V(L^+)\cap V(L^-)$ and  $\foot(S_2^-)\cap \foot(S_4^+)\cap \foot(Q_3^+)\sbs V(L^+)\cap V(L^-)$, so that
\[
(S_3^-\cap Q_2^+)\cup V(L^+)\cup (Q_3^+\cap S_2^-)
\]
lies in the intersection $(T_1^+\stm Q_4^-)\cap T_1^-$, and contains a path joining the points assigned to the triples from \eref{overlapFs}.

To see that this $L^+$ is indeed the edge-set of a path, suppose it is false.    In particular, we can, in cyclic order in $\partial Q_4^-$, find edges $e_1$, $e_2$, $e_3$, $e_4$ such that $e_1,e_3\in L^+$ and $e_2,e_4\in L^-\stm L^+$.  But  then since $T_1^+\stm Q_4^-$ is connected, there is a path of squares in $T_1^+\stm Q_4^-$ from $e_1$ to $e_3$.  But, together with the squares of $Q_4^-$, these squares must enclose either $s_2$ or $s_4$, where $s_i$ is the square incident with $e_i$ lying outside of $Q_4^-$.  In particular, $T_1^+\stm Q_4^-$ cannot be simply connected, contradicting the inductive hypothesis \eref{tarea}.  

Thus the hypotheses of \lref{topological} are satisfied for $(\tilde \ttt,\tilde \eee)$.  
\end{proof}

\begin{remark}\label{r.topologyvscases}
  We emphasize that our proof of Claim \ref{c.applytop} did not work by showing that the intersections of adjacent tiles in $\tilde \ttt_L$ are all simple paths.  In particular, for the tile pair $(T_1^+\stm Q_4^-,T_1^-)$, the proof only establishes that the intersection contains a path.  The problem with using induction to establish that the intersection equals a path for this pair is that the tile relationships in the double decomposition (e.g., $S_3^-,S_4^+,Q_2^+$, etc.) are not fixed, and depend on the relative order of the sizes of the circles $C_1,C_2,C_3,C_4$.   Instead, the proof given above uses an ad-hoc argument for this pair based on Claim \ref{c.isolatetwin}.  In principle, one could prove a stronger version of Claim \ref{c.applytop} by consider several cases according to the relative sizes of the circles, and accounting for the various subtile-relationships that arise in each case.  This would allow one to apply a more straigthforward topological lemma at the expense of a blow-up in the tile analysis.
\end{remark}

\section{Boundary Strings}
\label{s.strings}

Having constructed tiles, we now turn our attention towards constructing odometers.  The main idea is to mimic the tile construction, using the duality between the rules for the $v_{ij}$ and $a_{ij}$ in \eref{lattice} to attach superharmonic data to the tiles.  However, there is a problem in directly lifting the tile construction:  the odometers are only $180^\circ$ symmetric in general, and we used $90^\circ$ symmetry in constructing the tiles.

An examination of the argument in \sref{tiles} suggests that we made essential use of $90^\circ$ symmetry only to conclude that the lattice $\Lambda_C$ generated by $\{ v_{10}, v_{20}, v_{30} \}$ is a tiling lattice for $T_0$ from the fact that $\I \Lambda_C$ is a tiling lattice for $T_0$ (Remark \ref{rem.90}).  To give a proof using only $180^\circ$ symmetry, we must therefore find a way to express the interface between $T_0$ and $T_0 + v_{i0}$ directly in terms of relationships between subtiles.

\begin{figure}[t]
\begin{center}
\nofig{\input{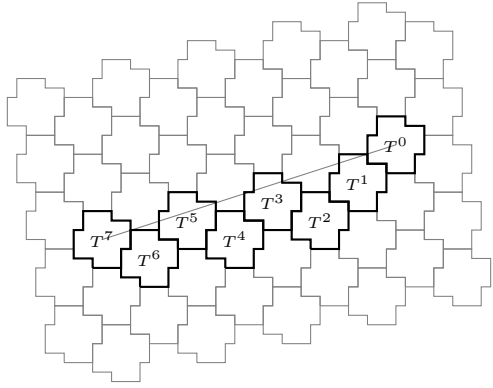}}
\end{center}
\caption{A $\ttt$-string $T^0, \dots, T^7$ in the regular tiling $\ttt=T_C + \Lambda_C$ for the circle $C = (28,7+20\I)$.}
\label{f.string}
\end{figure}

We need a new type of tile decomposition, which we call {\em boundary strings}. If $\ttt$ is a regular tiling, then we call a sequence $T^0, T^1, \dots, T^n$ of tiles in $\ttt$ the $\ttt$-\emph{string} (or simply: \emph{string}) from $T^0$ to $T^n$ if:
\begin{enumerate}
\item $T^i$ touches $T^{i+1}$ for $0 \leq i < n$,
\item Each $\cent(T^i)$ $(0<i<n)$ lies in the closed half-plane to the left of the ray from $\cent(T^0)$ to $\cent(T^n)$; i.e.:
\[
\Im\left(\frac{\cent(T^i)-\cent(T^0)}{\cent(T^n) - \cent(T^0)}\right) \geq 0
\]
\item Each $T^i$ $(0<i<n)$ touches some $S^i\in \ttt$  whose centroid lies outside of the closed half-plane to the left of the ray from $\cent(T^0)$ to $\cent(T^n)$.
\end{enumerate}
(See \fref{string}.)
A $\ttt$-string is the left-handed approximation of the line segment from $\cent(T^0)$ to $\cent(T^n)$ in $\ttt$; see \fref{string}.  It is not hard to show from the definition that there is a unique $\ttt$-string between any two tiles in a regular tiling (indeed, the partial sequences $T^0,T^1,\dots,T^j$ are uniquely  determined, inductively, given $T_0$ and the ray from $\cent(T^0)$ towards $\cent(T^n)$).   The \emph{interior tiles} of the $\ttt$-string $\sss$ from $T^0$ to $T^n$ are the tiles in the $\ttt$-string other than the endpoints $T^0,T^n$, and the \emph{interior} of the $\ttt$-string is the union of the footprints of all interior tiles.  

Given a tile $T_0$ for the circle $C_0$ in the proper Descartes quadruple $(C_0,C_1,C_2,C_3)$, the $C_i$ \emph{boundary-string} $(i=1,2,3)$ for the tile $T_0$ is the string from the tiles $R_i^-$ to $R_i^+$ for the tiling associated to $C_i$, where 
\begin{equation*}
\cent(R_i^\pm) = \cent(T_0) + \tfrac{1}{2} (v_{i0} \pm v_{0i}).
\end{equation*}

The following lemma shows that the boundary strings for a tile can be constructed by concatenating certain smaller strings together.  

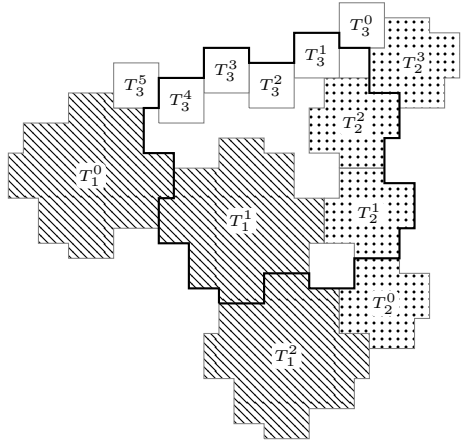
\begin{figure}[t]
\begin{center}
\nofig{\usetikzlibrary{patterns}
\begin{tikzpicture}[font=\tiny,scale=1/5]
\definecolor{myred}{rgb}{1,0.95,0.95}
\definecolor{mygreen}{rgb}{0.95,1,0.95}
\definecolor{myblue}{rgb}{0.95,0.95,1}
\begin{scope}[draw=gray]
\draw[pattern=north west lines] (14.5,22.5) -- (15.5,22.5) -- (16.5,22.5) -- (17.5,22.5) -- (17.5,23.5) -- (17.5,24.5) -- (18.5,24.5) -- (19.5,24.5) -- (20.5,24.5) -- (20.5,25.5) -- (20.5,26.5) -- (21.5,26.5) -- (21.5,27.5) -- (21.5,28.5) -- (21.5,29.5) -- (20.5,29.5) -- (19.5,29.5) -- (19.5,30.5) -- (19.5,31.5) -- (19.5,32.5) -- (18.5,32.5) -- (17.5,32.5) -- (17.5,33.5) -- (16.5,33.5) -- (15.5,33.5) -- (14.5,33.5) -- (14.5,32.5) -- (14.5,31.5) -- (13.5,31.5) -- (12.5,31.5) -- (11.5,31.5) -- (11.5,30.5) -- (11.5,29.5) -- (10.5,29.5) -- (10.5,28.5) -- (10.5,27.5) -- (10.5,26.5) -- (11.5,26.5) -- (12.5,26.5) -- (12.5,25.5) -- (12.5,24.5) -- (12.5,23.5) -- (13.5,23.5) -- (14.5,23.5) -- (14.5,22.5);
\node [fill=white,rounded corners=2pt,inner sep=1pt] at (16.000000,28.000000) {$T_1^{0}$};
\draw[pattern=north west lines] (24.5,19.5) -- (25.5,19.5) -- (26.5,19.5) -- (27.5,19.5) -- (27.5,20.5) -- (27.5,21.5) -- (28.5,21.5) -- (29.5,21.5) -- (30.5,21.5) -- (30.5,22.5) -- (30.5,23.5) -- (31.5,23.5) -- (31.5,24.5) -- (31.5,25.5) -- (31.5,26.5) -- (30.5,26.5) -- (29.5,26.5) -- (29.5,27.5) -- (29.5,28.5) -- (29.5,29.5) -- (28.5,29.5) -- (27.5,29.5) -- (27.5,30.5) -- (26.5,30.5) -- (25.5,30.5) -- (24.5,30.5) -- (24.5,29.5) -- (24.5,28.5) -- (23.5,28.5) -- (22.5,28.5) -- (21.5,28.5) -- (21.5,27.5) -- (21.5,26.5) -- (20.5,26.5) -- (20.5,25.5) -- (20.5,24.5) -- (20.5,23.5) -- (21.5,23.5) -- (22.5,23.5) -- (22.5,22.5) -- (22.5,21.5) -- (22.5,20.5) -- (23.5,20.5) -- (24.5,20.5) -- (24.5,19.5);
\node [fill=white,rounded corners=2pt,inner sep=1pt] at (26.000000,25.000000) {$T_1^{1}$};
\draw[pattern=north west lines] (27.5,10.5) -- (28.5,10.5) -- (29.5,10.5) -- (30.5,10.5) -- (30.5,11.5) -- (30.5,12.5) -- (31.5,12.5) -- (32.5,12.5) -- (33.5,12.5) -- (33.5,13.5) -- (33.5,14.5) -- (34.5,14.5) -- (34.5,15.5) -- (34.5,16.5) -- (34.5,17.5) -- (33.5,17.5) -- (32.5,17.5) -- (32.5,18.5) -- (32.5,19.5) -- (32.5,20.5) -- (31.5,20.5) -- (30.5,20.5) -- (30.5,21.5) -- (29.5,21.5) -- (28.5,21.5) -- (27.5,21.5) -- (27.5,20.5) -- (27.5,19.5) -- (26.5,19.5) -- (25.5,19.5) -- (24.5,19.5) -- (24.5,18.5) -- (24.5,17.5) -- (23.5,17.5) -- (23.5,16.5) -- (23.5,15.5) -- (23.5,14.5) -- (24.5,14.5) -- (25.5,14.5) -- (25.5,13.5) -- (25.5,12.5) -- (25.5,11.5) -- (26.5,11.5) -- (27.5,11.5) -- (27.5,10.5);
\node [fill=white,rounded corners=2pt,inner sep=1pt] at (29.000000,16.000000) {$T_1^{2}$};
\draw[pattern=dots] (34.5,16.5) -- (35.5,16.5) -- (36.5,16.5) -- (37.5,16.5) -- (37.5,17.5) -- (37.5,18.5) -- (38.5,18.5) -- (38.5,19.5) -- (38.5,20.5) -- (38.5,21.5) -- (37.5,21.5) -- (36.5,21.5) -- (36.5,22.5) -- (35.5,22.5) -- (34.5,22.5) -- (33.5,22.5) -- (33.5,21.5) -- (33.5,20.5) -- (32.5,20.5) -- (32.5,19.5) -- (32.5,18.5) -- (32.5,17.5) -- (33.5,17.5) -- (34.5,17.5) -- (34.5,16.5);
\node [fill=white,rounded corners=2pt,inner sep=1pt] at (35.500000,19.500000) {$T_2^{0}$};
\draw[pattern=dots] (33.5,22.5) -- (34.5,22.5) -- (35.5,22.5) -- (36.5,22.5) -- (36.5,23.5) -- (36.5,24.5) -- (37.5,24.5) -- (37.5,25.5) -- (37.5,26.5) -- (37.5,27.5) -- (36.5,27.5) -- (35.5,27.5) -- (35.5,28.5) -- (34.5,28.5) -- (33.5,28.5) -- (32.5,28.5) -- (32.5,27.5) -- (32.5,26.5) -- (31.5,26.5) -- (31.5,25.5) -- (31.5,24.5) -- (31.5,23.5) -- (32.5,23.5) -- (33.5,23.5) -- (33.5,22.5);
\node [fill=white,rounded corners=2pt,inner sep=1pt] at (34.500000,25.500000) {$T_2^{1}$};
\draw[pattern=dots] (32.5,28.5) -- (33.5,28.5) -- (34.5,28.5) -- (35.5,28.5) -- (35.5,29.5) -- (35.5,30.5) -- (36.5,30.5) -- (36.5,31.5) -- (36.5,32.5) -- (36.5,33.5) -- (35.5,33.5) -- (34.5,33.5) -- (34.5,34.5) -- (33.5,34.5) -- (32.5,34.5) -- (31.5,34.5) -- (31.5,33.5) -- (31.5,32.5) -- (30.5,32.5) -- (30.5,31.5) -- (30.5,30.5) -- (30.5,29.5) -- (31.5,29.5) -- (32.5,29.5) -- (32.5,28.5);
\node [fill=white,rounded corners=2pt,inner sep=1pt] at (33.500000,31.500000) {$T_2^{2}$};
\draw[pattern=dots] (36.5,32.5) -- (37.5,32.5) -- (38.5,32.5) -- (39.5,32.5) -- (39.5,33.5) -- (39.5,34.5) -- (40.5,34.5) -- (40.5,35.5) -- (40.5,36.5) -- (40.5,37.5) -- (39.5,37.5) -- (38.5,37.5) -- (38.5,38.5) -- (37.5,38.5) -- (36.5,38.5) -- (35.5,38.5) -- (35.5,37.5) -- (35.5,36.5) -- (34.5,36.5) -- (34.5,35.5) -- (34.5,34.5) -- (34.5,33.5) -- (35.5,33.5) -- (36.5,33.5) -- (36.5,32.5);
\node [fill=white,rounded corners=2pt,inner sep=1pt] at (37.500000,35.500000) {$T_2^{3}$};
\draw[fill=white] (32.5,36.5) -- (33.5,36.5) -- (34.5,36.5) -- (35.5,36.5) -- (35.5,37.5) -- (35.5,38.5) -- (35.5,39.5) -- (34.5,39.5) -- (33.5,39.5) -- (32.5,39.5) -- (32.5,38.5) -- (32.5,37.5) -- (32.5,36.5);
\draw (34.000000,38.000000) node {$T_3^{0}$};
\draw[fill=white] (29.5,34.5) -- (30.5,34.5) -- (31.5,34.5) -- (32.5,34.5) -- (32.5,35.5) -- (32.5,36.5) -- (32.5,37.5) -- (31.5,37.5) -- (30.5,37.5) -- (29.5,37.5) -- (29.5,36.5) -- (29.5,35.5) -- (29.5,34.5);
\draw (31.000000,36.000000) node {$T_3^{1}$};
\draw[fill=white] (26.5,32.5) -- (27.5,32.5) -- (28.5,32.5) -- (29.5,32.5) -- (29.5,33.5) -- (29.5,34.5) -- (29.5,35.5) -- (28.5,35.5) -- (27.5,35.5) -- (26.5,35.5) -- (26.5,34.5) -- (26.5,33.5) -- (26.5,32.5);
\draw (28.000000,34.000000) node {$T_3^{2}$};
\draw[fill=white] (23.5,33.5) -- (24.5,33.5) -- (25.5,33.5) -- (26.5,33.5) -- (26.5,34.5) -- (26.5,35.5) -- (26.5,36.5) -- (25.5,36.5) -- (24.5,36.5) -- (23.5,36.5) -- (23.5,35.5) -- (23.5,34.5) -- (23.5,33.5);
\draw (25.000000,35.000000) node {$T_3^{3}$};
\draw[fill=white] (20.5,31.5) -- (21.5,31.5) -- (22.5,31.5) -- (23.5,31.5) -- (23.5,32.5) -- (23.5,33.5) -- (23.5,34.5) -- (22.5,34.5) -- (21.5,34.5) -- (20.5,34.5) -- (20.5,33.5) -- (20.5,32.5) -- (20.5,31.5);
\draw (22.000000,33.000000) node {$T_3^{4}$};
\draw[fill=white] (17.5,32.5) -- (18.5,32.5) -- (19.5,32.5) -- (20.5,32.5) -- (20.5,33.5) -- (20.5,34.5) -- (20.5,35.5) -- (19.5,35.5) -- (18.5,35.5) -- (17.5,35.5) -- (17.5,34.5) -- (17.5,33.5) -- (17.5,32.5);
\draw (19.000000,34.000000) node {$T_3^{5}$};

\end{scope}\begin{scope}[thick]
\draw (24.5,19.5) -- (25.5,19.5) -- (26.5,19.5) -- (27.5,19.5) -- (27.5,20.5) -- (27.5,21.5) -- (28.5,21.5) -- (29.5,21.5) -- (30.5,21.5) -- (30.5,20.5) -- (31.5,20.5) -- (32.5,20.5) -- (33.5,20.5) -- (33.5,21.5) -- (33.5,22.5) -- (34.5,22.5) -- (35.5,22.5) -- (36.5,22.5) -- (36.5,23.5) -- (36.5,24.5) -- (37.5,24.5) -- (37.5,25.5) -- (37.5,26.5) -- (37.5,27.5) -- (36.5,27.5) -- (35.5,27.5) -- (35.5,28.5) -- (35.5,29.5) -- (35.5,30.5) -- (36.5,30.5) -- (36.5,31.5) -- (36.5,32.5) -- (36.5,33.5) -- (35.5,33.5) -- (34.5,33.5) -- (34.5,34.5) -- (34.5,35.5) -- (34.5,36.5) -- (33.5,36.5) -- (32.5,36.5) -- (32.5,37.5) -- (31.5,37.5) -- (30.5,37.5) -- (29.5,37.5) -- (29.5,36.5) -- (29.5,35.5) -- (28.5,35.5) -- (27.5,35.5) -- (26.5,35.5) -- (26.5,36.5) -- (25.5,36.5) -- (24.5,36.5) -- (23.5,36.5) -- (23.5,35.5) -- (23.5,34.5) -- (22.5,34.5) -- (21.5,34.5) -- (20.5,34.5) -- (20.5,33.5) -- (20.5,32.5) -- (19.5,32.5) -- (19.5,31.5) -- (19.5,30.5) -- (19.5,29.5) -- (20.5,29.5) -- (21.5,29.5) -- (21.5,28.5) -- (21.5,27.5) -- (21.5,26.5) -- (20.5,26.5) -- (20.5,25.5) -- (20.5,24.5) -- (20.5,23.5) -- (21.5,23.5) -- (22.5,23.5) -- (22.5,22.5) -- (22.5,21.5) -- (22.5,20.5) -- (23.5,20.5) -- (24.5,20.5) -- (24.5,19.5);
\end{scope}
\end{tikzpicture}}
\end{center}
\caption{The three boundary strings of a tile.}
\label{f.boundarystrings}
\end{figure}

\begin{lemma}
\label{l.boundarystring}
Let $(C_0,C_1,C_2,C_3) \in {\B}^4$ be a proper Descartes quadruple, write $v_{ij} = v(C_i,C_j)$, and suppose $T_0$ is a tile for $C_0$ with the tile decomposition $\{T_i^\pm\}$,  and suppose $i\in\{1, 2, 3\}$, $c_i>0$, and that $R_i^\pm$ is a tile for $C_i$ satisfying
\begin{equation*}
\cent(R_i^\pm) = \cent(T_0) + \tfrac{1}{2} (v_{i0} \pm v_{0i}).
\end{equation*}
Then we have that
\begin{equation}
\label{e.rioffset}
\cent(R_i^+)-\cent(T_i^-)=v_{ki},\quad \cent(T_i^-)-\cent(R_i^-)=-v_{ji},
\end{equation}
and the $(R_i^-+\Lambda_{C_i})$-string from $R_i^-$ to $R_i^+$ is the concatenation of the $(R_i^-+\Lambda_{C_i})$-strings from $R_i^-$ to $T_i^-$ and from $T_i^-$ to $R_i^+$, respectively.
\end{lemma}

Note that \eref{rioffset} implies via \lref{biglatticevectors} that the tiles $T_i^-$ and $R_i^+$ both lie in tiling $R_i^-+\Lambda_{C_i}$, ensuring that the referenced strings are well defined.

It is not hard at this point to use the lattice rules to strengthen \lref{boundarystring}, to show inductively that the strings from $R_i^-$ to $T_i^-$ and from $T_i^-$ to $R_i^+$ are themselves boundary strings of subtiles (when they have more than two tiles). This induction also gives, for example, that the interior tiles of a boundary string for $T_0$ lie in $T_0$, as seen in \fref{boundarystrings}.   We postpone this calculation until the next section, however, when we are prepared to simultaneously show that the strings have important compatibility properties with respect to our odometer construction.
\begin{proof}[Proof of \lref{boundarystring}]
The offsets in \eref{rioffset} result from a straightforward calculation using \eref{lattice}; for convenience, note that, referring to the same tile collection \eref{leftdecomp} used in the proof of \lref{tiling} (see \fref{righttile}) one can check that $R_i^+=R_i^1$, $R_i^-=R_i^0$ in that decomposition.

Examining the definition of a string, we see that the statement regarding the concatenation fails only if the interior of the triangle $\triangle \cent(T_i^-) \cent(R_i^-) \cent(R_i^+)$ contains the center of some tile in the tiling $T_i^- + \Lambda_{C_i}$.  However, the lattice generated by $-v_{ji}$ and $v_{ki}$ has determinant
\[
\frac 1 2 \left(\overline{v_{ik}} v_{ji} + v_{ik} \overline{v_{ji}}\right)=c_i
\]
by the lattice rules \eref{lattice}.  In particular, the triangle $\triangle \cent(T_i^-) \cent(R_i^-) \cent(R_i^+)$ has area $\frac 1 2 {c_i}$.  Thus the lemma follows from the fact that any triangle of area half the determinant of a lattice containing its three vertices can contain no other points of the lattice.  (This is a special case of Pick's theorem, for example.)
\end{proof}

The following topological lemma allows us to analyze tile interfaces -- and, in particular, odometer interfaces -- using only $180^\circ$ symmetry.
 
\begin{lemma}
\label{l.stringpair}
If $R$ and $S$ are tiles in the tiling $\ttt = T_0 + \Lambda_{C_0}$ corresponding to the circle $C_0\in \B$ with $c_0\geq 1$, then the intersection of the interiors of the $\ttt$-string $\rrr$ from $R$ to $S$ and the $\ttt$-string $\sss$ from $S$ to $R$ contains a path in $\Z[\I]$ from $R$ to $S$.
\end{lemma}

\begin{proof}
We consider the graph $G$ on $\ttt$ where $T,T'$ are adjacent if they touch.  $G$ is isomorphic to the graph of the triangular lattice unless $c_0=1$, in which case the lemma is easy to verify directly.  

For the former case, we draw $G$ in the plane by placing each vertex at the center of the corresponding tile, and drawing straight line segments between adjacent vertices of $G$; the result is a planar embedding of $G$ which is affine-equivalent to the equilateral triangle embedding of the triangular lattice.

We now consider the sequences $F^0,F^1,\dots,F^k$ of faces of $G$ through which the line segment from $\cent(R)$ to $\cent(S)$ passes, in order.  If we associate to each $F$ the point $\rho(F)$ in $\Ga$ lying in the intersection of the three tiles of $F$, then the consecutive faces $F^i$ and $F^{i+1}$ either share an edge corresponding to touching tiles $T\in \rrr,T'\in \sss$, or they share a vertex corresponding to a tile in both $\rrr$ and $\sss$; in either case, there is a path in $\Ga$ from $\rho(F)$ to $\rho(F')$ lying in the intersection of the interiors of $\sss$ and $\rrr$.  Concatenating these paths consecutively, we get a walk in $\Ga$ from $\rho(F^0)$ to $\rho(F^k)$, lying in the intersection of the interiors of $\rrr$ and $\sss$.
\end{proof}

\section{Tile odometers}
\label{s.odometers}

In this section we attach function data to our tiles.  Since we are no longer concerned with topological issues, our definition of a tile as a set of squares $s_x$ is no longer useful, and from here on \change{we identify a tile $T$ with the vertex set of its footprint graph, $\bigcup_{s_x \in T} s_x \subset \Ga$.}
In particular, $T\cap T'$ is now denotes a subset of $\Ga$ \change{rather than a set of squares}.

\subsection{Basic definitions}
A {\em partial odometer} is a function $h : T \to \Z$ with a finite domain $T \sbs \Z^2$. We write $T(h)$ for the domain of $h$, and $\slope(h) \in \C$ for the {\em slope} of $h$, which is the average of
\begin{multline}
\frac 1 2 \left(h(x+1)-h(x)+h(x+1+\I)-h(x+\I)\right)+\\
\frac{\I}{2}\left(h(x+\I)-h(x)+h(x+1+\I)-h(x+1)\right)
\end{multline}
over squares $\{x,x+1,x+\I,x+1+\I\}\sbs T$; this is a measure of an average gradient for $h$.  Note that the slope is not defined when $T$ is a singleton.  We say that two partial odometers $h_1$ and $h_2$ are {\em translations} of one another if
\begin{equation}
\label{e.translate}
T(h_1) = T(h_2) + v \quad \mbox{and} \quad h_1(x) = h_2(x + v) + a \cdot x + b,
\end{equation}
for some $v, a \in \Z^2$ and $b \in \Z$.  
\begin{definition}
\label{d.comp}
We say that two partial odometers $h_1$ and $h_2$ are \emph{compatible} if $h_2-h_1=c$ on $T(h_1) \cap T(h_2)$ for some constant $c$, which we call the \emph{offset constant} for the pair $(h_1,h_2)$, or if $T(h_1)\cap T(h_2)=\varnothing$. \end{definition}
\noindent  When the offset constant is 0, or in the second case, we write $h_1 \cup h_2$ for the common extension of the $h_i$ to the union of their domains.  The next lemma allows us to glue together pairwise compatible partial odometers. Recall we have defined a \emph{tiling} as a collection of tiles $\ttt$ such that every square $s_x = \{x,x+1,x+i,x+1+i\}$ of $\Z^2$ is contained in exactly one element of $\ttt$.

\begin{lemma}
\label{l.gluing}
If $\hhh$ is a collection of pairwise compatible partial odometers such that $\ttt = \{ T(h) : h \in \hhh \}$ is a hexagonal tiling, then there is a function $g : \Z^2 \to \Z$, unique up to adding a constant, which is compatible with every $h \in \hhh$.
\end{lemma}
\begin{proof}[Proof]
  Since $\ttt$ is a hexagonal tiling, its intersection graph $G$ is a planar triangulation, and each face $\{h_0,h_1,h_2\}$ of $G$ corresponds to a touching triple of tiles $T(h_i)$ with $T(h_0) \cap T(h_1) \cap T(h_2) \neq \varnothing$. In particular, writing $d(h_i,h_j)$ for the constant value of $h_i-h_j$ on $T(h_i) \cap T(h_j)$, we have $d(h_0,h_1)+d(h_1,h_2)+d(h_2,h_1)=0$.  It follows that the sum of $d$ over any cycle is zero and therefore that $d$ can be written the gradient of a vertex function $f:\ttt\to \Z$ which is unique up to additive constant.  Now fix any $h_0 \in \hhh$ and set \change{$g(x) = h_0(x-y) + f(T(h))$ on $x\in T(h)$, where $y\in \Z^2$ is the translation such that $T(h)=T(h_0)+y$.}
\end{proof}

Our goal is to associate a partial odometer $h$, unique up to odometer translation \eref{translate}, to every circle $C \in \B$.

\begin{definition}
\label{d.O0}
A partial odometer $h_0:T_0\to \Z$ is a \emph{tile odometer} for $C_0\in \B$ if $T_0$ is a tile for $C_0$ and either:
\begin{itemize}
\item $C_0=(0,\pm 1)$ (so $T_0$ is a singleton),
\item $C_0=(1,1+2z)$ for some $z\in \Ga$, and $h$ is any translation of the partial odometer $h':\{0,1,\I,1+\I\}\to \Z$ given by $h(0)=h(1)=h(\I)=0$, $h(1+\I)=\Im(z)/2$, or\\
\item $(C_0,C_1,C_2,C_3) \in {\B}^4$ is a proper Descartes quadruple, $C_i = (c_i, z_i)$, 
$a_{ij} = a(C_i,C_j)$, $v_{ij} = v(C_i,C_j)$, and
\begin{equation}
\label{e.odecomp}
h_0 = h_1^+ \cup h_1^- \cup h_2^+ \cup h_2^- \cup h_3^+ \cup h_3^-,
\end{equation}
where $h_i^\pm$ is a tile odometer for $C_i$ such that 
\begin{equation*}
\cent(T(h_i^\pm)) - \cent(T(h_0)) = \pm \tfrac{1}{2} (v_{kj} - \I v_{kj}),
\end{equation*}
and
\begin{equation*}
c_i = 0 \quad \mbox{or} \quad \slope(h_i^\pm) - \slope(h_0) = \pm \tfrac{1}{2} (a_{kj} + \I a_{kj}),
\end{equation*}
for all rotations $(i,j,k)$ of $(1,2,3)$.
\end{itemize}
We call the $h_i^\pm$'s the \emph{subodometers} of $h_0$. 
\end{definition}

When $T_1, ..., T_n$ is a sequence of tiles such $T_k$ is a subtile of $T_{k+1}$, we call $T_1$ an {\em ancestor tile} of $T_n$. Restrictions $h|T$ of odometers to ancestor tiles $T$ of $T(h)$ are called \emph{ancestor odometers} of $h_0$.  The following lemma asserts that this makes sense:

\begin{lemma}
\label{l.Oancestor}
If $h_0:T_0\to \Z$ is a tile odometer for $C_0$ and the tile $T$ for the circle $C\in \B$ is an ancestor of $T_0$, then $h_0|T$ is a tile odometer for $C$.\qed
\end{lemma}

By induction, tile odometers are easily seen to be centrally symmetric:

\begin{lemma}
\label{l.osym}
If $h$ is a tile odometer with domain $T$, then $x \mapsto h(-x)$ with domain $-T$ is a translation of $h$. $\qed$
\end{lemma}

Our main goal in this section is to prove inductively that each circle $C\in \B$ has an associated tile odometer.  (Note that, inductively, the definition immediately gives that tile odometers for a given circle are unique up to odometer translation.)  However, before proving that circles in $\B$ do have tile odometers, we will prove  that tile odometers must have certain compatibility properties when they do exist.

Given a proper Descartes quadruple $(C_0,C_1,C_2,C_3)$ with $c_0>0$, we say that tile odometers $h_0$ and $h_0'$ for $C_0$ are \emph{left-lattice adjacent} if 
\begin{align*}
\cent(T(h_0')) - \cent(T(h_0)) &=\pm v_{i0}\\
\slope(h_0') - \slope(h_0) &= \pm a_{i0}
\end{align*}
for some $i \in \{1,2,3\}$ (with matching signs, as usual) 
and similarly \emph{right-lattice adjacent} if
\begin{align*}
\cent(T(h_0')) - \cent(T(h_0)) &=\pm v_{0i}\\
\slope(h_0') - \slope(h_0) &= \pm a_{0i}.
\end{align*}
for some $i \in \{1,2,3\}$.
Given tile odometers $h_0$ and $h_i$ for $C_0$ and $C_i$ $(i=1,2,3)$, we say that $h_i$ is  \emph{subtile-lattice adjacent} to $h_0$ if 
\begin{align*}
\cent(T(h_i)) - \cent(T(h_0)) &= \I^s \tfrac{1}{2}  (v_{i0} + v_{0i})\\
\slope(h_i) - \slope(h_0) &= (-\I)^s\tfrac{1}{2} (a_{i0} + a_{0i}).
\end{align*}
for some $s \in \{0,1,2,3\}$.  Note that the subtile-lattice adjacency relationship is not symmetric.

\subsection{Outline of construction}
The essential difference between the induction in \sref{tiles} and the argument we are forced to carry out in this section is that the odometer analog of \eref{t90} \change{in \lref{tiling}}  is false; tile odometers are not $90^\circ$ symmetric in a straightforward way.   In particular, we cannot build a $\Lambda_C$-periodic global odometer using that inductive argument.  (Instead, that argument would give an $\I \Lambda_C$ periodic function, which does not have the correct growth to be an odometer for $C$.  These $\I\Lambda_C$ functions are interesting in their own right, however.)

Note that \dref{O0} is analogous to \dref{T0}, and that left-lattice adjacency and subtile-lattice adjacency as defined here correspond for the domains of $h_0$ and $h_t'$ $(t=0,1,2,3)$ to the two types of pairwise tile relationships which were seen in \change{\eref{ttile} and} \eref{ttouch1} of \lref{tiling}. However, as the proof of \eref{ttouch1} ends with an application of $90^\circ$ symmetry, the inductive argument of \sref{tiles} can actually be adapted to the cases of right-lattice adjacency and subtile-lattice adjacency of tile-odometers, but \emph{not} left-lattice adjacency.  This adaptation is carried out in \lref{induct-adjacency}, below, which shows that right-lattice adjacency and subtile-lattice adjacency do indeed reduce, inductively, to (subtile- or left-lattice) adjacencies among smaller tile odometers.  

The main innovation in this section is then to make use of boundary strings to establish the compatibility of left-lattice adjacent tile odometers (Lemmas \ref{l.odomboundarystring} and \ref{l.otile}).  Once left-lattice adjacency is handled, the reduction from the adapatation of the tile argument described above implies that right-lattice and subtile-lattice adjacent tile odometers are also compatible (\lref{otileR}).  Finally, this allows us to prove that tile odometers exist for each circle (\lref{hasodom}).

\subsection{Adapting the tile argument}
Here we adapt the tile argument to reduce right- and subtile-lattice adjacency to subtile- and left-lattice adjacency among smaller tile odoemters.


\begin{lemma}
\label{l.induct-adjacency}
If the tile odometers $h$ and $h_0$ are right-lattice adjacent, then any subodometer of $h$ intersecting $h_0$ is subtile-lattice adjacent to $h_0$.  Similarly, if the tile odometer $h$ is subtile-lattice adjacent to $h_0$, then for any subodometer $h_i^\pm$ of $h_0$ which intersects $h$, either $h,h_i^\pm$ are left-lattice adjacent, or $h$ is subtile-lattice adjacent to $h_i^\pm$, or $h_i^\pm$ is subtile-lattice adjacent to $h$.

\end{lemma}
\begin{proof}
Recall the tiles $R_i^t$ defined in \eref{leftdecomp} from the proof of \lref{tiling}, as shown in \fref{righttile}; each $R_i^t$ is defined by 
\begin{align*}
R_i^0&=T_i^++v_{0k}\\
R_i^1&=T_i^+-v_{0j}\\
R_i^2&=T_i^--v_{0k}\\
R_i^3&=T_i^--v_{0j},
\end{align*}
where $(i,j,k)$ is a rotation of $(1,2,3)$.
We define for each $R_i^t$ a translation $h_i^t$ of the tile odometer for the circle $C_i$ whose domain is $R_i^t$, by
\begin{align*}
\cent(h_i^0)& =\cent(h_i^+)+v_{0k} & \slope(h_i^0)&=\slope(h_i^+)+a_{0k}\\
\cent(h_i^1)& =\cent(h_i^+)-v_{0j} & \slope(h_i^1)&=\slope(h_i^+)-a_{0j}\\
\cent(h_i^2)& =\cent(h_i^-)-v_{0k} & \slope(h_i^2)&=\slope(h_i^-)-a_{0k}\\
\cent(h_i^3)& =\cent(h_i^-)+v_{0j} & \slope(h_i^3)&=\slope(h_i^-)+a_{0j}
\end{align*}
We know from \eref{ttouch1} and the application in \sref{tiles} of \lref{topological} that the boundary of $T_0$ is covered by the $R_i^t$'s.  Thus, to prove the first part of the lemma, it is sufficient to show compatibility of $h_0$ with the $h_i^t$.  And by part \eref{tgooddecomp} of \lref{tiling}, it is sufficient to show compatibility of tile odometers $h_i^\pm$ and $h_s^t$ whose domains are drawn as touching in \fref{righttile}.  It can be checked by hand using the lattice rules \eref{lattice} that such a \change{pair} $h_i^\pm$ and $h_s^t$ are either left-lattice adjacent or subtile-lattice adjacent, proving the full statement of the lemma, since the $h_s^t$ include all odometers which are subtile-lattice adjacent to $h_0$.
\end{proof}

Looking ahead, if we knew left-lattice adjacent tile odometers to be compatible, then by induction, \lref{induct-adjacency} would give compatibility of right-lattice adjacent and subtile-lattice adjacent odometers as well.  Indeed, we will give this as \lref{otileR}, below.

\subsection{Using boundary strings for left-lattice adjacency}

Unlike the argument for tiles, we can not apply $90^\circ$ symmetry and $v_{0i} = \I v_{i0}$ to add the case of left-lattice adjacent tile odometers to \lref{induct-adjacency}, since odometers are only $180^\circ$ symmetric in general.  Instead, we will express the shared boundary of the touching tiles in terms of boundary strings, and \change{inductively} use the compatibility of the restrictions of the odometers to the tiles making up the boundary strings.

To do this, we first need to strengthen our notion of boundary string: We say a partial odometer \emph{respects} a string when its domain includes all tiles of the string, and its restrictions to those tiles are tile odometers which are consecutively left-lattice adjacent.

\begin{lemma}
\label{l.odomboundarystring}
Suppose $(C_0,C_1,C_2,C_3)$ is a proper Descartes quadruple, $h_0$ is a tile odometer for the circle $C_0$ with domain $T_0$,  and write $v_{ij} = v(C_i,C_j)$.  For each $i = 1, 2, 3$ for which $c_i>0$, we have that if $R_i^\pm$ are the endpoints of the $C_i$ boundary string $\rrr$ for $T_0$ 
and $h_{R_i^\pm}$ is a tile odometer for $C_i$ with domain $R_i^\pm$, which satisfies
\begin{equation*}
\slope(h_{R_i^\pm})-\slope(h_0)= \tfrac {1}{2} (v_{i0} \pm v_{0i}),
\end{equation*}
then $h_{R_i^+}$ and $h_{R_i^-}$ are each subtile-lattice adjacent to $h_0$, and $h_0\cup h_{R_i^+}\cup h_{R_i^-}$ respects the $C_i$ boundary string of $T_0$.
\end{lemma}

\begin{figure}[t]
\begin{center}
\nofig{\begin{tikzpicture}[font=\tiny,scale=1/5]
\begin{scope}[draw=gray]
\draw (29.5,26.5) -- (30.5,26.5) -- (31.5,26.5) -- (32.5,26.5) -- (32.5,27.5) -- (32.5,28.5) -- (33.5,28.5) -- (34.5,28.5) -- (35.5,28.5) -- (35.5,29.5) -- (35.5,30.5) -- (36.5,30.5) -- (36.5,31.5) -- (36.5,32.5) -- (36.5,33.5) -- (35.5,33.5) -- (34.5,33.5) -- (34.5,34.5) -- (34.5,35.5) -- (34.5,36.5) -- (33.5,36.5) -- (32.5,36.5) -- (32.5,37.5) -- (31.5,37.5) -- (30.5,37.5) -- (29.5,37.5) -- (29.5,36.5) -- (29.5,35.5) -- (28.5,35.5) -- (27.5,35.5) -- (26.5,35.5) -- (26.5,34.5) -- (26.5,33.5) -- (25.5,33.5) -- (25.5,32.5) -- (25.5,31.5) -- (25.5,30.5) -- (26.5,30.5) -- (27.5,30.5) -- (27.5,29.5) -- (27.5,28.5) -- (27.5,27.5) -- (28.5,27.5) -- (29.5,27.5) -- (29.5,26.5);
\draw (31.000000,32.000000) node {$T_1^+$};
\draw (21.5,28.5) -- (22.5,28.5) -- (23.5,28.5) -- (24.5,28.5) -- (24.5,29.5) -- (24.5,30.5) -- (25.5,30.5) -- (25.5,31.5) -- (25.5,32.5) -- (25.5,33.5) -- (24.5,33.5) -- (23.5,33.5) -- (23.5,34.5) -- (22.5,34.5) -- (21.5,34.5) -- (20.5,34.5) -- (20.5,33.5) -- (20.5,32.5) -- (19.5,32.5) -- (19.5,31.5) -- (19.5,30.5) -- (19.5,29.5) -- (20.5,29.5) -- (21.5,29.5) -- (21.5,28.5);
\draw (22.500000,31.500000) node {$T_2^+$};
\draw (23.5,33.5) -- (24.5,33.5) -- (25.5,33.5) -- (26.5,33.5) -- (26.5,34.5) -- (26.5,35.5) -- (26.5,36.5) -- (25.5,36.5) -- (24.5,36.5) -- (23.5,36.5) -- (23.5,35.5) -- (23.5,34.5) -- (23.5,33.5);
\draw (25.000000,35.000000) node {$T_3^-$};
\draw (32.5,36.5) -- (33.5,36.5) -- (34.5,36.5) -- (35.5,36.5) -- (35.5,37.5) -- (35.5,38.5) -- (35.5,39.5) -- (34.5,39.5) -- (33.5,39.5) -- (32.5,39.5) -- (32.5,38.5) -- (32.5,37.5) -- (32.5,36.5);
\draw (34.000000,38.000000) node {$R_3^-$};
\draw (17.5,32.5) -- (18.5,32.5) -- (19.5,32.5) -- (20.5,32.5) -- (20.5,33.5) -- (20.5,34.5) -- (20.5,35.5) -- (19.5,35.5) -- (18.5,35.5) -- (17.5,35.5) -- (17.5,34.5) -- (17.5,33.5) -- (17.5,32.5);
\draw (19.000000,34.000000) node {$R_3^+$};
\end{scope}\begin{scope}[thick]
\draw (24.5,19.5) -- (25.5,19.5) -- (26.5,19.5) -- (27.5,19.5) -- (27.5,20.5) -- (27.5,21.5) -- (28.5,21.5) -- (29.5,21.5) -- (30.5,21.5) -- (30.5,20.5) -- (31.5,20.5) -- (32.5,20.5) -- (33.5,20.5) -- (33.5,21.5) -- (33.5,22.5) -- (34.5,22.5) -- (35.5,22.5) -- (36.5,22.5) -- (36.5,23.5) -- (36.5,24.5) -- (37.5,24.5) -- (37.5,25.5) -- (37.5,26.5) -- (37.5,27.5) -- (36.5,27.5) -- (35.5,27.5) -- (35.5,28.5) -- (35.5,29.5) -- (35.5,30.5) -- (36.5,30.5) -- (36.5,31.5) -- (36.5,32.5) -- (36.5,33.5) -- (35.5,33.5) -- (34.5,33.5) -- (34.5,34.5) -- (34.5,35.5) -- (34.5,36.5) -- (33.5,36.5) -- (32.5,36.5) -- (32.5,37.5) -- (31.5,37.5) -- (30.5,37.5) -- (29.5,37.5) -- (29.5,36.5) -- (29.5,35.5) -- (28.5,35.5) -- (27.5,35.5) -- (26.5,35.5) -- (26.5,36.5) -- (25.5,36.5) -- (24.5,36.5) -- (23.5,36.5) -- (23.5,35.5) -- (23.5,34.5) -- (22.5,34.5) -- (21.5,34.5) -- (20.5,34.5) -- (20.5,33.5) -- (20.5,32.5) -- (19.5,32.5) -- (19.5,31.5) -- (19.5,30.5) -- (19.5,29.5) -- (20.5,29.5) -- (21.5,29.5) -- (21.5,28.5) -- (21.5,27.5) -- (21.5,26.5) -- (20.5,26.5) -- (20.5,25.5) -- (20.5,24.5) -- (20.5,23.5) -- (21.5,23.5) -- (22.5,23.5) -- (22.5,22.5) -- (22.5,21.5) -- (22.5,20.5) -- (23.5,20.5) -- (24.5,20.5) -- (24.5,19.5);
\end{scope}
\end{tikzpicture}}
\end{center}
\caption{Verifying a boundary string.}
\label{f.checkstring}
\end{figure}
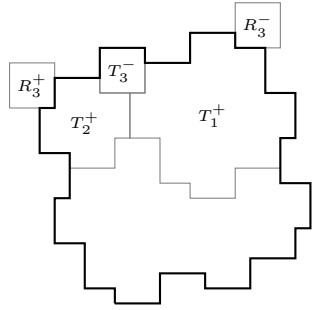

\begin{proof}
Recall from the definition of boundary strings that
\begin{equation*}
\cent(R_i^\pm)-\cent(T_0) = \tfrac{1}{2} (v_{i0} \pm v_{0i}).
\end{equation*}
Thus each $h_{R_i^\pm}$ is subtile-lattice adjacent to $h_0$ as claimed. 
Moreover, \lref{boundarystring} gives
\begin{equation}
\label{e.string-tiles}
\cent(R_i^+)-\cent(T_i^-)=v_{ki}\quad \mbox{and}\quad \cent(T_i^-)-\cent(R_i^-)=-v_{ji}.
\end{equation}
In particular, for $i=1$, then our assumption that $c_1>c_2,c_3$ implies that $R_1^-,T_i^-,R_1^+$ are a triple of consecutively touching tiles in a tiling of $T_1^-$ under the lattice generated by $\{v_{41},v_{21},v_{31}\}$; this is already sufficient to imply that they form a string (either from $R_1^-$ to $R_1^+$ or vice versa), and the sign in \eref{lattice.determinant} implies that they are the string from $R_1^-$ to $R_1^+$.  Moreover, since 
\[
\slope(h_1^-)-\slope(h_0)=-\frac 1 2 (a_{32}+\I a_{32}),
\]
we can calculate using the lattice rules \eref{lattice} that
\begin{equation}
\label{e.string-slopes}
\slope(h_{R_1^+})-\slope(h_i^-)=a_{31},\quad \slope(h_1^-)-\slope(h_{R_1^-})=-a_{21},
\end{equation}
and \eref{string-tiles} and \eref{string-slopes} together give that $h_{R_i^-},h_{T_1^-},h_{R_i^+}$ are consecutively left-lattice adjacent; thus $h_0\cup h_{R_i^+}\cup h_{R_i^-}$ respects the string.

  The cases $i = 2, 3$ require the induction hypothesis and, since they are similar, we handle only the case $i = 3$.  
Decomposing $h_0$ according to tile odometers $h_i^\pm$ on tiles $T_i^\pm$ according to \dref{O0} and using the lattice rules \eref{lattice}, we check that
\begin{align*}
\cent(R_3^+) - \cent(T_2^+) = \tfrac{1}{2}(v_{32} - v_{23}) \quad \mbox{and} \quad \cent(T_3^-) - \cent(T_2^+) = \tfrac{1}{2}(v_{32} + v_{23}),\\
\slope(h_{R_3^+}) - \slope(h_2^+) = \tfrac{1}{2}(a_{32} - a_{23}) \quad \mbox{and} \quad \slope(h_3^-) - \slope(h_2^+) = \tfrac{1}{2}(a_{32} + a_{23}),
\end{align*}
(see \fref{checkstring}).
In particular, the string $\sss'$ from $T_3^-$ to $R_3^+$ is the $C_3$ boundary string of $T_2^+$, and by induction, $h_0\cup h_{R_i^+}\cup h_{R_i^-}$ respects this string.  In exactly the same way, we can check that the string $\sss$ from $R_3^-$ to $T_3^-$ is the $C_3$ boundary string of $T_1^+$, and is respected by $h_0\cup h_{R_i^+}\cup h_{R_i^-}$.  By \lref{boundarystring}, the $C_3$-string for $T_0$ is the concatenation of $\sss$ and $\sss'$, and this string is respected by $h_0\cup h_{R_i^+}\cup h_{R_i^-}$ since all pairs of consecutive tiles in the concatenation are already consecutive tiles in either $\sss$ or $\sss'$.
\end{proof}

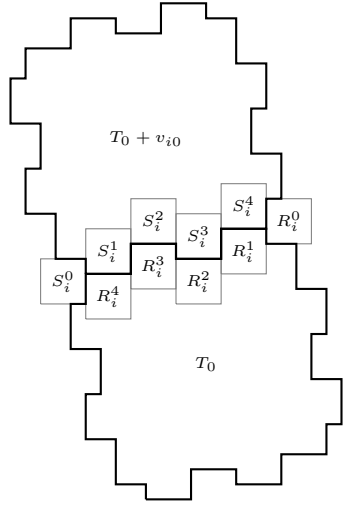
\begin{figure}[t]
\begin{center}
\begin{tikzpicture}[font=\tiny,scale=1/5]
\begin{scope}[draw=gray]
\draw (32.5,36.5) -- (33.5,36.5) -- (34.5,36.5) -- (35.5,36.5) -- (35.5,37.5) -- (35.5,38.5) -- (35.5,39.5) -- (34.5,39.5) -- (33.5,39.5) -- (32.5,39.5) -- (32.5,38.5) -- (32.5,37.5) -- (32.5,36.5);
\draw (34.000000,38.000000) node {$R_i^{0}$};
\draw (29.5,34.5) -- (30.5,34.5) -- (31.5,34.5) -- (32.5,34.5) -- (32.5,35.5) -- (32.5,36.5) -- (32.5,37.5) -- (31.5,37.5) -- (30.5,37.5) -- (29.5,37.5) -- (29.5,36.5) -- (29.5,35.5) -- (29.5,34.5);
\draw (31.000000,36.000000) node {$R_i^{1}$};
\draw (29.5,37.5) -- (30.5,37.5) -- (31.5,37.5) -- (32.5,37.5) -- (32.5,38.5) -- (32.5,39.5) -- (32.5,40.5) -- (31.5,40.5) -- (30.5,40.5) -- (29.5,40.5) -- (29.5,39.5) -- (29.5,38.5) -- (29.5,37.5);
\draw (31.000000,39.000000) node {$S_i^{4}$};
\draw (26.5,32.5) -- (27.5,32.5) -- (28.5,32.5) -- (29.5,32.5) -- (29.5,33.5) -- (29.5,34.5) -- (29.5,35.5) -- (28.5,35.5) -- (27.5,35.5) -- (26.5,35.5) -- (26.5,34.5) -- (26.5,33.5) -- (26.5,32.5);
\draw (28.000000,34.000000) node {$R_i^{2}$};
\draw (26.5,35.5) -- (27.5,35.5) -- (28.5,35.5) -- (29.5,35.5) -- (29.5,36.5) -- (29.5,37.5) -- (29.5,38.5) -- (28.5,38.5) -- (27.5,38.5) -- (26.5,38.5) -- (26.5,37.5) -- (26.5,36.5) -- (26.5,35.5);
\draw (28.000000,37.000000) node {$S_i^{3}$};
\draw (23.5,33.5) -- (24.5,33.5) -- (25.5,33.5) -- (26.5,33.5) -- (26.5,34.5) -- (26.5,35.5) -- (26.5,36.5) -- (25.5,36.5) -- (24.5,36.5) -- (23.5,36.5) -- (23.5,35.5) -- (23.5,34.5) -- (23.5,33.5);
\draw (25.000000,35.000000) node {$R_i^{3}$};
\draw (23.5,36.5) -- (24.5,36.5) -- (25.5,36.5) -- (26.5,36.5) -- (26.5,37.5) -- (26.5,38.5) -- (26.5,39.5) -- (25.5,39.5) -- (24.5,39.5) -- (23.5,39.5) -- (23.5,38.5) -- (23.5,37.5) -- (23.5,36.5);
\draw (25.000000,38.000000) node {$S_i^{2}$};
\draw (20.5,31.5) -- (21.5,31.5) -- (22.5,31.5) -- (23.5,31.5) -- (23.5,32.5) -- (23.5,33.5) -- (23.5,34.5) -- (22.5,34.5) -- (21.5,34.5) -- (20.5,34.5) -- (20.5,33.5) -- (20.5,32.5) -- (20.5,31.5);
\draw (22.000000,33.000000) node {$R_i^{4}$};
\draw (20.5,34.5) -- (21.5,34.5) -- (22.5,34.5) -- (23.5,34.5) -- (23.5,35.5) -- (23.5,36.5) -- (23.5,37.5) -- (22.5,37.5) -- (21.5,37.5) -- (20.5,37.5) -- (20.5,36.5) -- (20.5,35.5) -- (20.5,34.5);
\draw (22.000000,36.000000) node {$S_i^{1}$};
\draw (17.5,32.5) -- (18.5,32.5) -- (19.5,32.5) -- (20.5,32.5) -- (20.5,33.5) -- (20.5,34.5) -- (20.5,35.5) -- (19.5,35.5) -- (18.5,35.5) -- (17.5,35.5) -- (17.5,34.5) -- (17.5,33.5) -- (17.5,32.5);
\draw (19.000000,34.000000) node {$S_i^{0}$};
\end{scope}\begin{scope}[thick]
\draw (24.5,19.5) -- (25.5,19.5) -- (26.5,19.5) -- (27.5,19.5) -- (27.5,20.5) -- (27.5,21.5) -- (28.5,21.5) -- (29.5,21.5) -- (30.5,21.5) -- (30.5,20.5) -- (31.5,20.5) -- (32.5,20.5) -- (33.5,20.5) -- (33.5,21.5) -- (33.5,22.5) -- (34.5,22.5) -- (35.5,22.5) -- (36.5,22.5) -- (36.5,23.5) -- (36.5,24.5) -- (37.5,24.5) -- (37.5,25.5) -- (37.5,26.5) -- (37.5,27.5) -- (36.5,27.5) -- (35.5,27.5) -- (35.5,28.5) -- (35.5,29.5) -- (35.5,30.5) -- (36.5,30.5) -- (36.5,31.5) -- (36.5,32.5) -- (36.5,33.5) -- (35.5,33.5) -- (34.5,33.5) -- (34.5,34.5) -- (34.5,35.5) -- (34.5,36.5) -- (33.5,36.5) -- (32.5,36.5) -- (32.5,37.5) -- (31.5,37.5) -- (30.5,37.5) -- (29.5,37.5) -- (29.5,36.5) -- (29.5,35.5) -- (28.5,35.5) -- (27.5,35.5) -- (26.5,35.5) -- (26.5,36.5) -- (25.5,36.5) -- (24.5,36.5) -- (23.5,36.5) -- (23.5,35.5) -- (23.5,34.5) -- (22.5,34.5) -- (21.5,34.5) -- (20.5,34.5) -- (20.5,33.5) -- (20.5,32.5) -- (19.5,32.5) -- (19.5,31.5) -- (19.5,30.5) -- (19.5,29.5) -- (20.5,29.5) -- (21.5,29.5) -- (21.5,28.5) -- (21.5,27.5) -- (21.5,26.5) -- (20.5,26.5) -- (20.5,25.5) -- (20.5,24.5) -- (20.5,23.5) -- (21.5,23.5) -- (22.5,23.5) -- (22.5,22.5) -- (22.5,21.5) -- (22.5,20.5) -- (23.5,20.5) -- (24.5,20.5) -- (24.5,19.5);
\draw (28.500000,28.500000) node {$T_0$};
\draw (20.5,34.5) -- (21.5,34.5) -- (22.5,34.5) -- (23.5,34.5) -- (23.5,35.5) -- (23.5,36.5) -- (24.5,36.5) -- (25.5,36.5) -- (26.5,36.5) -- (26.5,35.5) -- (27.5,35.5) -- (28.5,35.5) -- (29.5,35.5) -- (29.5,36.5) -- (29.5,37.5) -- (30.5,37.5) -- (31.5,37.5) -- (32.5,37.5) -- (32.5,38.5) -- (32.5,39.5) -- (33.5,39.5) -- (33.5,40.5) -- (33.5,41.5) -- (33.5,42.5) -- (32.5,42.5) -- (31.5,42.5) -- (31.5,43.5) -- (31.5,44.5) -- (31.5,45.5) -- (32.5,45.5) -- (32.5,46.5) -- (32.5,47.5) -- (32.5,48.5) -- (31.5,48.5) -- (30.5,48.5) -- (30.5,49.5) -- (30.5,50.5) -- (30.5,51.5) -- (29.5,51.5) -- (28.5,51.5) -- (28.5,52.5) -- (27.5,52.5) -- (26.5,52.5) -- (25.5,52.5) -- (25.5,51.5) -- (25.5,50.5) -- (24.5,50.5) -- (23.5,50.5) -- (22.5,50.5) -- (22.5,51.5) -- (21.5,51.5) -- (20.5,51.5) -- (19.5,51.5) -- (19.5,50.5) -- (19.5,49.5) -- (18.5,49.5) -- (17.5,49.5) -- (16.5,49.5) -- (16.5,48.5) -- (16.5,47.5) -- (15.5,47.5) -- (15.5,46.5) -- (15.5,45.5) -- (15.5,44.5) -- (16.5,44.5) -- (17.5,44.5) -- (17.5,43.5) -- (17.5,42.5) -- (17.5,41.5) -- (16.5,41.5) -- (16.5,40.5) -- (16.5,39.5) -- (16.5,38.5) -- (17.5,38.5) -- (18.5,38.5) -- (18.5,37.5) -- (18.5,36.5) -- (18.5,35.5) -- (19.5,35.5) -- (20.5,35.5) -- (20.5,34.5);
\draw (24.500000,43.500000) node {$T_0 + v_{i0}$};
\end{scope}
\end{tikzpicture}
\end{center}
\caption{Verifying tiling $T_{C_0}$ by $\Lambda_{C_0}$ using boundary strings.}
\label{f.lefttile}
\end{figure}

We now prove the compatibility of left-lattice adjacent odometers:

\begin{lemma}
\label{l.otile}
If the odometer $h_0'$ for $C_0$ is left-lattice adjacent to $h_0$, then $h_0'$ and $h_0$ are compatible.  Moreover, if $c_i>1$, where $\cent(T(h_0'))-\cent(T(h_0))=\pm v_{i0}$, then any vertex in $T(h_0)\cap T(h_0')$ lies, together with all of its lattice neighbors, in the union of the domains of some pair of proper ancestor odometers of $h_0$ and $h_0'$ which are pairwise left-lattice adjacent.
\end{lemma}

\begin{proof}
If $c_i=0$ then $C_0$ is a Ford circle and this lemma is easy to verify from the construction in \sref{ford}, so we may assume $c_i>0$.

From the definition of left-lattice adjacency, we have without loss of generality that $h_0'\change{(x)}=h_0^i(x) = h_0(x - v_{i0}) + a_{i0} \cdot x$ of $h_0$ with domain $T_0^i=T_0+v_{i0}$, and we let $R_i$ and $S_i$ be tiles for $C_i$ satisfying
\begin{align*}
\cent(R_i)-\cent(T_0) =\tfrac{1}{2} (v_{i0} - v_{0i}),\\
\cent(S_i)-\cent(T_0)= \tfrac{1}{2} (v_{i0} + v_{0i}).
\end{align*}

From \lref{tiling}, we know that $R_i^-,T_0,T_0^i$ and $R_i^+,T_0,T_0^i$ are both touching triples; in particular, the intersection of $T_0$ and $T_0^i$ is a path from $S_i$ to $R_i$.  Therefore, let $\rrr=R_i^0,R_i^1,\dots,R_i^t$ and $\sss=S_i^0,S_i^1,\dots,S_i^t$ be the string from $R_i^-=R_i^0$ to $R_i^+=R_i^t$ and the string from $R_i^+=S_i^0$ to $R_i^-=S_i^t$, respectively; \change{note that they are equivalent under a central reflection} (\fref{lefttile}). \lref{odomboundarystring} guarantees that the interiors of $\rrr$ and $\sss$ lie in $T_0$ and $T_0+v_{i0}$, respectively. \change{By \lref{stringpair}, the intersection of the interiors of $\rrr$ and $\sss$ contains a simple path from from $S_i$ to $R_i$, so $T_0\cap T_0^i$ equals the intersection of the interiors of $\rrr$ and $\sss$.}  So it suffices to show that each restriction for each pair of touching $R_i^{\ell_1}$, $S_i^{\ell_2}$ from $\rrr$ and $\sss$ that $h_0|R_i^{\change{\ell_1}}$ and $h_0|S_i^{\change{\ell_2}}$ are left-lattice adjacent (thus compatible by induction).  Note that $c_i>1$ implies that there is no vertex in $\Z^2$ which lies in the intersection of 4 tiles in a $T_i+\Lambda_{C_i}$ tiling of $\Z^2$, justifying the Moreover clause.

We let $h_{R_i^\pm}$ be defined as in \lref{odomboundarystring}, and let $f_0^i=h_0\cup h_{R_i^-}\cup h_{R_i^+}$ and $f_1^i=h_0^i\cup h_{R_i^-}\cup h_{R_i^+}$, which, by \lref{odomboundarystring} \change{and inductive application of \lref{otileR}, below, to the sublattice adjaceny of $h_{R_i^{\pm}}$ to $h_0$,} are well-defined partial odometers which respect the strings $\rrr$ and $\sss$, respectively.  Considering now touching tiles $R_i^{\ell_1}$ and $S_i^{\ell_2}$ from $\rrr$ and $\sss$, respectively, we see that
\[
R_i^{\ell_1},R_i^{\ell_1-1},\dots, R_i^0=S_i^t,S_i^{t-1},\dots,S_i^{\ell_2}
\]
is a sequence of \change{$(\ell_1+\ell_2+1-t)$} tiles where each consecutive pair $U^m,U^{m+1}$ in the sequence forms a two-tile string which is respected either by $f_0^i$ or $\change{f_1^i}$.  In particular, \change{let} $C_i^1,C_i^2,C_i^3$ be the parents of $C_i$ in clockwise order.  \change{For each} pair $U^m,U^{m+1}$ \change{which is} respected by $f_{\change{q_m}}^i$ ($\change{q_m} \in \{0,1\}$), \change{we have for some $n$ (
  which may vary with $m$) that}
\begin{align}
\slope(f_{\change{q_m}}^i|U^{m+1})-\slope(f_{\change{q_m}}^i|U^m)&=\pm a(C_i,C_i^n),\quad\mbox{where}\\
\cent(U^{m+1})-\cent(U^m)&=\pm v(C_i,C_i^n).
\end{align}
We now have
\[
\pm v(C_i,C_i^{n})=\cent(S_i^{\ell_2})-\cent(R_i^{\ell_1})=\sum_{\change{m}} \left(\cent(U^{m+1})-\cent(U^m)\right)
\]
implies that
\[
\slope(f_1^i|S_i^{\ell_2})-\slope(f_0^i|R_i^{\ell_1})=\sum_{\change{m}} \left(\slope(f_{q_m}^i|U^{m+1})-\slope(f_{q_m}^i|U^{m})\right)=\pm a(C_i,C_i^{n}).
\]
In particular, $f_1^i|S_i^{\ell_2}$ and $f_0^i|R_i^{\ell_1}$ are left-lattice adjacent.
\end{proof}

As noted earlier, \lref{induct-adjacency} now gives us the following: 

\begin{lemma}
\label{l.otileR}
If $h$ is a tile odometer which is right-lattice or subtile-lattice adjacent to the tile odometer $h_0$, then $h$ and $h_0$ are compatible.\qed
\end{lemma}

\subsection{Existence of tile odometers}  We conclude with the following.

\begin{lemma}
\label{l.hasodom}
Every circle $C_0\in \B$ has a tile odometer $h_0:T_0\to \Z$.
\end{lemma}

\begin{proof}
In light of \sref{degenerate}, we may assume that $(C_0,C_1,C_2,C_3)$ is a proper Descartes quadruple with $c_1\geq c_2>1$.
Copying the proof of 
\lref{prototile}, we easily obtain tile odometers $h_i^\pm$ for each circle $C_i$ such that
\begin{equation*}
\cent(T(h_i^\pm)) - \cent(T_0) = \pm \tfrac{1}{2} (v_{kj} - \I v_{kj}),
\end{equation*}
and
\begin{equation*}
c_i = 0 \quad \mbox{or} \quad \slope(h_i^\pm) - \slope(h_0) = \pm \tfrac{1}{2} (a_{kj} + \I a_{kj}),
\end{equation*}
for all rotations $(i,j,k)$ of $(1,2,3)$.  The difficulty lies in checking there are height offsets so that the $h_i^\pm$ have a common extension to $T_0$.

We know from part \eref{tgooddecomp} of \lref{tiling} that the only pairs of subodometers $h_i^\pm$ whose domains intersect are the pairs
\begin{align*}
(h_i^\pm&,h_j^\mp)\phantom{,}\quad\mbox{for }i\neq j,\\
(h_1^\pm&,h_2^\pm),\quad\mbox{and}\\
(h_1^+&,h_1^-).
\end{align*}
Thus, as in the proof of \lref{gluing}, we only need to check compatibility of these pairs.  The tile odometers $h_i^\pm, h_j^\mp$ for $i \neq j$ are easily verified to be subtile-lattice adjacent, and so are compatible by \lref{otileR}.  It remains to check compatibility for tile odometer pairs corresponding to the tile pairs $(T_1^-,T_2^-)$, $(T_1^+,T_2^+)$, and $(T_1^+,T_1^-)$.  To accomplish this, we decompose $T_1^+$ and $T_1^-$ into unions of $Q_i^\pm$ and $S_i^\pm$ for $i = 2, 3, 4$ according to \eref{tdecomp}, as shown in \fref{doubledecomp}.   The second part of \eref{tgooddecomp} implies that it suffices to check compatibility of the restrictions of the odometers $h_i^\pm$ to the pairs 
$(S_2^-,T_2^-)$, 
$(Q_2^+,T_2^+)$,
$(Q_2^+,S_3^-)$, and
$(Q_3^+,S_2^-)$.
Using the lattice rules we verify that these pairs are right-lattice, right-lattice, subtile-lattice, and subtile-lattice adjacent, respectively, so we are done by induction.
\end{proof}

\section{Global odometers}
\label{s.maximality}

Having constructed tile odometers $h_C$ for each circle $C \in {\B}$ by \lref{hasodom}, we now check that the $h_C$ extend to global odometers $g_C$ satisfying $\Delta g_C\leq 1$, and prove our main theorem.   
We first observe that the partial odometers glue together to form global odometers with the correct periodicity and growth at infinity.  (As in the previous section, a tile now is just a subset of $\Ga$.)

\begin{lemma}
\label{l.periodic}
For every circle $C \in {\B}$, there is a function $g_C : \Z^2 \to \Z$ which has a restriction to a tile odometer for $C$, for which the periodicity condition \eref{periodic} holds for $v\in \Lambda_C$, and for which
\begin{equation}
\label{e.periodic.difference}
x \mapsto g_C(x) - \tfrac{1}{2} x^t A_C x - b \cdot x,
\end{equation}
is $\Lambda_C$-periodic for some $b \in \R^2$.
\end{lemma}

\begin{proof}
We may suppose $C = C_0$ and $(C_0,C_1,C_2,C_3)$ is a proper Descartes quadruple, $h_i$ is a tile odometer for $C_i$ whose domain is the tile $T_i$, and write $v_{ij} = v(C_i,C_j)$, $a_{ij} = a(C_i,C_j)$.   \change{We write $h=h_0$ and let $h(k,\ell)$ be the tile odometer on the domain $T(h)+kv_{10}+\ell v_{20}$ given by
\begin{equation*}
x \mapsto h(x - k v_{10} - \ell v_{20}) + (k a_{10} + \ell a_{20}) \cdot x,
\end{equation*}
Now any pair of overlapping tile odometers in 
\[
\hhh=\{h(k,\ell) : k,\ell\in \Z\}
\]
}
are left-lattice adjacent, and thus compatible.  By \lref{gluing}, there is a function $g : \Z^2 \to \Z$ with $g(0)=0$ that is compatible with every $h \in \hhh$.  Since $\hhh$ is invariant under
\begin{equation*}
h \mapsto (x \mapsto h(x - v_{i0}) + a_{i0} \cdot x)
\end{equation*}
for each $i \in \{1,2,3\}$, we see from \lref{gluing} that
\begin{equation*}
x \mapsto g(x - v_{i0}) + a_{i0} \cdot x,
\end{equation*}
differs from $g$ by some constant:
\begin{equation*}
g(x + v_{i0}) = \beta_i + a_{i0} \cdot x + g(x).
\end{equation*}
But $g(0)=0$ implies that $g(v_{i0})=\beta_i$, so that 
\begin{equation*}
g(x + v_{i0}) = g(v_{i0}) + a_{i0} \cdot x + g(x),
\end{equation*}
for all $i = 1, 2, 3$ and $x \in \Z^2$.  Together with $a_{i0} = A_{C_0} v_{i0}$ from \eref{ai0vi0}, this implies the periodicity condition \eref{periodic} for $v\in \Lambda_C$ and that \eref{periodic.difference} is $\Lambda_C$-periodic for some $b \in \R^2$.
\end{proof}

To prove this criterion for general odometers $g_C$, we follow the outline of \pref{ford}.  To begin, we need to understand the Laplacian $\Delta g_C$ for all $C \in {\B}$.  Let $N(x)$ and $\bar N(x)$ denote the set $\{x\pm 1, x\pm \I\}$ of lattice neighbors of $x\in \Ga$ and $\{x\}\cup N(x)$, respectively.

\begin{lemma}
\label{l.web}
Let $h^1$ and $h^2$ be compatible tile odometers for (tangent or identical) circles  in $\B$, let $h=h^1\cup h^2$, and let $x$ such that $x\in T(h^1)\cap T(h^2)$ and $\bar N(x)\sbs T(h^1)\cup T(h^2)$.  If $h^1,h^2$ are left-lattice adjacent, then $\Delta h(x)=1$.  If $h^1,h^2$ are right-lattice adjacent or subtile-lattice adjacent, then $\Delta h(x)=0$ if $x\in s_y\not\subseteq T(h^1)\cup T(h^2)$ for some $y$, and $\Delta h(x)=1$ otherwise.
\end{lemma}

\begin{proof}
The proof is by induction on the areas of $T(h^1)$ and $T(h^2)$.  For the cases of right-lattice and subtile-lattice adjacency, the base case occurs when $h^1$ and $h^2$ are both tile odometers for a circle of curvature $1$, in which case the statement can be checked by hand; it is sufficient to check for the circle $(1,1)$.  Thus for the inductive step in this case, we assume (without loss of generality) $h^1$ is a tile odometer for a circle of curvature $>1$: in particular, it can be decomposed into subodometers according to \dref{O0}.

For this case we let $x$ be a point such that $\bar N(x)\sbs T(h^1)\cup T(h^2)$.  We claim that there is a subodometer $h'$ of $h^1$ such that $\bar N(x)\sbs T(h')\cup T(h^2)$.   Indeed, the definition of a tile ensures that if $h^i$ ($i=1,2$) covers $s_x$ and $s_{x-1-\I}$ then it must also cover either $s_{x-1}$ or $s_{x-\I}$.  (In particular, at least three of the four squares containing $x$ as a vertex are covered by $h$.)  Thus, without loss of generality, we have that $s_{x}$ is covered by $h^2$ and $s_{x-1-\I}$ is covered by $h^1$.  In this case we let $h'$ be the subodometer of $h^1$ whose domain covers $s_{x-1-\I}$, and we have that $\bar N(x)\sbs T(h')\cup T(h^2)$.  \lref{induct-adjacency} now implies that $h'$ and $h^2$ are subtile-lattice adjacent or left-lattice adjacent, so we are done by induction (in particular, note that the condition on $s_y$ which determines whether $\Delta h(x)$ is 0 or 1 is unchanged with the inductive step). 

For the case of left-lattice adjacency, \lref{otile} gives the statement by induction. The base case is when $h^1$ and $h^2$ are left-lattice adjacent along a $v_{i0}$ for which the curvature $c_i$ of the corresponding parent circle $C_i$ is 1.  In this case, however, the proof of \lref{otile} gives that $x$ lies, together with all of its neighbors, in the union of four tile odometers for $C_i=(1,1+2z)$ ($z\in \Ga$)  which are cyclically left-lattice adjacent, and this case can be checked by hand.
\end{proof}

\noindent Note that from the case of left-lattice adjacency, \lref{web} has the crucial consequence that $\Delta g_C\equiv 1$ on the ``web'' of its tile boundaries, as seen in \fref{odom}.
Next, we analyze $\Delta g_C$ on the interior of $T_C$.

\begin{lemma}
\label{l.interior}
Suppose $(C_0,C_1,C_2,C_3) \in {\B}^4$ is a proper Descartes quadruple, write $C_i = (c_i, z_i)$ and $T_i$ is a tile for $C_i$, and decompose $T_0$ as a union of $T_i^\pm$ as in \eref{tdecomp}.  Let $x \in T_0 \setminus \partial T_0$, and let $k$ denote the number of boundaries $\partial T_i^\pm$ of subtiles that contain $x$.  Then we have
\begin{enumerate}
\item If $k = 2$ and $x\neq \cent(T_0)$ then $\Delta g_{C_0}(x) =1$.
\item If $k = 2$ and $x=\cent(T_0)$ then $\Delta g_{C_0}(x)=0$.
\item If $k = 3$ then $\Delta g_{C_0}(x) =0$.
\item If $k = 4$ and $x\neq \cent(T_0)$ then $\Delta g_{C_0}(x) = -1$.\label{cyan}
\item If $k = 4$ and $x = \cent(T_0)$ then $\Delta g_{C_0}(x) = -2$.\label{blue}
\end{enumerate}
\end{lemma}

In particular, as $180$ degree symmetry precludes $k=3$ when $x=\cent(T_0)$, we have \[ \Delta g_{C_0}(x) = 3 - k - \id_{\{ \cent(T) \}}(x) \] 
whenever $k(x)\geq 2$.   Note that cases (\ref{cyan}) and (\ref{blue}) arise only when $C_0$ is a Ford or Diamond circle, so these cases have been proved already in \sref{degenerate}.  These rules can be witnessed in ``general'' tile decomposition in \fref{dec}.  (When looking at this figure, keep in mind that the curve surrounding the Soddy twin tile is not the meeting of two borders.)

\begin{remark}
  The above lemma provides a fast algorithm for recursively generating tiles with their associated Laplacian patterns; the base case for the recursion is given by the base cases of \dref{O0}.  The appendix lists all such patterns associated to circles in $\mathcal{B}$ of curvature $1 \leq c \leq 100$.

  One may be tempted to define the odometers using this lemma in place of our complicated recursive construction.  Indeed, by Liouville theorem, the Laplacian of the odometer determines the odometer up to a harmonic polynomial.  Every periodic Laplacian pattern can be integrated to a function in this way.  However, there is no guarantee in general that the function so constructed is integer valued.  Thus in some sense, an important point of the construction in this manuscript is that these particular patterns do integrate to integer valued functions.
\end{remark}

\begin{proof}[Proof of Lemma~\ref{l.interior}]
For the Ford and Diamond circles, this Lemma has already been verified in \sref{ford}.  Thus, we may assume by induction that the $c_i$ are distinct and positive, and that the lemma holds for the proper Descartes quadruples $(C_1,C_4,C_2,C_3)$ and $(C_2,C_3,C_5,C_6)$ for $C_4, C_5, C_6 \in {\B}$.  Since $c_1 > 1$, we can decompose $T_1^+$ and $T_1^-$ into $Q_i^\pm$ and $S_i^\pm$ as in Claim \ref{c.doubledecomp} and \fref{doubledecomp}.

First consider the case where $k=3$.  \dref{O0} (and part \eref{tgooddecomp} from \lref{tiling}) give that the two subtiles containing $x$ are subtile-lattice adjacent.  Since $k=3$, there exists the square $s_y$ in the final hypothesis of \lref{web}, and thus \lref{web} gives that $\Delta g_{C_0}(x)=0$.

Next consider the case where $x \neq \cent(T_0)$ and $k = 2$.  If the two subtiles whose boundaries contain $x$ are not the pair $\{T_1^+,T_1^-\}$, then by \dref{O0} (and part \eref{tgooddecomp} from \lref{tiling}), the two subtiles containing $x$ are subtile-lattice adjacent, and \lref{web} now gives the result. 

If on the other hand $k=2$ and $x\in \partial T_1^+\cap \partial T_1^-$, then we use the double decomposition of $T_1^+,T_1^-$.  As in the proof of \lref{web}, we are guaranteed that two tiles from among $S_3^-,Q_3^+,S_2^-,Q_2^+,S_4^+$ cover the neighborhood $\bar N(x)$.  As in the proof of \lref{hasodom}, all pairs among these tiles are known 
to have induced tile odometers which are subtile-lattice adjacent or right-lattice adjacent except for the pairs $\{Q_2^+,S_2^-\}$ and $\{S_3^-,Q_3^+\}$.  But \eref{tgooddecomp} implies that this case cannot occur unless $c_i=0$ for some $i\in \{3,4\}$, in which case $C_0$ is a Ford circle.

Finally, if $x=\cent(T)$, then $x\in \partial T_1^+\cap \partial T_1^-$ implies that $c_4=0$.  In particular, $C_1 = C_{pq}$ is a Ford circle with Ford circle parents $C_2 = C_{p_1 q_1}$ and $C_3 = C_{p_2 q_2}$.  We have that $\bar N(x)$ is covered by $T_1^+\cup T_1^-$.  In this case, the tile odometers $g_{pq}^\pm$ on $T_1^\pm$ are related by $q_{pq}^+(x - (q_2-q_1,q)) = q_{pq}^-(x) + (p,p_2-p_1) \cdot x + k$ for some constant $k \in \Z$.  Thus, the explicit formula from \sref{ford} can be used to verify this case.
\end{proof}

We now generalize the inductive argument in \pref{ford} to obtain maximality of general odometers.

\begin{lemma}
\label{l.maximal}
For each $C \in {\B}$, $g_C$ is maximal.
\end{lemma}

\begin{proof}
Suppose $X \subseteq \Z^2$ is connected and infinite and $\Delta (g_C + \1_X) \leq 1$.  Let $Y$ be a connected component of $\Z^2 \setminus X$ and observe that $\Delta (g_C - \1_Y) \leq 1$.  In particular, we may assume that $Y = \Z^2 \setminus X$.  It is enough to show that $Y$ must be empty.  Now, if $Y$ is not contained in $T \setminus \partial T$ for some $T \in T_C + \Lambda_C$, then by \lref{web}, there is a point $x \in Y$ such that $\Delta g_C(x) = 1$ and $\Delta \id_X(x) > 0$, since the tile odometers of which $g_C$ consist are left-lattice adjacent.  Thus, we may assume $Y \subseteq T_C \setminus \partial T_C$.  The lemma is now immediate from the following claim.

{\em Claim.}  Suppose $Y \subseteq T_C$ is simply connected, $Y \setminus \partial T_C$ is non-empty, $Y \cap \partial T$ is connected, and $Y \cap \partial T_C \cap T_i^\pm$ is nonempty for at most one subtile $T_i^\pm$.  Then there is a vertex $x \in Y \setminus \partial T_C$ such that $\Delta (g_C - \1_Y)(x) > 1$.

We prove this by induction on the curvature of $C$.  Note that, by the proofs of \pref{ford} and \pref{diamond}, we may assume that $C$ is neither a Ford nor a Diamond circle.  In particular, each $T_i^\pm$ contains at least one square and the pairwise intersections are exactly what we expect from the picture.  We may assume that $C_1$ is the largest parent.  Thus $T_i^\pm$ and $T_j^\mp$ have simply connected intersection when $i \neq j$ and $T_i^\pm$ and $T_i^\mp$ are disjoint except when $i = 1$, in which case the intersection is simply connected.

{\em Case 1.} $Y$ is contained in the interior of some $T_i^\pm$.  The claim follows either by induction hypothesis or \change{by} the corresponding result\change{s} for Ford and Diamond circles in \pref{ford} and \pref{diamond}.

{\em Case 2.} Some $x \in \partial Y \setminus \partial T$ lies in the boundary of exactly $T_i^+$ and $T_j^-$ with $i \neq j$.  Observe that $\Delta \1_Y(x) < 0$ and $x \neq \cent(T)$. Thus \lref{interior} gives the claim.

{\em Case 3.} Some $x \in \partial Y \setminus \partial T$ lies in the boundary of exactly three $T_i^\pm$.  Since case 2 is excluded, we must have $\Delta \1_Y(x) < -1$ and thus \lref{interior} again gives the claim.

{\em Case 4.} In the exclusion of the above three cases, the topology of the tile decomposition implies that $Y$ lies in the union of the interior of $T_1^+$, the interior of $T_1^-$, and the intersection of one $T_1^\pm \cap \partial T$.  In $Y \cap T_1^+$ is non-empty, then we can inductively apply the claim to $Y \cap T_1^+$.  Otherwise, we can inductively apply the claim to $Y \subseteq T_1^-$.
\end{proof}

From \lref{periodic} and \lref{maximal}, we immediately obtain \tref{odometer}, modulo checking that $\Lambda_{C_0}$ and $L_{C_0}$ are in fact the same lattice.
\begin{theorem}
\label{t.lattice2}
$L_{C}=\Lambda_{C}$ for all $C \in \B$.
\end{theorem}

\begin{proof}
\lref{lattice1} verified $\Lambda_C \subseteq L_{C}$, thus it remains to verify $L_{C}\subseteq \Lambda_{C}$.
For $\change{C}\in \B$, let $g_{\change{C}}$ be the odometer for $\change{C}$, as constructed in \sref{maximality}.  
Recall from \eref{periodic} that $g_{\change{C}}$ satisfies
\[
g_{\change{\change{C}}}(x+v)=g_{\change{\change{C}}}(x)+x^t A_{\change{\change{C}}} v+ g_{\change{\change{C}}}(v)
\]
for $v\in \Lambda_{\change{C}}$.
We will now modify $g_{\change{C}}$ to produce an odometer for $\change{C}$ \change{satisfies the periodicity condition \eref{periodic} for the lattice $L_{\change{C}}$}.  In particular, for $v\in L_{\change{C}}$, we let
\[
g_{\change{C}}^v(x)=g_{\change{C}}(x+v)-x^t A_{\change{C}} v-g_{\change{C}}(v),
\]
which is an integer since $A_{\change{C}} v\in \Z^2$ by the definition of $L_{\change{C}}$.  Note that for $v\in \Lambda_{\change{C}}$ we have $g^v(x)=g(x)$.

We now define
\[
g'(x):=\min_{v\in L_{\change{C}}} g^v(x).
\]
Note that the periodicity $g^v(x)=g(x)$ for $v\in \Lambda_{\change{C}}$ implies that this can be interpreted as a finite minimum over the quotient $L_{\change{C}}/\Lambda_{\change{C}}$.  

In particular, up to an additive constant, $g'(x)$ is still an odometer for $A_{\change{C}}$, and now satisfies the periodicity condition \eref{periodic} for the lattice $L_{\change{C}}$.  In particular, we have that the average Laplacian $\bar \Delta g'$ of $g'$ over one period of $L_{\change{C}}$ must satisfy $\frac 1 d =\trace(A_{\change{C}})\leq \bar \Delta g'\leq \bar \Delta g=\frac 1 d$.  But then we must have $\det(L_{\change{C}})\geq \frac 1 d$, and thus $L_{\change{C}}=\Lambda_{\change{C}}$.
\end{proof}

\appendix

\section{Table of Odometer Patterns}

Here we display proper Descartes quadruples $(C_0,C_1,C_2,C_3) \in \B$, along with the Soddy precursor $C_4=2(C_1+C_2+C_3)-C_0$ of $C_0$, the vectors $v(C_i,C_0)$ and $a(C_i,C_0)$ $(i=1,2,3)$, the tile odometer for $C_0$, and a tiling neighborhood in $T_C+L_C$.  We display quadruples up to symmetry for $1 \leq c_0 \leq 156$.  An extended appendix with more circles is available as an ancillary file for this manuscript at \url{arXiv.org}.
\footnotesize
\setlength{\tabcolsep}{2pt}
\begin{longtable}[c]{rrrcc}
\hline
\begin{tabular}{r}$C_0$ \\$C_1$ \\$C_2$ \\$C_3$ \\$C_4$ \end{tabular}\hspace{1em} & \begin{tabular}{r} \\$v(C_1,C_0)$ \\$v(C_2,C_0)$ \\$v(C_3,C_0)$ \\ \\\end{tabular} & \begin{tabular}{r}  \\$a(C_1,C_0)$ \\$a(C_2,C_0)$ \\$a(C_3,C_0)$ \\ \\\end{tabular} & tile odometer & \begin{minipage}{2.8cm}One neighborhood of $T_C$ in $T_C+L_C$.\end{minipage} 

\ \\

\hline

\ \\

\input{app.tex}
\end{longtable}

\end{document}